\documentclass[10pt]{amsart}

\usepackage{amsmath, amsthm, amssymb}
\usepackage{mathtools}
\usepackage{verbatim}
\usepackage{color}
\usepackage[shortlabels]{enumitem}
\usepackage[margin=2.75cm]{geometry}
\usepackage{hyperref}

\allowdisplaybreaks[2]

\DeclareMathOperator{\Imm}{Imm}
\newcommand{\dd}[1]{\"{#1}}
\newcommand{\R}{\mathbb{R}}
\newcommand{\N}{\mathbb{N}}
\newcommand{\SK}{\mathcal{K}}
\newcommand{\SE}{\mathcal{E}}

\renewcommand{\S}{\mathbb{S}}
\newcommand{\IP}[2]{\left<#1,#2\right>}
\newcommand{\vn}[1]{\lVert#1\rVert}

\newtheorem{thm}{Theorem}[section]
\newtheorem*{thm*}{Theorem}
\newtheorem{theorem}[thm]{Theorem}
\newtheorem{prop}[thm]{Proposition}
\newtheorem{lem}[thm]{Lemma}
\newtheorem{lemma}[thm]{Lemma}
\newtheorem{cor}[thm]{Corollary}

\theoremstyle{definition}
\newtheorem{rmk}[thm]{Remark}

\begin{document}

\title{The gradient flow for entropy on closed planar curves}
\author{Lachlann O'Donnell\and Glen Wheeler\and Valentina-Mira Wheeler}

\thanks{}
\address{University of Wollongong\\
Northfields Ave\\
2522 NSW, Australia}
\email{glenw,vwheeler,lachlann@uow.edu.au}

\begin{abstract}
In this paper we consider the steepest descent $L^2$-gradient flow of the
entropy functional.
The flow expands convex curves, with the radius of an initial circle growing
like the square root of time.
Our main result is that, for any initial curve (either immersed locally strictly convex
of class $C^2$ or embedded of class $W^{2,2}$ bounding a strictly convex body), the
flow converges smoothly to a round expanding multiply-covered circle.
\end{abstract}
\maketitle

\section{Introduction}

Suppose $\gamma:\S^1\rightarrow\R^2$ is a locally strictly convex closed plane curve with turning number $\omega$
 and consider the energy
\[
	\SE(\gamma) = \int_0^L k\log k\, ds = \int_0^{2\omega\pi} \log k\,d\theta
\,.
\]
In the first equality $\gamma$ is parametrised by arc-length, and in the second parametrised by angle.

The $L^2(d\theta)$-gradient flow of $\SE$ is the one-parameter family of smooth immersed curves $\gamma:\S^1\times(0,T)\rightarrow\R^2$ with normal velocity equal to
$-\text{grad}_{L^2(d\theta)}(\SE(\gamma))$ (see Section \ref{SCprelim}), that is
\begin{equation}
\partial_t\gamma(\theta,t) = -(k_{\theta\theta}(\theta,t) + k(\theta,t))N(\theta)
\,.
\label{EF}
\tag{EF}
\end{equation}
Note that in the above evolution equation the $\theta$ and $t$ variables are independent.

The entropy flow \eqref{EF} is a highly degenerate system of fourth order
parabolic partial differential equations.
The $\partial_\theta$ derivative involves division by the curvature scalar $k$.
Expressing \eqref{EF} in the arbitrary parametrisation, the leading-order term
$k_{\theta\theta}$ is
\begin{equation}
\label{EQefingenlparam}
	- \frac{|\gamma_u|}{\IP{\gamma_{uu}}{N}}
	  \bigg(
      		\frac{|\gamma_u|}{\IP{\gamma_{uu}}{N}}
		\bigg(
			\frac{\IP{\gamma_{uu}}{N}}{|\gamma_u|^{2}}
		\bigg)_u
	  \bigg)_u
\end{equation}
whose highest-order component is
\[
	- \frac{\IP{\gamma_{uuuu}}{N}}{\IP{\gamma_{uu}}{N}^2}
\,.
\]
Since $\IP{\gamma_{uu}}{N} = k|\gamma_u|^2$, and $|\gamma_u|\ne0$ along any
regular curve, the degeneracy in the highest-order term in the operator is
proportional to $k^{-2}$.
This is much more than the usual ``powers of the gradient $|\gamma_u|$''
degeneracy that one typically finds in curvature flow of curves.
For local existence, we see this in the hypothesis of our first local existence theorem (Theorem \ref{TM1}).
This result says that the flow exists uniquely from locally strictly convex initial data of class $W^{3,\infty}(du)$.
The flow instantly becomes smooth (that is, for $t>0$), and if $T<\infty$ then
the $W^{3,\infty}(du)$ norm explodes as $t\nearrow T$.
The proof of Theorem \ref{TM1} is via the method of maximal regularity and
analytic semigroups.

The entropy flow expands any initial circle, with radius over time given by
$r(t) = \sqrt{r_0^2 + 2t}$ ($r_0$ is the radius of the initial circle).
A natural question is on the stability of this homothetic solution.
This correlates well with other work on fourth-order flows, for instance
stability of circles under curve diffusion \cite{CR20,MO20,W13}, the elastic
flow \cite{DKS02}, and Chen's flow \cite{BWW20,CWW20}.

Here there are a few troubling aspects of the entropy flow from the outset.

First, the energy is \emph{not bounded from below} and so we should be concerned about
$\SE(t)\searrow-\infty$ as $t\nearrow T < \infty$.

Second, the flow only makes sense in the class of locally strictly convex
curves.
It is not expected that any kind of \emph{positivity preserving principle}
holds for higher-order evolution equations \cite{AGGMNSGLN86,GGS10}.
The failure of closed strictly convex plane curves to remain convex under a
large class of fourth-order curvature flows is well-known (see Blatt \cite{B10},
Giga-Ito \cite{GI98pinching,GI99loss}, Elliott-MaierPaape \cite{EM01losing}
for related results).

The paper of Andrews \cite{A99} stands in stark contrast to the literature
above.
There, the gradient flow for the affine length functional is considered.
The main result is that any embedded convex curve flows for infinite time,
becoming larger and closer to an homothetically expanding ellipse.
A basic property of the flow is that convexity is \emph{preserved}, despite the
flow being of fourth-order.
The idea is that the degeneracy in the operator is so powerful that the flow
moves away from singular curves -- those that are not strictly convex.

Our work here adds another example to the one of Andrews, of a higher-order
flow with strong degeneracy that has a similar `singularity avoidance' property.
Despite the energy functional $\SE$ having no lower bound, and the flow being
fourth-order defined only on locally strictly convex curves, the resultant flow
\eqref{EF} exhibits several beautiful characteristics.
These include concavity of evolving length and convexity of the energy over
time.
The main consequence is a gradient estimate for the curvature and preservation
of local strict convexity for all time (Proposition \ref{PN1}).
Another remarkable aspect of \eqref{EF} then comes into play: monotonicity of
the Dirichlet energy for the associated support function.
This powerful diffusive effect allows us to obtain convergence of a rescaling
of the flow in the $C^1$-topology.

None of these results require any kind of smallness condition or closeness to a circle.
Furthermore, they work together to give a powerful smoothing effect, beyond
that already attained by the $W^{3,\infty}(du)$ theorem mentioned above.
In the class of embeddings, we push this to initial curves that have curvature function in $L^2(ds)$, and bound a convex domain.
For immersions we are able to treat locally strictly convex curves with continuous curvature function.

\begin{thm}
\label{TM2}
Suppose $\gamma_0:\S^1\rightarrow\R^2$ is either
\begin{itemize}
\item an immersed locally strictly convex closed curve of class $C^2(ds)$ with turning number $\omega$; or
\item an embedded curve with $k\in L^2(ds)$ that bounds a strictly convex body.
\end{itemize}
The entropy flow $\gamma:\S^1\times(0,T)\rightarrow\R^2$ with $\gamma_0$ as
initial data is a one-parameter family of regular smooth immersed
locally strictly convex closed curves that exists for all time ($T=\infty$).
The rescaling $\eta = \gamma/\sqrt{L_0^2 + 8\omega^2\pi^2 t}$ converges
to a round $\omega$-circle in the sense that the rescaled support function
converges to a constant in $C^\infty(d\theta)$.
\end{thm}

Above we use the terminology from \cite{Schneider}: a \emph{strictly convex body} is a
non-empty, compact, strictly convex subset of $\R^2$.

\begin{rmk}
Given that our flow is higher-order, the initial data being either immersed of class $C^2(ds)$ or embedded of class $H^2(ds)$ is quite weak.
A similar phenomenon (but lower order) has been observed by Andrews \cite[Theorem I2.1]{AndrewsCC}.
\end{rmk}

This paper is organised as follows.
Section \ref{SCprelim} is concerned with setting notation, giving some well-known properties
of strictly convex plane curves, and the derivation of the flow as the steepest
descent $L^2(d\theta)$-gradient flow for $\SE$.
The diligent reader may then continue with Appendix \ref{SClocalwellp}, which
gives an outline of our first local existence proof that requires
$W^{3,\infty}(du)$ data.
This is an important foundation for Section \ref{SCfundest}.
The main goal of Section \ref{SCfundest} is to derive powerful a-priori
estimates that yield preservation of convexity and a smoothing effect.
The latter gives, via a compactness argument, Theorem \ref{TMTM21}.
This proof uses the $W^{3,\infty}(du)$ result from Appendix \ref{SClocalwellp}.
The global existence part of Theorem \ref{TM2} is also established in Section
\ref{SCfundest} (see Corollary \ref{CYglobal}).
Section \ref{SCglobal} studies the global behaviour of the flow, and is concerned with
giving estimates for a scale-invariant quantity that imply it is eventually
monotone.
In Section 5 we also prove a sharp lower bound for $kL$, which is an optimal
convexity estimate.
Finally, in Section \ref{SCrescaling}, we study the rescaled flow $\eta$.
We prove uniform estimates for rescaled length and curvature, then use these to
show exponential decay of $||h^\eta_{\theta^p}||_2^2(t)$ for all $p\in\N$,
establishing the asymptotic roundness part of Theorem \ref{TM2} in a standard way.

\section*{Acknowledgements}

The work of the first author was completed under the financial support of an
Australian Postgraduate Award.
He is grateful for their support.
The work of the third author is partially supported by ARC Discovery Project
DP180100431 and ARC DECRA DE190100379.
She is grateful for their support.
The authors are grateful to the referee of the first version of this article
for their careful reading and comments, which led to an improvement of the
paper.

\section{Preliminaries}
\label{SCprelim}

To any regular closed curve $\gamma:\S^1\rightarrow\R^2$ with $\gamma\in C^1(du)$ we can
associate an angle function $\theta:\S^1\rightarrow[0,2\omega\pi)$ (here $\omega$
is the turning number, and note that $\theta$ is not periodic), defined by
\[
N(\theta(u)) = -(\cos\theta(u), \sin\theta(u))\,,
\]
where $N$ is the inward unit normal vector.

If $\gamma\in C^2$ is strictly locally convex, the mapping $s\mapsto \theta(s)$ (here $s$ is
the Euclidean arc-length) satisfies
\[
	\theta_s = k > 0
	\,.
\]
This makes $\theta$ invertible on $[0,2\omega\pi)$, and so we may use $\theta$ as an
independent variable.
We define the \emph{support function} $h:[0,2\omega\pi)\rightarrow(0,\infty)$ via
\[
	h(\theta) = \IP{\gamma(\theta)}{-N(\theta)}
\,.
\]
Note that our energy (and the resultant flow) is invariant under translation,
so we may without loss of generality assume that $\gamma(\S^1)$ encloses the
origin so that the support function is positive.

If $\gamma\in W^{2,2}(ds)$ is an embedding, the Jordan curve theorem implies
that $\R^n\setminus\gamma(\S^1)$ consists of two connected components: one
unbounded and the other bounded, commonly referred to in this context as the
`interior' of $\gamma$.
Here, we use the notation $\SK(\gamma)$ for this `interior' set together with
its boundary.
In words, we say that $\SK(\gamma)$ is the \emph{strictly convex body} associated to
$\gamma$.
As suggested by the terminology, we assume that $\SK(\gamma)$ is a non-empty
compact strictly convex set, so the support function $h:[0,2\pi)\rightarrow(0,\infty)$
again exists but with the more usual convex geometry definition that
\[
	h(\theta) = h_{\SK(\gamma)}(H_\theta) = \sup_{x\in\SK(\gamma)} \IP{x}{(\cos\theta,\sin\theta)}
\,.
\]
Here we used $H_\theta$ to denote the supporting line to 
$\SK(\gamma)$ that is normal to $(\cos\theta, \sin\theta)$.
The strict convexity of $\SK(\gamma)$ ensures that the support function $h$ is
of class $C^1$ and a.e. of class $C^2$.
Note that the condition that $k\in L^2(ds)$ implies that the tangent vector to $\gamma$
is uniformly H\"older continuous of class $C^{0,1/2}$ with respect to arclength.
We refer the reader to \cite{Schneider} for the essential theory of convex
bodies and related facts on support functions.

We work for the remainder of this section with smooth curves.
We have the fundamental relations (via the chain rule)
\[
k\,\partial_\theta = \partial_s\quad\text{and}\quad d\theta = k\,ds
\,.
\]
Note that then $\partial_\theta\partial_\theta = k^{-2}\partial_s\partial_s - k_sk^{-3}\partial_s$, so
\begin{align*}
h_{\theta\theta} + h 
&= \IP{\gamma_{\theta\theta}}{-N} + 2\IP{\gamma_\theta}{-N_\theta} - \IP{\gamma}{N_{\theta\theta}+N}
\\
&= -\IP{\gamma_{\theta\theta}}{-N} + 2\IP{\gamma_\theta}{-N_\theta}
\\
&= -\IP{k^{-2}\kappa - k_sk^{-3}T}{-N} + 2\IP{\frac1k T}{T}
\\
&= -\frac1k + 2\frac1k = \frac1k
\,.
\end{align*}
In the second equality we used $N_{\theta\theta} = -T_\theta = -N$, which follows from the definition of $\theta$.

Summarising, the curvature of $\gamma$ satisfies
\begin{equation}
\label{EQcurv}
k(\theta) = \frac{1}{h_{\theta\theta}(\theta) + h(\theta)}\,. 
\end{equation}

We are interested in the steepest descent $L^2$-gradient flow for the
\emph{entropy} of $\gamma$ in the $L^2(d\theta)$ metric. To this end, suppose
$\gamma:\S^1\times(0,T_0)\rightarrow\R^2$ is now a one-parameter family of smooth locally
convex curves differentiable in time with associated support functions $h:[0,2\omega\pi)\times(0,T_0)\rightarrow(0,\infty)$ satisfying
\[
h_t = F(h)
\,.
\]
Note that here the $t$ and $\theta$ derivatives are \emph{independent}.
We find:
\begin{lem}
\label{LM1}
Define $\SE(t) := \SE(\gamma(\cdot,t))$.
Then
\[
\SE'(t) = -\int_0^{2\omega\pi} F(k_{\theta\theta}+k)\,d\theta
\,.
\]
\end{lem}
\begin{proof}
We calculate using \eqref{EQcurv} and integration by parts:
\begin{align*}
\SE'(t) &= \int_0^{2\omega\pi} \frac{k_t}{k}\,d\theta
\\
&=  -\int_0^{2\omega\pi} \frac1k\frac{(h_{\theta\theta}+h)_t}{(h_{\theta\theta}+h)^2}\,d\theta
\\
&=  -\int_0^{2\omega\pi} k(h_{\theta\theta}+h)_t\,d\theta
\\
&=  -\int_0^{2\omega\pi} k(F_{\theta\theta}+F)\,d\theta
\\
&=  -\int_0^{2\omega\pi} F(k_{\theta\theta}+k)\,d\theta
\,.
\end{align*}
\end{proof}

Lemma \ref{LM1} allows us to conclude that the \emph{steepest descent $L^2(d\theta)$-gradient flow} for $\SE$ is the evolution equation
\[
h_t = k_{\theta\theta} + k
\]
where $h:\S^1\times(0,T_0)\rightarrow(0,\infty)$ is the support function associated to a family of curves $\gamma:\S^1\times(0,T_0)\rightarrow(0,\infty)$.

\section{Fundamental estimates and rough data}
\label{SCfundest}

In this section, we establish bespoke estimates to consider general initial data.
This is either locally convex and in $C^2(ds)$ or an embedding bounding a convex planar domain with curvature in $L^2(ds)$.

The fact that the entropy flow is the steepest descent $L^2(d\theta)$ flow for $\SE$ immediately implies
\[
	\int_0^t \vn{F}_2^2(\hat t)\,d\hat t = \SE(0) - \SE(t)
	\,.
\]
However note that this is of limited use a-priori since $\SE(t)$ is, for general convex curves, unbounded from below.

First, we must rule out that for $t\nearrow T_0<\infty$, $\SE(t)\searrow -\infty$.
We begin with the following 
 generic estimate.

\begin{lem}
\label{LM4}
Suppose $\gamma:\S^1\rightarrow\R^2$ is a locally convex curve.
The entropy $\SE(\gamma)$ is bounded from below by the logarithm of the reciprocal of length:
\[
	\SE(\gamma) \ge 2\omega\pi\log \frac{2\omega\pi}{L(\gamma)}
	\,.
\]
\end{lem}
\begin{proof}
Since
$\log$ is concave, we have by Jensen's inequality
\[
\log\bigg(\frac1{2\omega\pi}\int \frac1k\,d\theta\bigg) \ge \frac1{2\omega\pi}\int \log\frac1k\,d\theta\,.
\]
Now just recall that the RHS of the above estimate is $-\SE(\gamma)$ and
\[
L(\gamma) = \int \frac1k\,d\theta
\,.
\]
This implies
\[
\log \frac{L(\gamma)}{2\omega\pi} \ge -\frac1{2\omega\pi}\SE(\gamma)\,.
\]
Multiplying by $-2\omega\pi$ yields the claimed estimate.
\end{proof}

The following lemma shows that our initial data hypotheses bound the initial energy.
\begin{lem}
\label{LMinitialenergybdd}
Suppose $\gamma:\S^1\rightarrow\R^2$ is either
\begin{itemize}
\item an immersed locally strictly convex closed curve of class $C^2(ds)$ with turning number $\omega$; or
\item an embedded curve with $k\in L^2(ds)$ that bounds a strictly convex body.
\end{itemize}
Then
\begin{equation}
\label{EQenergyinitialbdd}
	\bigg|\int \log k\,d\theta\bigg| = |\SE(\gamma)| < \infty\,.
\end{equation}
\end{lem}
\begin{proof}
In the first case, we can work pointwise.
Let $\underline{k} = \inf k$, $\overline{k} = \sup k$.
We know that $0 < \underline{k} \le \overline{k} < \infty$ due to the fact that $\gamma$ is of class $C^2(ds)$.
Then $\log \underline{k} \le \log k \le \log \overline{k}$ and \eqref{EQenergyinitialbdd} follows immediately by integration.

In the second case, we 
estimate from below using Lemma \ref{LM4}. 
The estimate from above is as follows:
\begin{align*}
\SE(\gamma) = \int_0^{L} k\log k\,ds
	\le \int_0^{L} k(k-1)\,ds
	\le ||k||_{L^2(ds)}^2\,.
\end{align*}
\end{proof}

Let us now show that the energy is convex in time.

\begin{lem}
\label{LMeconvex}
Suppose $\gamma:\S^1\times(0,T)\rightarrow\R^2$ is a smooth entropy flow.
Then
\begin{equation}
\label{EQfirst}
	\SE' = -\int (k_{\theta\theta} + k)^2\,d\theta = - \int F^2\,d\theta
\end{equation}
(where we recall that $F = k_{\theta\theta} + k$) and
\begin{equation}
\label{EQsecond}
	\SE'' = 2\int \big(k(F_{\theta\theta} + F)\big)^2\,d\theta
\,.
\end{equation}
\end{lem}
\begin{proof}
The equality \eqref{EQfirst} follows immediately from Lemma \ref{LM1} and the definition of the flow.
Differentiating again, we find:
\begin{align}
	\SE'' 
&= -\frac{d}{dt}\int (k_{\theta\theta} + k)^2\,d\theta
\notag
\\
&= -2\int (k_{\theta\theta} + k)(k_{t\theta\theta} + k_t)\,d\theta
\notag
\\
&= -2\int (k_{\theta\theta} + k)_{\theta\theta}k_{t}\,d\theta
 + -2\int (k_{\theta\theta} + k)k_t\,d\theta
\notag
\\
&= -2\int k_t\Big(F_{\theta\theta} + F\Big)\,d\theta
\,.
\label{EQder}
\end{align}
Note that
\begin{equation}
\label{EQcurvatureevo}
k_t = \bigg(\frac1{h_{\theta\theta}+h}\bigg)_t
= -k^2 (F_{\theta\theta} + F)
\end{equation}
which, upon combining with \eqref{EQder}, yields \eqref{EQsecond}.
\end{proof}

The convexity of $\SE$ yields the following a-priori estimate.

\begin{cor}
\label{CYenest}
Suppose $\gamma:\S^1\times(0,T)\rightarrow\R^2$ is a smooth entropy flow.
Fix $\delta\in(0,T)$.
Then for all $t\in(0,T)$ we have
\begin{equation}
\label{EQclaimo}
	\int k_{\theta\theta}^2 + k^2\,d\theta\bigg|_t \le 2\int k_\theta^2\,d\theta\bigg|_t + c_0(\delta)\,,
\end{equation}
where
\[
	c_0(\delta) = \vn{F}_2^2(\delta)
	\,.
\]
\end{cor}
\begin{proof}
Integrating \eqref{EQsecond} yields
\[
\SE'(t_0) + 2\int_\delta^{t_0} \int \big(k(F_{\theta\theta} + F)\big)^2\,d\theta\,dt
= \SE'(\delta)
\]
which implies
\begin{equation}
\label{CYenest1}
\vn{F}_2^2(t_0) \le c_0(\delta)\,.
\end{equation}
Integration by parts on the cross term in $\vn{F}_2^2(t)$ yields
\begin{equation}
\label{CYenest2}
\vn{F}_2^2
= \int k_{\theta\theta}^2 + k^2\,d\theta + 2\int k_{\theta\theta}k\,d\theta
= \int k_{\theta\theta}^2 + k^2\,d\theta - 2\int k_{\theta}^2\,d\theta
\,.
\end{equation}
Combining \eqref{CYenest1} with \eqref{CYenest2} gives the claimed estimate \eqref{EQclaimo}.
\end{proof}

Lemma \ref{LM4} implies that any possible finite-time explosion of the energy to $-\infty$ would require a finite-time explosion to $+\infty$ of the evolving length.
This leads us to investigate the dynamics of $L$. 

\begin{lem}
\label{LM2}
Suppose $\gamma:\S^1\times(0,T)\rightarrow\R^2$ is a smooth entropy flow.
The length functional $L$ is monotone increasing and concave; in particular,
\begin{equation}
\label{EQlvel}
	L' = \frac{d}{dt} \int h\,d\theta = \int k\,d\theta
\end{equation}
and
\[
	L'' \le -\frac12\int kk_{\theta\theta}^2\,d\theta - \frac13 \int k^3\,d\theta
\,.
\]
\end{lem}
\begin{proof}
Recalling \eqref{EQcurvatureevo} and $h_t = F = k_{\theta\theta} + k$, we calculate
\begin{align*}
L' = \frac{d}{dt}\int h\,d\theta
= \int k_{\theta\theta} + k\,d\theta
= \int k\,d\theta\,,
\end{align*}
which is positive.

Differentiating again, we find
\begin{align*}
L'' &= \frac{d}{dt}\int k\,d\theta
\\
&= \int k_t\,d\theta
\\
&= -\int k^2(F_{\theta\theta} + F)\,d\theta
\\
&= 
    2\int kk_\theta F_\theta\,d\theta
  - \int k^2F\,d\theta
\\
&= 
  - 2\int (kk_\theta)_\theta k_{\theta\theta}\,d\theta
  + 2\int kk_\theta^2\,d\theta
  + \int (k^2)_\theta k_\theta\,d\theta
  - \int k^3\,d\theta
\\
&= 
  - 2\int kk_{\theta\theta}^2\,d\theta
  - 2\int k^2k_{\theta\theta}\,d\theta
  - \int k^3\,d\theta
\\
&\le
  - \frac12\int kk_{\theta\theta}^2\,d\theta
  - \frac13\int k^3\,d\theta
\,.
\end{align*}
In the last step we used an instance of the Cauchy inequality: $ab \le (3/4)a^2 + (1/3)b^2$.
\end{proof}

Similarly to Corollary \ref{CYenest}, this yields a uniform estimate.
Note that in terms of the initial regularity, the estimate requires only that the $L^1(d\theta)$-norm of $k$ at initial time be bounded.
(In the arclength parametrisation, this is the $L^2(ds)$-norm of $k$.)

\begin{cor}
\label{CYkinl1}
Suppose $\gamma:\S^1\times[0,T)\rightarrow\R^2$ is a smooth entropy flow with $\gamma(\cdot,0)$ satisfying the hypothesis of Lemma \ref{LMinitialenergybdd}.
Then
\begin{equation}
\label{EQclaimy}
	\int k\,d\theta + \int_0^t\int \bigg(\frac12kk_{\theta\theta}^2 + \frac13 k^3\bigg)\,d\theta\,d\hat t \le c_1
\,,
\end{equation}
where $c_1 = \vn{k}_1(0)$.
\end{cor}

We now wish to control the long-time asymptotic growth of length from above and below along the flow (although we don't have long-time existence yet, once we do, this will be decisive).
An easy estimate from above of the form
\[
	L(t) \le L_0 + c_1t
\]
follows by combining Corollary \ref{CYkinl1} with \eqref{EQlvel}.
However this is far from sharp (the length of growing circles is like $\sqrt
t$), and we can do significantly better.

\begin{lem}
\label{LM3}
Suppose $\gamma:\S^1\times(0,T)\rightarrow\R^2$ is a smooth entropy flow.
The rate of blowup for length is asymptotically $\sqrt t$; more precisely
\[
	\sqrt{L_0^2 + 8\omega^2\pi^2t} \le L(t) 
					\le L_0 + 4\omega\pi c_1^{-1}\bigg( \sqrt{4\omega^2\pi^2 + tc_1^2} - 2\omega\pi \bigg)
	\,.
\]
\end{lem}
\begin{proof}
For the estimate from below, we use again Jensen's inequality with length.
Since $t\mapsto 1/t$ is convex, we find
\[
\frac{1}{\frac1{2\omega\pi}\int k\,d\theta} \le \frac1{2\omega\pi} \int \frac1k\,d\theta
\,.
\]
Therefore
\[
\int k\,d\theta \ge \frac{4\omega^2\pi^2}{L}
\,.
\]
Combining this with \eqref{EQlvel} yields
\[
(L^2)' \ge 8\omega^2\pi^2
\]
and the bound from below follows by integration.

First, note that we may estimate
\[
  - 2\int kk_{\theta\theta}^2\,d\theta
  - 2\int k^2k_{\theta\theta}\,d\theta
  - \int k^3\,d\theta
\le -\frac12\int k^3\,d\theta
\]
by using $2ab \le 2a^2 + \frac12b^2$. (This does not yield the good higher
order term from Lemma \ref{LM2}, but it is better for our present aim.)

Using H\"older we find
\[
\frac12\int k^3\,d\theta
\ge \frac12 \frac{1}{4\omega^2\pi^2} \bigg(\int k\,d\theta\bigg)^3
= \frac1{8\omega^2\pi^2}(L')^3
\,.
\]
Using this estimate we find
\begin{equation*}
L'' \le -\frac1{8\omega^2\pi^2} (L')^3(t)
\end{equation*}
which implies
\begin{equation}
\label{EQkdecay}
|L'| \le \frac{2\omega\pi}{\sqrt{4\omega^2\pi^2c_1^{-2} + t}}
\,.
\end{equation}
(Note that $L'(t) > 0$ so the absolute value sign is not strictly needed.)
This yields the estimate
\[
L(t) \le L_0 + 4\omega\pi c_1^{-1}\bigg( \sqrt{4\omega^2\pi^2 + tc_1^2} - 2\omega\pi \bigg)
\]
which is the upper bound, and finishes the proof.
\end{proof}

\begin{rmk}
Note that the estimate \eqref{EQkdecay} implies that $\int k\,d\theta$
asymptotically decays (assuming long time existence), with rate $\frac1{\sqrt t}$.
\end{rmk}

\begin{rmk}
The length estimate presented here is the remarkably powerful statement that
all convex curves have length evolving with the same asymptotic power of $t$
(which is of course the same as the circle).
This will be more interesting after we have shown that generic curves exist
globally in time.
An initial circle with radius $r_0>0$ has support function $h(\theta,t) =
\sqrt{r_0^2 + 2t}$, and length
\[
	L^{\text{circ}}(t) = \sqrt{8\omega^2\pi^2t + L_0^2}
	\,.
\]
This shows that the lower bound is sharp.
The rate ($\sqrt{t}$) of the upper bound is sharp but the form it is in could be tighter.
For the evolving circle, we have $c_1 =
\frac{2\omega\pi}{r_0} = \frac{4\omega^2\pi^2}{L_0}$, so the corresponding
upper bound of Lemma \ref{LM3} reads
\[
 L_0 + \frac{L_0}{\omega\pi}
		\bigg( \sqrt{4\omega^2\pi^2 + 16t\,\pi^4\omega^4} - 2\omega\pi \bigg)
= L_0 +  2\bigg( \sqrt{4\omega^2\pi^2 t + L_0^2} - L_0\bigg)
\,.
\]
Nevertheless it is good enough for our purposes.
\end{rmk}

\begin{rmk}
Given our study of length, it is natural to consider the signed enclosed area.
We find $A$ is increasing monotonically, blowing up as $t\nearrow\infty$.
The rate is approximately linear in time; in particular,
\[
	A' = \frac{d}{dt} \frac12\int \frac{h}{k}\,d\theta = 2\pi + \int ((\log k)_\theta)^2\,d\theta
\,.
\]
This yields a sharp bound from below: $A(t) \ge A_0 + 2\pi t$. The isoperimetric inequality and Lemma \ref{LM3} gives a bound from above.
\end{rmk}

Combining Lemma \ref{LM3} with Lemma \ref{LM4} allows us to rule out the possibility of the energy exploding to $-\infty$ in finite time.

\begin{cor}
\label{CYenbel}
Suppose $\gamma:\S^1\times[0,T)\rightarrow\R^2$ is a smooth entropy flow with $\gamma(\cdot,0)$ satisfying the hypothesis of Lemma \ref{LMinitialenergybdd}.
Then
\[
\SE_0 \ge \SE(t) \ge 2\omega\pi \log \frac{2\omega\pi}{L_0 + 4\omega^2\pi^2c_1^{-1}\big( \sqrt{4\omega^2\pi^2 + tc_1^2} - 2\omega\pi \big)}
\,.
\]
\end{cor}

The length estimate also allows us to conclude a smoothing-type effect for the velocity in $L^2$.

\begin{cor}
\label{CYvel2}
Suppose $\gamma:\S^1\times[0,T)\rightarrow\R^2$ is a smooth entropy flow with $\gamma(\cdot,0)$ satisfying the hypothesis of Lemma \ref{LMinitialenergybdd}.
For each $t\in(0,T)$,
\begin{equation}
\label{EQclaim1}
\vn{F}_2^2(t) < \infty
\,.
\end{equation}
\end{cor}
\begin{proof}
We have
\[
	\int_0^t \vn{F}_2^2(\hat t)\,d\hat t = \SE_0 - \SE(t)
	\,.
\]
Note that here we used Lemma \ref{LMinitialenergybdd} to ensure that $\SE_0$ on the RHS of the above is finite.

Using Corollary \ref{CYenbel} we refine this to
\begin{equation}
\label{EQfl2uptozero}
	\int_0^t \vn{F}_2^2(\hat t)\,d\hat t \le \SE_0 + 2\omega\pi \log \frac{L_0 + 4\omega^2\pi^2c_1^{-1}\big( \sqrt{4\omega^2\pi^2 + tc_1^2} - 2\omega\pi \big)}{2\omega\pi}
	\,.
\end{equation}
In particular, this shows that $\int_0^t \vn{F}_2^2(\hat t)\,d\hat t$, where $t\in(0,T)$. is
uniformly bounded (by a constant depending only on $\SE_0$, $L_0$, $c_1$, $\omega$ and $T$). 
This implies that there is a $t_0\in(0,t)$ such that $\vn{F}_2^2(t_0) < \infty$.
Then, \eqref{EQsecond} implies that $\vn{F}_2^2(t') \le \vn{F}_2^2(t_0) < \infty$ for all $t'\ge t_0$.
In particular, this holds for $t'=t$ which proves \eqref{EQclaim1}.
%
\end{proof}

Corollary \ref{CYvel2} gives that $||F||_2^2(t)$ is instantaneously bounded,
even if $||F||_2^2(0)$ is undefined.
The next result shows that $||F||_2^2(t)$ decays with rate $\frac{\log t}t$.
%
%

\begin{cor}
\label{CYl2dec}
Suppose $\gamma:\S^1\times[0,T)\rightarrow\R^2$ is a smooth entropy flow with $\gamma(\cdot,0)$ satisfying the hypothesis of Lemma \ref{LMinitialenergybdd}.
Then for each $T_0\in(0,T)$,
\begin{align*}
\vn{F}_2^2(T_0)
 \le \frac1{T_0}\bigg(\SE_0 + 2\omega\pi \log \frac{L_0 + 4\omega^2\pi^2c_1^{-1}\big( \sqrt{4\omega^2\pi^2 + T_0c_1^2} - 2\omega\pi \big)}{2\omega\pi}\bigg)
\,.
\end{align*}
\end{cor}
\begin{proof}
Calculate
\begin{align*}
(t\vn{F}_2^2(t))'
 &= \vn{F}_2^2(t) - 2t\int (k(F_{\theta\theta} + F))^2\,d\theta
\\
 &\le \vn{F}_2^2(t)
\,.
\end{align*}
Here we used Lemma \ref{LMeconvex}.
Integrating from zero to $T_0$ and using \eqref{EQfl2uptozero} we find
\begin{align*}
T_0\vn{F}_2^2(T_0)
 &\le \int_0^{T_0} \vn{F}_2^2(\hat t)\,d\hat t
\\
 &\le \SE_0 + 2\omega\pi \log \frac{L_0 + 4\omega^2\pi^2c_1^{-1}\big( \sqrt{4\omega^4\pi^4 + T_0c_1^2} - 2\omega^2\pi^2 \big)}{2\omega\pi}
\,.
\end{align*}
Dividing through by $T_0$ gives the result.
\end{proof}

Control on $\vn{F}_2^2$ allows us to obtain control on the $L^2(d\theta)$-norm of curvature.

\begin{prop}
\label{PRk2est}
Suppose $\gamma:\S^1\times[0,T)\rightarrow\R^2$ is a smooth entropy flow with $\gamma(\cdot,0)$ satisfying the hypothesis of Lemma \ref{LMinitialenergybdd}.
Let $[t_0,t_1]\subset\subset(0,T)$ be a compact time interval.
For all $t\in[t_0,t_1]$ we have the estimate
\[
\int k^2\,d\theta
\le C(t_0,t_1,\SE_0,\omega,\vn{k}_{L^1(d\theta)}(0))
\,.
\]
In particular, we have the following estimate on the whole time interval:
For $t\in(0,T)$
\[
\int k^2\,d\theta
\le C(\SE_0,L_0,\omega,\vn{k}_{L^1(d\theta)}(0))\Big(t + \frac{1}{t}\Big) 
\,.
\]
\end{prop}
\begin{proof}
We calculate
\begin{align*}
\frac{d}{dt}\int k^2\,d\theta
 &= -4\int k^3(F_{\theta\theta} + F)\,d\theta
\\
 &= -4\int (k_{\theta\theta}+k)((k^3)_{\theta\theta} + k^3)\,d\theta
\\
 &= -4\int (k_{\theta\theta}+k)((3k^2k_\theta)_{\theta} + k^3)\,d\theta
\\
 &= -4\int (k_{\theta\theta}+k)(3k^2k_{\theta\theta} + 6kk_\theta^2 + k^3)\,d\theta
\\
 &= -4\int 3k^2k_{\theta\theta}^2 + k_{\theta\theta}(6kk_\theta^2 + k^3)\,d\theta
    -4\int - 3k^2k_\theta^2 + k^4\,d\theta
\\
 &= -12\int k^2k_{\theta\theta}^2\,d\theta
    -4\int k^4\,d\theta
    + 8\int k_\theta^4 \,d\theta
    + 24\int k^2k_{\theta}^2\,d\theta
\\
 &\le -12\int k^2k_{\theta\theta}^2\,d\theta
    -\int k^4\,d\theta
    + 56\int k_\theta^4 \,d\theta
\,.
\end{align*}
Set
\begin{equation}
C_0(t) = \frac1{t}\bigg(\SE_0 + 2\omega\pi \log \frac{L_0 + 4\omega^2\pi^2c_1^{-1}\big( \sqrt{4\omega^2\pi^2 + tc_1^2} - 2\omega\pi \big)}{2\omega\pi}\bigg)
\,.
\label{EQczero}
\end{equation}
The reverse Sobolev inequality, Corollary \ref{CYenest}, combined with the decay estimate of Corollary \ref{CYl2dec} yields
\begin{equation}
\label{EQrepl}
\int k_{\theta\theta}^2 + k^2\,d\theta \le 2\int k_\theta^2\,d\theta
 + C_0(t)
\,.
\end{equation}
Estimating the right hand side (with integration by parts and $ab \le \frac14a^2 + b^2$), we find
\begin{align*}
\int k_{\theta\theta}^2\,d\theta 
 &\le 
\frac12\int k_{\theta\theta}^2\,d\theta + \int k^2\,d\theta
 + C_0(t)
\end{align*}
which implies
\begin{equation}
\label{EQrevsob}
\int k_{\theta\theta}^2\,d\theta 
 \le 
 2\int k^2\,d\theta
 + 2C_0(t)
\,.
\end{equation}
Above we have written explicitly that $C_0$ is multiplied by $2$.

We use \eqref{EQrevsob} in the following way. H\"older and Poincar\'e first imply
\[
\int k_\theta^4 \,d\theta
	\le  \vn{k_\theta}_\infty^2\vn{k_\theta}_2^2
	\le c\vn{k_{\theta\theta}}_2^4
\]
which combines with \eqref{EQrevsob} and then H\"older to yield
\[
\int k_\theta^4 \,d\theta
 \le 
 c\bigg(\int k^2\,d\theta\bigg)^2
 + c\,C_0^2(t)
 \le 
 c\int k\,d\theta\int k^3\,d\theta
 + c\,C_0^2(t)
\,.
\]
Using this estimate in the evolution equation for $\int k^2\,d\theta$ we find
\begin{align*}
\frac{d}{dt}\int k^2\,d\theta
 &\le   - 12\int k^2k_{\theta\theta}^2\,d\theta
	- \int k^4\,d\theta
	+ c\,c_1\int k^3\,d\theta
	+ c\,C_0^2(t)
\,,
\end{align*}
where we used Corollary \ref{CYkinl1}.

Now we compute
\begin{align*}
\frac{d}{dt}\bigg( (t-t_0/2)_+\int k^2\,d\theta \bigg)
 &\le   - 12(t-t_0/2)_+\int k^2k_{\theta\theta}^2\,d\theta
	- (t-t_0/2)_+\int k^4\,d\theta
\\&\qquad
	+ c\,c_1(t-t_0/2)_+\int k^3\,d\theta
	+ c\,C_0^2(t)(t-t_0/2)_+
\\&\qquad
        + \chi_{t\ge t_0/2}\int k^2\,d\theta
\\&\le
	- 12(t-t_0/2)_+\int k^2k_{\theta\theta}^2\,d\theta
	- \frac12(t-t_0/2)_+\int k^4\,d\theta
\\&\qquad
	+ c(c_1^2+C_0^2(t))(t-t_0/2)_+
        + c\frac{\chi_{t\ge t_0/2}^2}{(t-t_0/2)_+}
\,.
\end{align*}
Above we interpolated the terms $||k||_3^3$ and $||k||_2^2$.
Integrating yields
\begin{align}
\label{EQestfork2}
\int k^2\,d\theta
&\le
	c\,c_1^2(t-t_0/2)_+
	+ c(t-t_0/2)_+^{-1}\int_{t_0}^t\,C_0^2(\tilde t)(\tilde t-t_0/2)_+\,d\tilde t
\\&\qquad\notag
        + c\chi_{t\ge t_0/2}^2(t-t_0/2)_+^{-1}\log\frac{(t-t_0/2)_+}{t_0/2}
\,.
\end{align}
Let $\alpha^2(t) = t^2C_0^2(t)$.
Then the integral on the right hand side above satisfies
\begin{align*}
\int_{t_0}^t\,C_0^2(\tilde t)(\tilde t-t_0/2)_+\,d\tilde t
&\le
	\alpha^2(t)\Big(\log\frac{t}{t_0} - \frac{t_0}{2}\Big( \frac1{t_0} - \frac1t \Big)\,\Big)
=
	\alpha^2(t)\Big(\log\frac{t}{t_0} + \frac{t_0}{2t} - \frac12\Big)
\,.
\end{align*}
Note that with the factor $(t-t_0/2)_+^{-1}$ in front, this integral decays to zero for large $t$.

The right hand side of \eqref{EQestfork2} is uniformly bounded for $t\in[t_0,t_1]$, which finishes the proof of the first claimed estimate.
For the second estimate, we may apply the first on an interval $[t_0,t_1]$
where $t_0\in(0,t)$ and take $t_0\searrow0$, noting that the asymptotics (see
\eqref{EQestfork2}) for small $t$ are controlled by $\frac1t$ and
for large $t$ are controlled by $t$.
Since the resultant estimate holds for all $t_1<T$ and in fact does not depend
on $t_1$, it holds for all $t<T$.
\end{proof}

This then allows us to conclude an estimate for $k_{\theta\theta}$ in $L^2(d\theta)$.

\begin{cor}
\label{CYkttest}
Suppose $\gamma:\S^1\times[0,T)\rightarrow\R^2$ is a smooth entropy flow with $\gamma(\cdot,0)$ satisfying the hypothesis of Lemma \ref{LMinitialenergybdd}.
For $t\in(0,T)$ we have the estimate
\[
\int k_{\theta\theta}^2\,d\theta
\le C(\SE_0,L_0,\omega,\vn{k}_{L^1(d\theta)}(0))\Big(t + \frac{1}{t}\Big) 
\,.
\]
\end{cor}
\begin{proof}
Combine estimate \eqref{EQrevsob} from the proof of Proposition \ref{PRk2est} with the conclusion of Proposition \ref{PRk2est}.
\end{proof}

\begin{prop}
\label{PNgrad}
Suppose $\gamma:\S^1\times[0,T)\rightarrow\R^2$ is a smooth entropy flow with $\gamma(\cdot,0)$ satisfying the hypothesis of Lemma \ref{LMinitialenergybdd}.
For $t\in(0,T)$ we have the estimate
\[
\vn{k_{\theta}}_{L^\infty(d\theta)}^2
\le C(\SE_0,L_0,\omega,\vn{k}_{L^1(d\theta)}(0))\Big(t + \frac{1}{t}\Big) 
:= c_2(t)^2
\,.
\]
\end{prop}
\begin{proof}
Use Corollary \ref{CYkttest} with the Sobolev and then H\"older inequalities:
\begin{align*}
	k_\theta^2 
&\le 2\omega\pi\int_0^{2\omega\pi} k_{\theta\theta}^2\,d\vartheta
\le C(\SE_0,L_0,\omega,\vn{k}_{L^1(d\theta)}(0))\Big(t + \frac{1}{t}\Big) 
\,.
\end{align*}
Taking a supremum finishes the proof.
\end{proof}

We now use the gradient estimate to obtain preservation of $k>0$.

\begin{prop}
\label{PN1}
Suppose $\gamma:\S^1\times[0,T)\rightarrow\R^2$ is a smooth entropy flow with $\gamma(\cdot,0)$ satisfying the hypothesis of Lemma \ref{LMinitialenergybdd}.
Then the flow remains strictly convex for all time, and if
$T<\infty$, the infimum of the curvature as $t\nearrow T$ remains uniformly
positive.
More precisely, the infimum of the curvature at time $t$ satisfies
\[
k_0 \ge C\frac{t}{t^2+1}\frac{1}{\exp\bigg(C\Big(t+\frac1t\Big)\Big(\frac{L_0}{2}+2\omega\pi c_1^{-1}\big( \sqrt{4\omega^2\pi^2 + tc_1^2} - 2\omega\pi \big)\Big)\bigg)-1}
\,,
\]
where $c_1 = \vn{k}_{L^1(d\theta)}(0)$ and $C=C(\SE_0,L_0,\omega,\vn{k}_{L^1(d\theta)}(0))$.
\end{prop}
\begin{proof}
Parameterise $\gamma$ on $[-\omega\pi,\omega\pi)$ so that $k_0(t) := k(0,t) = \inf_{\theta\in[-\omega\pi,\omega\pi)}k(\theta,t)$.
Observe that the gradient bound above implies $k(\theta) \le \inf k + c_2(t)|\theta|$, so
\begin{align*}
L &= \int_{-\omega\pi}^{\omega\pi} \frac1k\,d\theta
\\
&\ge  \int_{-\omega\pi}^{\omega\pi} \frac1{k_0 + c_2(t)|\theta|}\,d\theta
\\
&=  2\int_{0}^{\omega\pi} \frac1{k_0 + c_2(t)\theta}\,d\theta
\\
&= \frac{2}{c_2(t)}\log\bigg(\frac{k_0 + c_2(t)\omega\pi}{k_0}\bigg)
= \frac{2}{c_2(t)}\log\bigg(1 + \frac{c_2(t)\omega\pi}{k_0}\bigg)
\,.
\end{align*}
Rearranging yields
\[
k_0 \ge \frac{1}{c_2(t)\omega\pi}\frac{1}{e^{c_2(t)L/2}-1}\,.
\]
Combining this with Lemma \ref{LM3} and using the expression for $c_2(t)$ from Proposition \ref{PNgrad} yields the claimed estimate.
\end{proof}

While the form of the estimate for $k$ from below in Proposition \ref{PN1} is
somewhat messy, and degenerates as $t\searrow0$ or as $t\nearrow\infty$, the
important point is that it is strictly positive for any particular $t>0$.
This is enough to preserve strict convexity, independent of final time (unless
the final time is infinite).
Furthermore, Corollary \ref{CYkttest} implies $W^{3,\infty}(du)$ regularity up
to and including final time.
We may then apply Theorem \ref{TM1} in a standard way to obtain infinite
maximal time of existence, so long as we have $W^{3,\infty}$ initial data.

\begin{cor}
\label{CYglobal}
Suppose $\gamma:\S^1\times[0,T)\rightarrow\R^2$ is a smooth entropy flow with $\gamma(\cdot,0)$ satisfying the hypothesis of Lemma \ref{LMinitialenergybdd}.
If $||\gamma(\cdot,0)||_{W^{3,\infty}(du)} < \infty$, then the maximal time of existence is infinite.
\end{cor}

We wish to push the initial regularity requirement as far as we can; in fact,
we wish to remove the $W^{3,\infty}$ condition entirely.
Our estimates are well-suited to this, as we have been careful to
require as little initial regularity as possible.
Up to now, we have a strictly (locally) convex initial curve 
that has curvature in $L^1(d\theta)$.
Lemma \ref{LMinitialenergybdd} shows that this implies the initial entropy $\SE_0$ is well-defined (and bounded).
This is all we have used to ensure that the flow smoothes out and remains strictly convex.

Our strategy is to use a compactness argument with these estimates.
The main missing component is uniform interior estimates for the flow, which we now establish.

We first give some special cases: the support function and its first two derivatives in $L^2(d\theta)$.

\begin{prop}
\label{PNhbase}
Suppose $\gamma:\S^1\times[0,T)\rightarrow\R^2$ is a smooth entropy flow with $\gamma(\cdot,0)$ satisfying the hypothesis of Lemma \ref{LMinitialenergybdd}.
The support function satisfies
\[
	\vn{h}_2^2(t) = \vn{h}_2^2(0) + 4t\omega\pi
	\,,
\]
\[
	(\vn{h_\theta}_2^2)' = -2\int k^{-2}k_\theta^2\,d\theta
	\,,
\]
and
\[
	(\vn{h_{\theta\theta}}_2^2)' =
    -2\int k^{-2}k_{\theta\theta}^2d\theta
    +4\int k^{-4}k_\theta^4\,d\theta
	\,.
\]

\end{prop}
\begin{proof}
Differentiating, we find
\begin{align*}
\frac{d}{dt}\int h^2\,d\theta
 &= 2\int h(k_{\theta\theta} + k)\,d\theta
\\
 &= 2\int k(h_{\theta\theta} + h)\,d\theta = 4\omega\pi
\,.
\end{align*}
For the derivative estimate, we calculate
\begin{align*}
\frac{d}{dt}\int h_\theta^2\,d\theta
 &= 2\int h_\theta(k_{\theta\theta} + k)_\theta\,d\theta
\\
 &= -2\int h_{\theta\theta}(k_{\theta\theta} + k)\,d\theta
\\
 &= -2\int k^{-1}(k_{\theta\theta} + k)\,d\theta
  + 2\int h(k_{\theta\theta} + k)\,d\theta
\\
 &= -2\int k^{-1}(k_{\theta\theta} + k)\,d\theta
  + 2\int k(h_{\theta\theta} + h)\,d\theta
\\
 &= -2\int k^{-1}k_{\theta\theta}\,d\theta
  = -2\int k^{-2}k_{\theta}^2\,d\theta
\,.
\end{align*}
For the second derivative we compute
\begin{align*}
\frac{d}{dt}\int h_{\theta\theta}^2\,d\theta
 &= 2\int h_{\theta\theta}(k_{\theta\theta} + k)_{\theta\theta}\,d\theta
\\
 &= 2\int h_{\theta^4}(k_{\theta\theta} + k)\,d\theta
\\
 &= 2\int (h_{\theta^6} + h_{\theta^4})k\,d\theta
\\
 &= 2\int (1/k)_{\theta^4}k\,d\theta
  = 2\int (1/k)_{\theta\theta}k_{\theta\theta}\,d\theta
\\
 &= 
    2\int (-k^{-2}k_{\theta\theta} + 2k^{-3}k_\theta^2)k_{\theta\theta}\,d\theta
\\
 &= 
    -2\int k^{-2}k_{\theta\theta}^2d\theta
    +4\int k^{-4}k_\theta^4\,d\theta
\,.
\end{align*}
\end{proof}

We note that we also have the following pointwise estimates for the support
function and curvature (from above).

\begin{lem}
\label{LMpwsupport}
Suppose $\gamma:\S^1\times[0,T)\rightarrow\R^2$ is a smooth entropy flow with $\gamma(\cdot,0)$ satisfying the hypothesis of Lemma \ref{LMinitialenergybdd}.
Then $||h_\theta||_2^2(0) = \lim_{t\searrow0} ||h_\theta||_2^2(t)$ exists, and
\[
	||h||_\infty
	\le C(1+\sqrt{t})
\,,
\]
where $C=C(||h_\theta||_2^2(0),||k||_1(0),\omega)$.
\end{lem}
\begin{proof}
We find
\[
	h = h - \overline{h} + \overline{h}
	\le \int |h_\theta|\,d\theta
	  + \frac1{2\omega\pi}L
	\le \sqrt{2\omega\pi}||h_\theta||_2
	  + \frac1{2\omega\pi}L
\]
using the H\"older inequality, fundamental theorem of calculus and formula $L = \int h\,d\theta$.
Since
\[
	\int h_\theta^2\,d\theta
	= -\int h_{\theta\theta}h\,d\theta
	= -\int \frac{h}{k}\,d\theta
	+ \int h^2\,d\theta
	= -2A+||h||_2^2
\]
we know that $||h_\theta||_2^2(0) = \lim_{t\searrow0} ||h_\theta||_2^2(t)$ exists.

Then Lemma \ref{LM3} and Proposition \ref{PNhbase} imply
\[
	||h||_\infty
	\le \sqrt{2\omega\pi}||h_\theta||_2(0)
		+ \frac{1}{2\omega\pi}L_0
		+ 2c_1^{-1}\bigg( \sqrt{4\omega^2\pi^2 + tc_1^2} - 2\omega\pi \bigg)
\,.
\]
\end{proof}

\begin{lem}
\label{LMkfromabove}
Suppose $\gamma:\S^1\times[0,T)\rightarrow\R^2$ is a smooth entropy flow with $\gamma(\cdot,0)$ satisfying the hypothesis of Lemma \ref{LMinitialenergybdd}.
Then 
\[
k^2 \le C(\SE_0,L_0,\omega,\vn{k}_{L^1(d\theta)}(0))\Big(t + \frac{1}{t}\Big)
\,.
\] 
\end{lem}
\begin{proof}
Using Corollary \ref{CYkinl1} and Proposition \ref{PNgrad} we calculate
\[
k = \overline{k} + \int |k_\theta|\,d\theta
\le \frac{1}{2\omega\pi}c_1 2\omega\pi\,c_2(t)
\,.
\]
\end{proof}

While the norm $||h_{\theta\theta}||_2^2$ is not decaying or have a
particularly simple evolution, due to Proposition \ref{PNgrad} and Proposition
\ref{PN1} on a compact time interval $[t_1,t_2]\subset\subset(0,T)$ we are able
to uniformly estimate the reaction term $\int k^{-4}k_\theta^4\,d\theta$ by a
constant.
The next proposition uses this kind of crude technique to obtain uniform
estimates on compact time intervals for $||h_{\theta^p}||_2^2$ for any $p$.

\begin{prop}
\label{PNallest}
Suppose $\gamma:\S^1\times[0,T)\rightarrow\R^2$ is a smooth entropy flow with $\gamma(\cdot,0)$ satisfying the hypothesis of Lemma \ref{LMinitialenergybdd}.
Let $[t_0,t_1]\subset\subset(0,T)$ be a compact time interval, and let $p\in\N_0$.
For all $t\in[t_0,t_1]$ we have the estimate
\begin{equation}
\int h_{\theta^p}^2\,d\theta
\le C(p,t_0,t_1,\SE_0,\omega,\vn{k}_1(0),||h_\theta||_2(0))
\,.
\label{EQsupporthigher}
\end{equation}
\end{prop}
\begin{proof}
First, we establish the required estimate for $p\in\{0,1,2,3,4,5,6\}$.
Here and throughout the proof we take $t\in[t_1,t_2]$, the constant $C$ depends
on $p. t_0, t_1, \SE_0, \omega$, and $\vn{k}_{L^1(d\theta)}(0)$, and may vary
from line to line.

{\bf Step 1. $p=0$ and $p=1$}.
The estimate \eqref{EQsupporthigher} follows from Proposition \ref{PNhbase}.

{\bf Step 2. $p=2$ and $p=3$}.
Since
\begin{equation}
\label{EQinview}
 h_{\theta\theta}
 =
	\frac1k - h
\,,\text{ and}\,,
h_{\theta\theta\theta} = -k^{-2}k_\theta - h_\theta
\,,
\end{equation}
we find
\begin{align*}
	||h_{\theta\theta}||_2^2 
	&\le 2\omega\pi||1/k||^2_\infty 
	\le C\,,\text{ and}
\\
	||h_{\theta\theta\theta}||_2^2 
	&\le 
		2\int k^{-4}k_\theta^2\,d\theta + 2\int h_\theta^2\,d\theta
	\le 2\omega\pi||k^{-1}||^4_\infty ||k_\theta||^2_\infty + 2||h_\theta||_2^2(0)
	\le C
\end{align*}
using Proposition \ref{PNgrad}, Proposition \ref{PN1} and Proposition
\ref{PNhbase} (and recalling $t\in[t_1,t_2]$ plus what $C$ depends on).

{\bf Step 3. $p=4$, $p=5$ and $p=6$}.
We calculate
\begin{align}
\frac{d}{dt} \int h_{\theta^{4}}^2\,d\theta
&=
	2\int h_{\theta^{4}}(k_{\theta\theta} + k)_{\theta^{4}}\,d\theta
\notag\\
&=
	2\int (h_{\theta^{6}} + h_{\theta^{4}})k_{\theta^{4}}\,d\theta
\notag\\
&=
	2\int (k^{-1})_{\theta^{4}}k_{\theta^{4}}\,d\theta
\,.
\label{EQp4evo}
\end{align}
Using
\begin{align}
(k^{-1})_{\theta^4}
&=
	(-k^{-2}k_\theta)_{\theta^3}
\notag\\&=
	(
	- k^{-2}k_{\theta\theta}
	+ 2k^{-3}k_\theta^2
	)_{\theta^2}
\notag\\&=
	(
	- k^{-2}k_{\theta^3}
	+ 6k^{-3}k_{\theta\theta}k_\theta
	- 6k^{-4}k_\theta^3
	)_{\theta}
\notag\\&=
	- k^{-2}k_{\theta^4}
	+ 8k^{-3}k_{\theta^3}k_{\theta}
	+ 6k^{-3}k_{\theta\theta}^2
	- 36k^{-4}k_{\theta\theta}k_\theta^2
	+ 24k^{-5}k_\theta^4
\label{EQ4threcipderiv}
\end{align}
in \eqref{EQp4evo} and then estimating all occurrences of $k^{-1}$ and $k_\theta$ by constants, we find
\begin{align*}
\frac{d}{dt} \int h_{\theta^{4}}^2\,d\theta
&\le
	-2\int k^{-2}k_{\theta^{4}}^2\,d\theta
	+ C\int |k_{\theta^{4}}|
	  (|k_{\theta^3}|
	+ k_{\theta\theta}^2
	+ |k_{\theta\theta}|
	+ 1)
	\,d\theta
\,.
\end{align*}
Interpolating refines this to
\begin{align}
\frac{d}{dt} \int h_{\theta^{4}}^2\,d\theta
&\le
	-\frac32\int k^{-2}k_{\theta^{4}}^2\,d\theta
	+ C\int 
	  k_{\theta^3}^2
	+ k_{\theta\theta}^4
	\,d\theta
	+ C
\notag\\&\le
	-c\int k_{\theta^{4}}^2\,d\theta
	+ C\int 
	  k_{\theta\theta}^4
	\,d\theta
	+ C
\,.
\label{EQreturn}
\end{align}
Note that we used integration by parts and interpolation on the
$||k_{\theta^3}||_2^2$ term, and we applied Lemma \ref{LMkfromabove} (with the
assumption $t\in[t_0,t_1]$) to estimate $-k^{-1}$ from above.

Let $A\in\R$ be a constant and $\delta>0$ be fixed.
For the remaining term, we use the following estimate
\begin{equation}
\label{EQfollowest}
	A\int 
	  k_{\theta\theta}^4
	\,d\theta
	\le \delta
		\int
		  k_{\theta^4}^2
		\,d\theta
		+ C(A)
\end{equation}
with $\delta = \frac{c}{2}$.
To see this, use integration by parts, interpolation, and Proposition \ref{PNgrad} to find
\[
	A\int 
	  k_{\theta\theta}^4
	\,d\theta
	\le AC\int |k_{\theta^3}|\,k_{\theta\theta}^2\,d\theta
	\le \frac{A}{2}
		\int
		  k_{\theta\theta}^4
		\,d\theta
		+ AC
		\int
		  k_{\theta^3}^2
		\,d\theta
\,.
\]
Then, absorbing the first term on the left, we continue to estimate
\[
	A\int 
	  k_{\theta\theta}^4
	\,d\theta
	\le 
		2AC
		\int
		  |k_{\theta^4}|\,|k_{\theta\theta}|
		\,d\theta
	\le 
	    \frac{\delta}{4}
		\int
		  k_{\theta^4}^2
		\,d\theta
		+ \frac{A}{2}
		\int
		  k_{\theta\theta}^4
		\,d\theta
		+ C(A)
\,.
\]
Absorbing one more time on the left finishes the proof of \eqref{EQfollowest}.
Note that above we used the inequality $ab \le \varepsilon_0 a^2 +
\varepsilon_1 b^4 + C(\varepsilon_0,\varepsilon_1)$ which follows from two
applications of Young's inequality, and holds for all $\varepsilon_0, \varepsilon_1 > 0$.

With the estimate \eqref{EQfollowest}, we return to \eqref{EQreturn} and continue to find
\begin{align}
\frac{d}{dt} \int h_{\theta^{4}}^2\,d\theta
&\le
	-c\int k_{\theta^{4}}^2\,d\theta
	+ C\int 
	  k_{\theta\theta}^4
	\,d\theta
	+ C
\notag\\&\le
	-\frac{c}{2}\int k_{\theta^{4}}^2\,d\theta
	+ C
\,.
\label{EQreturn2}
\end{align}
Now we cut off in time.
We use \eqref{EQreturn2} to calculate
\begin{equation}
\label{EQfromhere2}
\frac{d}{dt}\bigg((t-t_0/2)_+^2\int h_{\theta^{4}}^2\,d\theta\bigg)
\le 
	- \frac{c}2(t-t_0/2)_+^2\int k_{\theta^4}^2\,d\theta
	+ C(t-t_0/2)_+^2
	+ 2(t-t_0/2)_+\int h_{\theta^{4}}^2\,d\theta
\,.
\end{equation}
In view of \eqref{EQinview}, \eqref{EQ4threcipderiv} and Proposition \ref{PN1}, we see that
\begin{align*}
||h_{\theta^{5}}||_2^2
 &\le
 	 2||h_{\theta^{3}}||_2^2
	+ 2||(k^{-1})_{\theta^3}||_2^2
\\
 &\le
 	  C(1 + ||k_{\theta^3}||_2^2 + ||k_{\theta^{2}}||_2^2)
\,.
\end{align*}
Applying this estimate after interpolation (and the $p=3$ step) we find
\begin{align*}
	2(t-t_0/2)_+\int h_{\theta^{4}}^2\,d\theta
&\le
	C\int h_{\theta^{5}}^2\,d\theta
	+ C\int h_{\theta^{3}}^2\,d\theta
\\
&\le
	C\int k_{\theta^{3}}^2+k_{\theta^{2}}^2\,d\theta
	+ C
\\
&\le
	\frac{c}{4} (t-t_0/2)_+^2\int k_{\theta^4}^2\,d\theta
	+ C
\,.
\end{align*}
We use this to absorb the rightmost term in \eqref{EQfromhere2}, giving
\begin{equation*}
\frac{d}{dt}\bigg((t-t_0/2)_+^2\int h_{\theta^{4}}^2\,d\theta\bigg)
\le 
	C
\,.
\end{equation*}
Integration gives the estimate \eqref{EQsupporthigher}, finishing the proof for $p=4$.

For $p=5$, we begin with
\begin{align}
(k^{-1})_{\theta^5}
\notag&=
	(- k^{-2}k_{\theta^4}
	+ 8k^{-3}k_{\theta^3}k_{\theta}
	+ 6k^{-3}k_{\theta\theta}^2
	- 36k^{-4}k_{\theta\theta}k_\theta^2
	+ 24k^{-5}k_\theta^4)_\theta
\notag\\&=
	- k^{-2}k_{\theta^5}
	+ k^{-3}(10k_{\theta^4}k_\theta
		+ 20k_{\theta^3}k_{\theta\theta})
	+k^{-4}(-60k_{\theta^3}k_{\theta}^2
		-90k_{\theta\theta}^2k_\theta)
	+276k^{-5}k_{\theta\theta}k_\theta^3
	-120k^{-6}k_\theta^5
\,.
\label{EQ5threcipderiv}
\end{align}
Using \eqref{EQ5threcipderiv} (and Proposition \ref{PNgrad}, Proposition \ref{PN1}, Lemma \ref{LMkfromabove}) we find
\begin{align*}
\frac{d}{dt} \int h_{\theta^{5}}^2\,d\theta
&=
	2\int (k^{-1})_{\theta^{5}}k_{\theta^{5}}\,d\theta
\\
&\le
	-2c\int k_{\theta^{5}}^2\,d\theta
	+ C\int |k_{\theta^{5}}|
	  (|k_{\theta^4}|
	+ |k_{\theta^3}|\,|k_{\theta\theta}|
	+ |k_{\theta^3}|
	+ |k_{\theta\theta}|^2
	+ |k_{\theta\theta}|
	+ 1)
	\,d\theta
\,.
\end{align*}
Interpolating refines this to
\begin{align}
\frac{d}{dt} \int h_{\theta^{5}}^2\,d\theta
&\le
	-\frac32c\int k_{\theta^{5}}^2\,d\theta
	+ C\int 
	  k_{\theta^4}^2
	+ k_{\theta^3}^2
	+ k_{\theta\theta}^2
	+ k_{\theta^3}^2 k_{\theta\theta}^2
	\,d\theta
	+ C
\notag\\&\le
	-c\int k_{\theta^{5}}^2\,d\theta
	+ \delta\int 
	  k_{\theta^3}^4
		\,d\theta
	+ C\delta^{-1}\int 
	  k_{\theta\theta}^4
	\,d\theta
	+ C
\,.
\label{EQrefinep5}
\end{align}
The $\delta^{-1}||k_{\theta\theta}||_4^4$ term is dealt with by estimate \eqref{EQfollowest} (and then interpolation).
For the $||k_{\theta^3}||_4^4$ term, we wish to use a similar estimate.
We take this opportunity to derive a general version.

Let $A\in\R$ be a constant.
Assume $||k_{\theta^{m-2}}||_\infty \le C$.
We claim
\begin{equation}
\label{EQfollowestgen}
	A \int
	k_{\theta^{m+1}}^2\,k_{\theta^{m-1}}^2
		+ k_{\theta^{m+1}}^\frac{8}{3}
		+ k_{\theta^{m}}^4\,d\theta
	\le AC
		\int
		  k_{\theta^{m+2}}^2
		\,d\theta
\end{equation}
First, estimate
\[
	A\int 
	  k_{\theta^m}^4
	\,d\theta
	\le AC\int |k_{\theta^{m+1}}|\,k_{\theta^m}^2|k_{\theta^{m-1}}|\,d\theta
	\le \frac{A}{2}
		\int
		  k_{\theta^m}^4
		\,d\theta
		+ AC
		\int
			k_{\theta^{m+1}}^2\,k_{\theta^{m-1}}^2\,d\theta
		\,d\theta
\,.
\]
Absorbing the first term on the left, we have
\begin{align*}
	A\int 
	  k_{\theta^m}^4
	\,d\theta
	&\le 
		AC
		\int
			k_{\theta^{m+1}}^2\,k_{\theta^{m-1}}^2\,d\theta
		\,d\theta\,.
\end{align*}
Continuing, we find
\begin{align*}
	A\int 
	  k_{\theta^m}^4
	\,d\theta
	+
	AC\int
	k_{\theta^{m+1}}^2\,k_{\theta^{m-1}}^2\,d\theta
	&\le 
		AC
		\int
			|k_{\theta^{m-2}}|(|k_{\theta^{m}}|\,k_{\theta^{m+1}}^2
				+ |k_{\theta^{m+2}}|\,|k_{\theta^{m+1}}|\,|k_{\theta^{m-1}}|)
		\,d\theta
	\\&\le 
		AC
		\int
			|k_{\theta^{m}}|\,k_{\theta^{m+1}}^2
				+ |k_{\theta^{m+2}}|\,|k_{\theta^{m+1}}|\,|k_{\theta^{m-1}}|
		\,d\theta
	\\&\le 
		AC\int
			k_{\theta^{m+2}}^2
		\,d\theta
		+ \frac{A}{2}
		\int
			k_{\theta^{m}}^4
		\,d\theta
		+ A^2C
		\int
			k_{\theta^{m+1}}^\frac{8}{3}
		\,d\theta
	\\&\quad
		+ \frac{AC}{2}
		\int
			k_{\theta^{m+1}}^2\,k_{\theta^{m-1}}^2
		\,d\theta\,.
\end{align*}
Absorbing again, we find
\begin{align}
	A \int
	k_{\theta^{m+1}}^2\,k_{\theta^{m-1}}^2\,d\theta
	+ AC \int
	k_{\theta^{m}}^4\,d\theta
	&\le 
		AC\int
			k_{\theta^{m+2}}^2
		\,d\theta
		+ AC
		\int
			k_{\theta^{m+1}}^\frac{8}{3}
		\,d\theta
\label{EQgenlabsorby}
\end{align}
Finally, we estimate
\begin{align*}
	AC
	\int
		k_{\theta^{m+1}}^\frac{8}{3}
	\,d\theta
	\le AC\int k_{\theta^{m+1}}^{\frac23}|k_{\theta^{m+2}}|\,|k_{\theta^m}|\,d\theta
	\le \frac{AC}{2}\int k_{\theta^{m+1}}^{\frac83}\,d\theta
	  + AC\int k_{\theta^{m+2}}^\frac{4}{3} k_{\theta^m}^\frac{4}{3} \,d\theta
\end{align*}
which gives, absorbing once again,
\begin{align*}
	AC
	\int
		k_{\theta^{m+1}}^\frac{8}{3}
	\,d\theta
	\le 
	  AC\int k_{\theta^{m+2}}^\frac{4}{3} k_{\theta^m}^\frac{4}{3} \,d\theta\,.
\end{align*}
Using this now in \eqref{EQgenlabsorby} we have
\begin{align*}
	A \int
	k_{\theta^{m+1}}^2\,k_{\theta^{m-1}}^2
		+ k_{\theta^{m+1}}^\frac{8}{3}
		+ k_{\theta^{m}}^4\,d\theta
	&\le
	A \int
	k_{\theta^{m+1}}^2\,k_{\theta^{m-1}}^2\,d\theta
	+ AC
	\int
		k_{\theta^{m+1}}^\frac{8}{3}
	\,d\theta
	+ AC \int
	k_{\theta^{m}}^4\,d\theta
	\\&
	\le 
		AC\int
			k_{\theta^{m+2}}^2
		\,d\theta
	  	+ AC\int k_{\theta^{m+2}}^\frac{4}{3} k_{\theta^m}^\frac{4}{3} \,d\theta
	\\&\le 
		AC\int
			k_{\theta^{m+2}}^2
		\,d\theta
	  	+ \frac{A}{2}\int k_{\theta^m}^4 \,d\theta
\,.
\end{align*}
Absorbing one final time gives the claimed estimate \eqref{EQfollowestgen}.

With the estimate \eqref{EQfollowestgen}
and using \eqref{EQfollowest} on the term $\delta^{-1}||k_{\theta\theta}||_4^4$ followed by interpolation,
we return to \eqref{EQrefinep5} to find
\begin{align}
\frac{d}{dt} \int h_{\theta^{5}}^2\,d\theta
&\le
	-\frac{c}{2}\int k_{\theta^{5}}^2\,d\theta
	+ C
\,.
\label{EQreturntwo5th}
\end{align}
Now we cut off in time, as before in the $p=4$ case.
The details are similar, so we will be brief.
We use \eqref{EQreturntwo5th} to calculate
\begin{equation}
\label{EQfromhere22}
\frac{d}{dt}\bigg((t-t_0/2)_+^2\int h_{\theta^{5}}^2\,d\theta\bigg)
\le 
	- \frac{c}2(t-t_0/2)_+^2\int k_{\theta^5}^2\,d\theta
	+ C(t-t_0/2)_+^2
	+ 2(t-t_0/2)_+\int h_{\theta^{5}}^2\,d\theta
\,.
\end{equation}
Since
\begin{align*}
||h_{\theta^{6}}||_2^2
 &\le
 	  C(1 + ||k_{\theta^4}||_2^2 + ||k_{\theta^{3}}||_2^2 + ||k_{\theta^{2}}||_2^2)
\,,
\end{align*}
we find
\begin{align*}
	2(t-t_0/2)_+\int h_{\theta^{5}}^2\,d\theta
&\le
	C\int h_{\theta^{6}}^2\,d\theta
	+ C\int h_{\theta^{4}}^2\,d\theta
\le
	C\int k_{\theta^{4}}^2+k_{\theta^{3}}^2+k_{\theta^{2}}^2\,d\theta
	+ C
\\
&\le
	\frac{c}{4} (t-t_0/2)_+^2\int k_{\theta^5}^2\,d\theta
	+ C
\,.
\end{align*}
We use this to absorb the rightmost term in \eqref{EQfromhere22}, giving
\begin{equation*}
\frac{d}{dt}\bigg((t-t_0/2)_+^2\int h_{\theta^{5}}^2\,d\theta\bigg)
\le 
	C
\,.
\end{equation*}
Integration gives the estimate \eqref{EQsupporthigher}, finishing the proof for $p=5$.

Finally $p=6$.
Analogously to before, we begin with $(k^{-1})_{\theta^6}$, but this time we write it more succinctly:
\begin{align}
(k^{-1})_{\theta^6}
\notag&=
	(- k^{-2}k_{\theta^5}
	+ k^{-3}(10k_{\theta^4}k_\theta
		+ 20k_{\theta^3}k_{\theta\theta})
	+k^{-4}(-60k_{\theta^3}k_{\theta}^2
		-90k_{\theta\theta}^2k_\theta)
	+276k^{-5}k_{\theta\theta}k_\theta^3
	-120k^{-6}k_\theta^5)_\theta
\notag\\&=
	-k^{-2}k_{\theta^6}
	+
	\sum_{j=2}^6 k^{-1-j} \sum_{i_1+\cdots+i_j} c(i_1,\ldots,i_j)k_{\theta^{i_1}}\cdots k_{\theta^{i_j}}
\,.
\label{EQ6threcipderiv}
\end{align}
Above we implicitly assume each $i_j\in\N$, and the constants
$c(i_1,\ldots,i_j)$ are universal.
The $p=5$ case just treated implies that $k_{\theta\theta}$ in addition to
$k^{-1}$ and $k_\theta$ is uniformly bounded in $L^\infty$ on $[t_1,t_2]$,
because
\begin{equation}
\label{EQnewcontrol}
	|k_{\theta\theta}|
	 = \bigg|\bigg(\frac{1}{h_{\theta\theta}+h}\bigg)_{\theta\theta}\bigg|
	 = \bigg|\bigg(-\frac{h_{\theta^3}+h_\theta}{(h_{\theta\theta}+h)^2}\bigg)_{\theta}\bigg|
	 = \bigg|
		\frac{-(h_{\theta\theta}+h)(h_{\theta^4}+h_{\theta^2}) + 2(h_{\theta^3}+h_\theta)^2}
			{(h_{\theta\theta}+h)^3}
		\bigg|
\,,
\end{equation}
and $||h_{\theta^m}||_\infty \le C$ for $m\in\{0,1,2,3,4\}$.
As with earlier cases, we find
\begin{align}
\frac{d}{dt} \int h_{\theta^{6}}^2\,d\theta
&=
	2\int (k^{-1})_{\theta^{6}}k_{\theta^{6}}\,d\theta
\notag\\&\le
	-2c\int k_{\theta^{6}}^2\,d\theta
	+ C\int |k_{\theta^{6}}|
	  (|k_{\theta^5}|
	+ |k_{\theta^4}|
	+ |k_{\theta^3}|^2
	+ 1)
	\,d\theta
\notag\\&\le
	-\frac32c\int k_{\theta^{6}}^2\,d\theta
	+ C\int 
	  k_{\theta^5}^2
	+ k_{\theta^4}^2
	+ k_{\theta^3}^4
	\,d\theta
	+ C
\notag\\&\le
	-c\int k_{\theta^{6}}^2\,d\theta
	+ \int 
	  k_{\theta^3}^4
	\,d\theta
	+ C
\,.
\label{EQrefinep6}
\end{align}
We now apply the estimate \eqref{EQfollowestgen} followed by interpolation to refine \eqref{EQrefinep6} to
\begin{align}
\frac{d}{dt} \int h_{\theta^{6}}^2\,d\theta
&\le
	-\frac{c}{2}\int k_{\theta^{6}}^2\,d\theta
	+ C
\,.
\label{EQreturntwo6th}
\end{align}
Now we cut off in time.
The details are similar, so we will be brief.
We use \eqref{EQreturntwo5th} to calculate
\begin{equation}
\label{EQfromhere2p6}
\frac{d}{dt}\bigg((t-t_0/2)_+^2\int h_{\theta^{5}}^2\,d\theta\bigg)
\le 
	- \frac{c}2(t-t_0/2)_+^2\int k_{\theta^5}^2\,d\theta
	+ C(t-t_0/2)_+^2
	+ 2(t-t_0/2)_+\int h_{\theta^{5}}^2\,d\theta
\le
	C
\,,
\end{equation}
where we interpolated the last term on the RHS using an argument completely
analogous to the $p=5$ case, so we omit it.
Integrating \eqref{EQfromhere2p6} gives the estimate \eqref{EQsupporthigher}, finishing the proof for $p=6$.

{\bf Step 4. Induction for large $p$.}
Suppose \eqref{EQsupporthigher} holds for some $p\ge6$.
We aim to show that it holds for $p+1$, thus finishing the proof by induction.

First, \eqref{EQsupporthigher} implies $||h_{\theta^{p-1}}||_\infty \le C$.
It thus follows from $k = (h_{\theta\theta}+h)^{-1}$ and Proposition \ref{PN1} (see also \eqref{EQnewcontrol})
that
\begin{equation}
\label{EQkhighinfinity}
	||k_{\theta^{p-3}}||_\infty \le C
\,.
\end{equation}
Let us now note the following formula, which is the general version of \eqref{EQ6threcipderiv},
\begin{equation}
\label{EQgeneralrecipderiv}
(k^{-1})_{\theta^m}
 = 
	- k^{-2}k_{\theta^m}
	+ \sum_{q=2}^m k^{-(q+1)} \sum_{i_1+\ldots+i_q=m} c(i_1,\ldots,i_q) k_{\theta^{i_1}} \cdots k_{\theta^{i_q}}
\,.
\end{equation}
Above we implicitly assume each $i_j\in\N$.
The last term in the sum is $-k^{-m-1}k_{\theta}^m$, has $q=m$, $i_1=i_2=\ldots=1$ and $c(1,\ldots,1) = (-1)^m\,m!$, which is uniformly bounded.
In general the pattern we saw earlier will continue: many terms will be bounded from work in previous steps, and the remaining terms can be estimated.
Let us carry out the details.

For $p+1$ (which is at least 7) we compute
\begin{align*}
\frac{d}{dt} \int h_{\theta^{p+1}}^2\,d\theta
&=
	2\int (k^{-1})_{\theta^{p+1}}k_{\theta^{p+1}}\,d\theta
\,.
\end{align*}
Using \eqref{EQgeneralrecipderiv} and then \eqref{EQkhighinfinity}, we estimate
\begin{align*}
\frac{d}{dt} \int h_{\theta^{(p+1)}}^2\,d\theta
&\le
	-\int k^{-2}k_{\theta^{(p+1)}}^2\,d\theta
	+C\sum_{q=2}^{p+1}
		\sum_{i_1+\ldots+i_q={p+1}} c(i_1,\ldots,i_q)
	\int 
	 k_{\theta^{i_1}}^2 \cdots k_{\theta^{i_q}}^2
	 \,d\theta
\\
&\le
	-\int k^{-2}k_{\theta^{(p+1)}}^2\,d\theta
	+ C
	+ C\sum_{q=2}^{4}
		\sum_{i_1+\ldots+i_q={p+1}} c(i_1,\ldots,i_q)
	\int 
	 k_{\theta^{i_1}}^2 \cdots k_{\theta^{i_q}}^2
	 \,d\theta
\\
&\le
	-\int k^{-2}k_{\theta^{(p+1)}}^2\,d\theta
	+ C
\\
&\qquad
	+ C\sum_{i_1+i_2={p+1}} 
	 \int 
	 	k_{\theta^{i_1}}^2 k_{\theta^{i_2}}^2
		\,d\theta
	+ C\sum_{i_1+i_2+i_3={p+1}} 
	 \int 
	 	k_{\theta^{i_1}}^2 k_{\theta^{i_2}}^2 k_{\theta^{i_3}}^2
		\,d\theta
\\
&\qquad
	+ C\sum_{i_1+i_2+i_3+i_4={p+1}}
	 \int 
	 	k_{\theta^{i_1}}^2 k_{\theta^{i_2}}^2 k_{\theta^{i_3}}^2 k_{\theta^{i_4}}^2
		\,d\theta
\,.
\\
\intertext{In the summations above we use again \eqref{EQkhighinfinity} to estimate all factors of the form $k_{\theta^m}$ for $m\le p-3$ (note that $m$ is at least 3), and find}
\frac{d}{dt} \int h_{\theta^{(p+1)}}^2\,d\theta
&\le
	-\int k^{-2}k_{\theta^{(p+1)}}^2\,d\theta
	+ C
\\
&\qquad
	+ C\int 
		k_{\theta^{p}}^2 k_{\theta}^2 
	 	+ k_{\theta^{p-1}}^2 k_{\theta^2}^2  
	 	+ k_{\theta^{p-2}}^2 k_{\theta^3}^2 
		\,d\theta
\\
&\qquad
	+ C\int 
	 	k_{\theta^{p-1}}^2 k_{\theta}^2 k_{\theta}^2
	 	+ k_{\theta^{p-2}}^2 k_{\theta^2}^2 k_{\theta}^2
		\,d\theta
\\
&\qquad
	+ C\int 
	 	k_{\theta^{p-2}}^2 k_{\theta}^2 k_{\theta}^2 k_{\theta}^2
		\,d\theta
\\
&\le
	-\int k^{-2}k_{\theta^{(p+1)}}^2\,d\theta
	+C\int k_{\theta^{p}}^2 \,d\theta
	+C\int k_{\theta^{(p-1)}}^2 \,d\theta
	+C\int k_{\theta^{(p-2)}}^2 \,d\theta
	+C
\,.
\intertext{Lemma \ref{LMkfromabove}, integration by parts and interpolation yields}
\frac{d}{dt} \int h_{\theta^{(p+1)}}^2\,d\theta
&\le
	-c\int k_{\theta^{(p+1)}}^2\,d\theta
	+C\int k_{\theta^{p}}^2 \,d\theta
	+C
\le
	-c\int k_{\theta^{p+1}}^2\,d\theta
	+ C
\,.
\end{align*}
From here we calculate
\begin{equation}
\label{EQfromhere}
\frac{d}{dt}\bigg((t-t_0/2)_+^2\int h_{\theta^{p+1}}^2\,d\theta\bigg)
\le 
	- c(t-t_0/2)_+^2\int k_{\theta^p}^2\,d\theta
	+ C(t-t_0/2)_+^2
	+ 2(t-t_0/2)_+\int h_{\theta^{p+1}}^2\,d\theta
\,.
\end{equation}
In view of \eqref{EQinview}, \eqref{EQgeneralrecipderiv} and Proposition \ref{PN1}, we see that
\begin{align*}
||h_{\theta^{p+2}}||_2^2
 &\le
 	 2||h_{\theta^{p}}||_2^2
	+ 2||(k^{-1})_{\theta^p}||_2^2
\\
 &\le
 	  C(1 + ||k_{\theta^p}||_2^2 + ||k_{\theta^{p-1}}||_2^2 + ||k_{\theta^{p-2}}||_2^2)
\,.
\end{align*}
Then we estimate using interpolation (and choosing $\varepsilon$ in the last step)
\begin{align*}
	2(t-t_0/2)_+\int h_{\theta^{p+1}}^2\,d\theta
&\le
	\varepsilon(t-t_0/2)_+^2\int h_{\theta^{p+2}}^2\,d\theta
	+ C_\varepsilon\int h_{\theta^{p}}^2\,d\theta
\\
&\le
	\varepsilon(t-t_0/2)_+^2\int k_{\theta^{p}}^2+k_{\theta^{p-1}}^2+k_{\theta^{p-2}}^2\,d\theta
	+ C_\varepsilon\int h_{\theta^{p}}^2\,d\theta
\\
&\le
	\frac{c}{2} (t-t_0/2)_+^2\int k_{\theta^p}^2\,d\theta
	+ C
\end{align*}
We use this to absorb the rightmost term in \eqref{EQfromhere}, giving
\begin{equation*}
\frac{d}{dt}\bigg((t-t_0/2)_+^2\int h_{\theta^{p+1}}^2\,d\theta\bigg)
\le 
	C
\,.
\end{equation*}
Integration gives the estimate \eqref{EQsupporthigher}, finishing the proof.
\end{proof}

We will need to also employ the following remarkably strong $L^2(d\theta)$-uniqueness property of the flow.

\begin{prop}
\label{PNunique}
Let $\{h_n\}$ be a sequence of support functions of smooth solutions to the
entropy flow such that for every $t\in(0,t_1]$, $h_n(\cdot,t)$ converges to
some $h(\cdot,t)$ in the $L^2(d\theta)$-norm.
Suppose that $\{h_n(\cdot,0)\}$ converges in $L^2(d\theta)$ to $h_0$, where
$h_0$ is the support function of a curve satisfying the conditions of Theorem
\ref{TM2}.
Then $h(\cdot,t)$ converges to $h_0$ in $L^2(d\theta)$ as $t\searrow0$.
\end{prop}
\begin{proof}
Let $h^1 = h_n$ and $h^2 = h_m$ (we also use superscripts $1$ and $2$
throughout to refer to quantities corresponding to the curves generated by
$h^1$ and $h^2$ respectively), and calculate
\begin{align*}
\frac{d}{dt}\int |h^1 - h^2|^2\,d\theta
 &= 2\int (h^1-h^2)( (k^1)_{\theta\theta} + k^1 - (k^2)_{\theta\theta} - k^2 )\,d\theta
\\
 &= 2\int (h^1-h^2)( (k^1-k^2)_{\theta\theta} + k^1 - k^2)\,d\theta
\\
 &= 2\int ((h^1-h^2)_{\theta\theta} + h^1-h^2)(k^1 - k^2)\,d\theta
\\
 &= 2\int (1/k^1 - 1/k^2)(k^1 - k^2)\,d\theta
\\
 &= -2\int 
	\frac{(k^2-k^1)^2}{k^1k^2}
	\,d\theta
 \le 0\,.
\end{align*}
We therefore have
\[
\int |h^1(\theta,t) - h^2(\theta,t)|^2\,d\theta
	\le \int |h^1(\theta,0) - h^2(\theta,0)|^2\,d\theta
	\le
	  2\int |h^1(\theta,0) - h_0(\theta)|^2\,d\theta
	+ 2\int |h_0(\theta) - h^2(\theta,0)|^2\,d\theta
\,.
\]
Observing that $h_n(\cdot,0) \rightarrow h_0(\cdot)$ in $L^2(d\theta)$ by assumption finishes the proof.
\end{proof}

Now we are able to conclude the existence of a global solution with weak data.
This is the first half of Theorem \ref{TM2}.
We state it as follows.

\begin{thm}
\label{TMTM21}
Suppose $\gamma_0:\S^1\rightarrow\R^2$ is either
\begin{itemize}
\item[(I)] an immersed locally convex closed curve of class $C^2(ds)$ with turning number $\omega$; or
\item[(E)] an embedded curve of with $k\in L^2(ds)$ bounding a convex planar domain (which has $\omega=1$).
\end{itemize}
The entropy flow $\gamma:\S^1\times(0,T)\rightarrow\R^2$ with $\gamma_0$ as
initial data exists uniquely, is smooth, and global ($T=\infty$).
The flow attains its initial data in $C^2(ds)$ for case (I) and in $H^2(ds)$ for case (E).
\end{thm}
\begin{proof}
Consider the support function $h_0$ corresponding to $\gamma_0$ and take a
sequence $\{h_n^0\}$ of smooth functions such that $h_n^0 \rightarrow h_0$ in
$C^2(ds)$ for case (I), or $H^2(ds)$ for case (E).
Let $h_n$ be the corresponding smooth entropy flow with $h_n^0$ as initial
data, whose existence is guaranteed by Theorem \ref{TM1}.
Each flow $h_n$ exists globally by Corollary \ref{CYglobal}.

We have uniform estimates for all derivatives of $\{h_n\}$ over every compact
subset of $\S^1\times(0,\infty)$ by Proposition \ref{PNallest}.  Note that, for
any fixed compact subset, these estimates depend only on universal quantities:
$\vn{\gamma_0}_{H^2(ds)}$, $\vn{h_\theta}_{L^2(d\theta)}$ $\SE_0$, $\omega$.
Since the convergence of $h_n^0 \rightarrow h_0$ is in (at least)
$H^2(ds)$, these quantities are also uniformly bounded along the sequence
$\{h_n^0\}$.
By a diagonal subsequence argument, we find a sequence $\{h_{n_j}\}$ converging
smoothly in every compact subset of $\S^1\times(0,\infty)$ to a smooth function
$h:\S^1\times(0,\infty)\rightarrow\R$.
The smooth convergence implies that $h$ is the support function of an entropy
flow (satisfying $h_t = k_{\theta\theta} + k$).
By Proposition \ref{PNunique}, $h(\cdot,t)$ converges to $h_0$ in the
$L^2(d\theta)$ topology as $t\searrow0$.
Interpolation and our uniform estimates upgrade this convergence to the
regularity of the initial curve: either $C^2(ds)$ (locally convex immersion
(I)) or $H^2(ds)$ (convex embedding (E)).
\end{proof}


\section{Global analysis}
\label{SCglobal}

It remains to establish the second half of Theorem \ref{TM2}.
This will be completed by the end of Section \ref{SCrescaling}.
The entropy flow is expanding, and so to examine its asymptotic shape, one approach is to consider appropriate parabolic rescaling.
Setting $h^\lambda(\theta,t) = \lambda h(\theta, t/\lambda^2)$, we see that $k^\lambda(\theta,t) = \lambda^{-1}k(\theta, t/\lambda)$, and
\[
	h^\lambda_t(\theta,t)
	= \lambda^{-1} h_t(\theta, t/\lambda^2)
	= \lambda^{-1}(k_{\theta\theta}+k)(\theta, t/\lambda^2)
	= (k^\lambda_{\theta\theta}+k^\lambda)(\theta, t/\lambda^2)
\]
so $h^\lambda$ is again an entropy flow.
Note that the $\theta$-derivative is scale-invariant.

Now $\vn{h^\lambda}_2^2(t) = \lambda^2\vn{h}_2^2(t/\lambda^2)$,
and $\vn{h^\lambda_\theta}_2^2(t) = \lambda^2\vn{h_\theta}_2^2(t/\lambda^2)$.
Take a sequence of times $\{t_j\}\rightarrow \infty$.
Then $||h||_2^2(t_j) = ||h||_2^2(0) + 4\omega\pi\,t_j$.
Set $\lambda_j = (||h||_2^2(0) + 4\omega\pi\,t_j)^{-1/2}$ and consider the sequence of rescalings $h^{\lambda_j}$.
Then $||h^{\lambda_j}||_2^2(t_j) = 1$ and 
\[
\vn{h^{\lambda_j}_\theta}_2^2(t)
 = (||h||_2^2(0) + 4\omega\pi\,t_j)^{-1} \vn{h_\theta}_2^2(t/\lambda^2)
 \le (||h||_2^2(0) + 4\omega\pi\,t_j)^{-1} \vn{h_\theta}_2^2(0)
 \rightarrow 0
\]
for any $t\in[0,\infty)$. In particular, this holds for $t=t_j$ and suggests
that $h^{\lambda_j}(\cdot,t_j)$ converges to a circle (with support function
equal to $\frac{1}{\sqrt{2\omega\pi}}$) as $j\rightarrow\infty$.


Another classical approach is to use a continuous rescaling (as for instance used by Huisken \cite{huisken1984}).
This is what we do in our treatment here of the entropy flow.
Given a solution $\gamma$ to the entropy flow, we construct a rescaling $\eta$ by setting 
\[
\eta(\theta,t) = \frac{\gamma(\theta,t)}{\sqrt{L_0^2 + 8\omega^2\pi^2 t}}
 = \frac{\gamma(\theta,t)}{\phi(t)}
\,,
\]
where we have used $\phi(t) = \sqrt{L_0^2 + 8\omega^2\pi^2 t}$.

The rescaling $\eta$ is strictly convex with support function $h^\eta$ satisfying
\[
	h^\eta(\cdot,t) = \frac1{\phi(t)}\, h^\gamma(\cdot, t)\,,
\]
and curvature satisfying
\[
	k^\eta(\cdot,t) = \phi(t)k^\gamma(\cdot,t)\,.
\]
We can calculate
\begin{align*}
	\partial_th^\eta(\cdot,t)
	&= 
	-\frac{\phi'(t)}{\phi^2(t)}\, h^\gamma(\cdot, t)
	 + \frac1{\phi(t)}\, \partial_th^\gamma(\cdot, t)
	\\&= 
	-\frac{4\omega^2\pi^2}{(L_0^2 + 8\omega^2\pi^2 t)^{\frac32}}\, h^\gamma(\cdot, t)
	 + \frac1{\phi(t)}\, (k^\gamma_{\theta\theta}+k^\gamma)(\cdot, t)
	\\&= 
	-\frac{4\omega^2\pi^2}{L_0^2 + 8\omega^2\pi^2 t}\, h^\eta(\cdot, t)
	 + \frac1{L_0^2 + 8\omega^2\pi^2 t}\, (k^\eta_{\theta\theta}+k^\eta)(\cdot, t)
	\\&=
	\frac1{L_0^2 + 8\omega^2\pi^2 t} 
	\bigg(
	  k^\eta_{\theta\theta}+k^\eta
	- 4\omega^2\pi^2\, h^\eta
	\bigg)(\cdot, t)
	\,.
\end{align*}
Then, reparametrise time with a new variable $t^\eta$ defined by
\[
	\partial_{t^\eta} = \phi^2 \partial_t
	\,,
\]
so that by the chain rule $h^\eta$ satisfies a new rescaled flow equation.
In particular we find
\[
	h^\eta_{t^\eta} = k^\eta_{\theta\theta} + k^\eta - h^\eta
	\,.
\]
Our eventual goal will be to prove that $h^\eta$ converges smoothly exponentially fast to a round circle.
In this section, we prove the remaining estimates needed on the un-scaled flow, whereas in Section 6 we study directly the rescaled flow.

First, let us use the global existence established above to show that eventually a certain scale-invariant quantity is small.

\begin{lem}
\label{LMeventsmall}
Consider an entropy flow $\gamma:\S^1\times(0,\infty)\rightarrow\R^2$ with initial data $\gamma_0$ satisfying the conditions of Theorem \ref{TM2}.
For any $\varepsilon>0$ there exists a $t^1_\varepsilon\in(0,\infty)$ such that
\[
	\int k^{-2}k_\theta^2\,d\theta\bigg|_{t=t^1_\varepsilon} \le \varepsilon\,.
\]
\end{lem}
\begin{proof}
Proposition \ref{PNhbase} implies
\[
	\int_0^t\int k^{-2}k_\theta^2\,d\theta\,d\hat t = \vn{h_\theta}_2^2(0) - \vn{h_\theta}_2^2(t^1_\varepsilon)
	                                               \le\vn{h_\theta}_2^2(0)
\,.
\]
Note that $\vn{h_\theta}_2^2(0)$ exists by Lemma \ref{LMpwsupport}.
Taking $t\rightarrow\infty$ yields that $\int k^{-2}k_\theta^2\,d\theta\in
L^1((0,\infty))$, which implies the result.
\end{proof}

Now, we need to show that sufficient eventual pointwise in time smallness of $\vn{k^{-1}k_\theta}_2^2$ is preserved.

\begin{prop}
\label{PNlogkthetsmallforever}
Consider an entropy flow $\gamma:\S^1\times(0,\infty)\rightarrow\R^2$ with initial data $\gamma_0$ satisfying the conditions of Theorem \ref{TM2}.
For any $\varepsilon\in(0,1/108]$, we have
\begin{equation}
\label{EQpressmall}
	\int k^{-2}k_\theta^2\,d\theta\bigg|_{t} \le \varepsilon\,,\quad\text{ for all }t\ge t^1_\varepsilon\,.
\end{equation}
\end{prop}
\begin{proof}
First, let us calculate (recall \eqref{EQcurvatureevo})
\begin{align*}
\frac{d}{dt} \int k_\theta^2k^{-2}\,d\theta
&= 
	2\int k_\theta k^{-2}(-k^2(F_{\theta\theta}+F))_\theta\,d\theta
	-2\int k_\theta^2k^{-3}(-k^2(F_{\theta\theta}+F))\,d\theta
\\&= 
	-2\int k_\theta F_{\theta^3}\,d\theta
	-2\int k_\theta F_\theta\,d\theta
	-4\int k_\theta^2 k^{-1}(F_{\theta\theta}+F)\,d\theta
\\&\qquad
	+2\int k_\theta^2k^{-1}(F_{\theta\theta}+F)\,d\theta
\\&= 
	-2\int F_{\theta}^2\,d\theta
	-2\int k_\theta^2 k^{-1}(F_{\theta\theta}+F)\,d\theta
\,.
\end{align*}
Now, integrating by parts and simplifying, we find
\begin{align*}
\frac{d}{dt} \int k_\theta^2k^{-2}\,d\theta
&= 
	-2\int F_{\theta}^2\,d\theta
	-2\int k_\theta^2 k^{-1}(F_{\theta\theta}+F)\,d\theta
\\
&= 
	-2\int F_{\theta}^2\,d\theta
	+2\int (k_\theta^2 k^{-1})_\theta F_{\theta}\,d\theta
	-2\int k_\theta^2 k^{-1}F\,d\theta
\\
&= 
	-2\int F_{\theta}^2\,d\theta
	+2\int (2k_{\theta\theta} k_\theta k^{-1} - k_\theta^3 k^{-2}) F_{\theta}\,d\theta
\\&\quad
	-2\int k_{\theta\theta}k_\theta^2 k^{-1}\,d\theta
	-2\int k_\theta^2\,d\theta
%
%
\\
&= 
	-2\int F_{\theta}^2\,d\theta
	+4\int k_{\theta\theta} k_\theta k^{-1} F_{\theta}\,d\theta
	-2\int k_\theta^3 k^{-2} k_{\theta^3}\,d\theta
\\&\quad
	-\frac83\int k_{\theta}^4k^{-2}\,d\theta
	-2\int k_\theta^2\,d\theta
\\
&= 
	-2\int F_{\theta}^2\,d\theta
	+4\int k_{\theta\theta} k_\theta k^{-1} F_{\theta}\,d\theta
\\&\quad
	+2\int (3k_{\theta\theta} k_\theta^2 k^{-2} -2 k_\theta^4 k^{-3})
		 k_{\theta\theta}\,d\theta
\\&\quad
	-\frac83\int k_{\theta}^4k^{-2}\,d\theta
	-2\int k_\theta^2\,d\theta
\\
&= 
	-2\int k_{\theta^3}^2\,d\theta
	-2\int k_\theta^2\,d\theta
	+4\int k_{\theta\theta}^2\,d\theta
\\&\quad
	- 2\int k_{\theta\theta}^2 ( k_{\theta\theta} k^{-1} - k_\theta^2 k^{-2} )
		\,d\theta
	+ \frac43\int k_\theta^4 k^{-2}\,d\theta
\\&\quad
	+ 6\int k_{\theta\theta}^2 k_\theta^2 k^{-2} 
		 \,d\theta
	- \frac{12}{5}\int k_\theta^6 k^{-4}
		 \,d\theta
\\&\quad
	-\frac83\int k_{\theta}^4k^{-2}\,d\theta
	-2\int k_\theta^2\,d\theta
\\
&= 
	-2\int k_{\theta^3}^2\,d\theta
	-4\int k_\theta^2\,d\theta
	-4\int k_{\theta^3}k_{\theta}\,d\theta
	- 2\int k_{\theta\theta}^3 k^{-1}
		\,d\theta
\\&\quad
	+ 8\int k_{\theta\theta}^2 k_\theta^2 k^{-2} 
		 \,d\theta
	- \frac{12}{5}\int k_\theta^6 k^{-4}
		 \,d\theta
	- \frac43\int k_{\theta}^4k^{-2}\,d\theta
\,.
\end{align*}
We need to estimate the third and fourth terms.
Observe that integration by parts implies
\begin{align*}
- 2\int k_{\theta\theta}^3 k^{-1}
		\,d\theta
&= 
  - 2\int k_{\theta}^2 k_{\theta\theta}^2 k^{-2}
		\,d\theta
  + 4\int k_{\theta} k_{\theta\theta} k_{\theta^3} k^{-1}
		\,d\theta
\\
&\le
	\delta\int k_{\theta^3}^2\,d\theta
      + \Big(- 2 + \frac{4}{\delta}\Big)
	\int k_{\theta}^2 k_{\theta\theta}^2 k^{-2}
		\,d\theta
\,.
\end{align*}
This estimates the fourth term.
For the third term, we simply use $ab \le \delta a^2 + \frac1{4\delta}b^2$.
All together we have
\begin{align*}
\frac{d}{dt} \int k_\theta^2k^{-2}\,d\theta
&\le 
	(5\delta-2)\int k_{\theta^3}^2\,d\theta
	- (4-1/\delta)\int k_\theta^2\,d\theta
\\&\quad
	+ (6 + 4/\delta)\int k_{\theta\theta}^2 k_\theta^2 k^{-2} \,d\theta
	- \frac{12}{5}\int k_\theta^6 k^{-4} \,d\theta
	- \frac43\int k_{\theta}^4k^{-2} \,d\theta
\,.
\end{align*}
Take $\delta = 1/3$, then, using also the Poincar\'e and H\"older inequalities, we find
\begin{align*}
\frac{d}{dt} \int k_\theta^2k^{-2}\,d\theta
&\le 
	-\frac16\int k_{\theta^3}^2\,d\theta
	-\int k_\theta^2\,d\theta
	+ 18\int k_{\theta\theta}^2 k_\theta^2 k^{-2} \,d\theta
\\&\quad
	- \frac{12}{5}\int k_\theta^6 k^{-4} \,d\theta
	- \frac43\int k_{\theta}^4k^{-2} \,d\theta
\\&\le
	-\bigg(\frac16 - 18\int k_\theta^2 k^{-2}\,d\theta\bigg)\int k_{\theta^3}^2\,d\theta
	-\int k_\theta^2\,d\theta
\\&\quad
	- \frac{12}{5}\int k_\theta^6 k^{-4} \,d\theta
	- \frac43\int k_{\theta}^4k^{-2} \,d\theta
\,.
\end{align*}
Now, Lemma \ref{LMeventsmall} implies that 
$\vn{k^{-1}k_\theta}_2^2(t^1_\varepsilon) \le \varepsilon$.
Recall that $\varepsilon \le 1/108$, so that at $t=t^1_\varepsilon$ the coefficient of the highest order term is non-positive.
Therefore
\begin{equation}
\label{EQestforlogktheta}
\frac{d}{dt} \int k_\theta^2k^{-2}\,d\theta\bigg|_{t=t^1_\varepsilon}
\le
	-\int k_\theta^2\,d\theta
	- \frac{12}{5}\int k_\theta^6 k^{-4} \,d\theta
	- \frac43\int k_{\theta}^4k^{-2} \,d\theta
\le 0
\,.
\end{equation}
This preserves the smallness condition, yielding \eqref{EQestforlogktheta} for every $t\ge t^1_\varepsilon$; therefore we have \eqref{EQpressmall} and we are done.
\end{proof}

We can now dramatically upgrade our preservation of convexity (Proposition \ref{PN1}) to an estimate that is as strong as possible (equality for circles).
At the same time we note some long-time asymptotic information on the curvature.

\begin{prop}
\label{PNunifpresconv}
Consider an entropy flow $\gamma:\S^1\times(0,\infty)\rightarrow\R^2$ with initial data $\gamma_0$ satisfying the conditions of Theorem \ref{TM2}.
Let $t_0>0$.
There exists a $c_0>0$ depending only on $t_0$,
$\SE_0,L_0,\omega,\vn{k}_{L^1(d\theta)}(0))$ such that
\[
	k(\theta,t)L(t) \ge c_0\,,\qquad\text{for all $t>t_0$}.
\]
Furthermore, $||k^{-1}k_\theta||_2^2(t)\rightarrow0$ and $\log k$ converges to its average in the topology $C^0(\S^1)$ as $t\nearrow\infty$.
\end{prop}
\begin{proof}
As we already know the flow is strictly convex, it remains to improve the estimate of Proposition \ref{PN1} outisde a compact time interval.
Applying Proposition \ref{PNlogkthetsmallforever} with $\varepsilon = \frac{1}{108}$ gives that
\begin{equation}
\label{EQkkts}
	 ||k^{-1}k_\theta||_2^2(t) \le \frac1{108}
\end{equation}
for $t\ge t^1_{1/108}$.
Let us set $t^1 = t^1_{1/108}$.
Then
\[
	\log k - \overline{\log k} \ge -\int k^{-1}|k_\theta|\,d\theta \ge -\sqrt{2\omega\pi} \bigg(\int k^{-2}k_\theta^2\,d\theta\bigg)^{\frac12}\,.
\]
Using \eqref{EQkkts} and the lower bound for the entropy (Corollary \ref{CYenbel}) yields, for $t\ge t^1$,
\[
	\log k \ge - \sqrt{\frac{\omega\pi}{54}}
	+ \log \frac{2\omega\pi}{L_0 + 4\omega^2\pi^2c_1^{-1}\big( \sqrt{4\omega^2\pi^2 + tc_1^2} - 2\omega\pi \big)}
\]
or
\[
	k \ge \frac{2\omega\pi}{L_0 + 4\omega^2\pi^2c_1^{-1}\big( \sqrt{4\omega^2\pi^2 + tc_1^2} - 2\omega\pi \big)}\exp\bigg(- \sqrt{\frac{\omega\pi}{54}}\bigg)
\qquad \text{for $t\ge t^1$}.
\]
Using Lemma \ref{LM3} first and then the estimate above yields
\begin{align}
	kL
	&\ge 
		k\sqrt{L_0^2 + 8\omega^2\pi^2t}
\notag\\
	&\ge
		\frac{2\omega\pi\sqrt{L_0^2 + 8\omega^2\pi^2t}}
		{L_0 + \sqrt2\,\omega\pi\big( \sqrt{32\omega^4\pi^4c_1^{-2} + 8\omega^2\pi^2t} - 2\omega\pi \big)}\exp\bigg(- \sqrt{\frac{\omega\pi}{54}}\bigg)
\label{EQklest}
\end{align}
which is uniformly bounded from below for $t\in [t^1,\infty)$.
Furthermore, Proposition \ref{PN1} implies that $kL$ is uniforomly bounded from
below on $[t_0,t^1]$, by a constant depending only on $t_0$, $t^1$ and
$\SE_0,L_0,\omega,\vn{k}_{L^1(d\theta)}(0))$.
This proves the first claim.



%
%
%
%
Then, using the bound from below on $k$ in \eqref{EQestforlogktheta} we find
\begin{equation*}
\frac{d}{dt} \int k_\theta^2k^{-2}\,d\theta
\le
	-\int k_\theta^2\,d\theta
\le
	-c_0^{2}L^{-2}\int k^{-2}k_\theta^2\,d\theta
\,.
\end{equation*}
Observe that
\begin{align*}
	-\int_{t^1}^t L^{-2}\,dt
	&\le - \int_{t^1}^t \frac{1}{\bigg(L_0 + 4\omega\pi c_1^{-1}\bigg( \sqrt{4\omega^2\pi^2 + tc_1^2} - 2\omega\pi \bigg)\bigg)^2}\,dt
	\\&= -\frac1{2\omega\pi}
		\log\bigg(\frac{L_0-8\omega^2\pi^2c_1^{-1} + 4\omega\pi\sqrt{4\omega^2\pi^2c_1^{-2}+t}}
				{L_0-8\omega^2\pi^2c_1^{-1} + 4\omega\pi\sqrt{4\omega^2\pi^2c_1^{-2}+t^1}}\bigg)
\\&\quad
	- \frac{L_0-8\omega^2\pi^2c_1^{-1}}{2\omega\pi}
		\bigg(
			\frac{1}{L_0-8\omega^2\pi^2c_1^{-1} + 4\omega\pi\sqrt{4\omega^2\pi^2c_1^{-2}+t}}
\\&\qquad\qquad\qquad\qquad\qquad\qquad
			- \frac{1}{L_0-8\omega^2\pi^2c_1^{-1} + 4\omega\pi\sqrt{4\omega^2\pi^2c_1^{-2}+t^1}}
		\bigg)
	\\&\le -\frac1{2\omega\pi}
		\log\Big(L_0-8\omega^2\pi^2c_1^{-1} + 4\omega\pi\sqrt{4\omega^2\pi^2c_1^{-2}+t}\Big)
	+ C
	\,,
\end{align*}
so
\begin{align*}
\int k^{-2}k_\theta^2\,d\theta
&\le
	C||k^{-1}k_\theta||_2^2(0)\exp\Big(
	       -\frac{c_0^2}{2\omega\pi}
		\log\Big(L_0-8\omega^2\pi^2c_1^{-1} + 4\omega\pi\sqrt{4\omega^2\pi^2c_1^{-2}+t}\Big)
\\&\le
	C||k^{-1}k_\theta||_2^2(0)
		\Big(L_0-8\omega^2\pi^2c_1^{-1} + 4\omega\pi\sqrt{4\omega^2\pi^2c_1^{-2}+t}\Big)^{-\frac{c_0^2}{2\omega\pi}}
\,.
\end{align*}
This implies that $||k^{-1}k_\theta||_2^2(t)\rightarrow0$ (second claim), and so,
\[
	|\log k - \overline{\log k}|(\theta,t) \le \sqrt{2\omega\pi}\,||k^{-1}k_\theta||_2(t) \rightarrow 0
	\,,
\]
which is the third claim and finishes the proof.
\end{proof}

Note that the curvature of growing circles decreases under the entropy flow with asymptotic rate $1/\sqrt{t}$, with $kL$ constant along the flow.
Proposition \ref{PNunifpresconv} is strong enough that along the continuous rescaling the estimate for curvature from below will be uniform.


\section{The rescaled flow}
\label{SCrescaling}

We briefly recall the continuous rescaling $\eta$ from the introduction of
Section \ref{SCglobal}.
Given an entropy flow $\gamma$ satisfying the hypotheses of Theorem \ref{TM1},
we rescale by setting $\eta(\theta,t) = \gamma(\theta,t)/\phi(t)$ where
where $\phi(t) = \sqrt{L_0^2 + 8\omega^2\pi^2 t}$.
Then, reparametrise time with a new variable $t^\eta$ defined by
$\partial_{t^\eta} = \phi^2 \partial_t$.

The rescaled flow equation is
\[
	h^\eta_{t^\eta} = k^\eta_{\theta\theta} + k^\eta - h^\eta
	\,.
\]
From now until the end of this section, we \emph{drop the $\eta$ superscript}.
Let us record some immediate facts about the rescaled flow from our previous
analysis.

\begin{cor}
\label{CYrescaledfacts}
Consider the rescaling of an entropy flow generated by Theorem \ref{TMTM21}.
The rescaled flow exists globally, and:
\begin{itemize}
\item Length is uniformly bounded, satisfying
\begin{equation}
\label{EQlengthrescaledstatement}
1 \le L(t) \le c_L
\end{equation}
where $c_L = c_L(\omega,L_0)$;
\item Curvature is uniformly bounded from below, satisfying
\[
	k(\theta,t)\ge c_0c_L^{-1}
\]
where $c_0$ is defined in Proposition \ref{PNunifpresconv};
\item We have $||k^{-1}k_\theta||_2^2(t)\rightarrow0$ and $\log k$ converges to
its average in the topology $C^0(\S^1)$ as $t\nearrow\infty$.
\end{itemize}
\end{cor}
\begin{proof}
For length, the following estimate follows immediately from
Lemma \ref{LM3}:
\begin{equation}
\label{EQlengthrescaled}
1 \le L(t) \le \frac{L_0 - 8\omega^2\pi^2c_1^{-1}}{\sqrt{L_0^2 + 8\omega\pi t}}
 + 4\omega\pi\sqrt{\frac{4\omega^2\pi^2c_1^{-2} + t}{L_0^2 + 8\omega^2\pi^2 t}}\,.
\end{equation}
To see that this bound is uniform, note that the upper bound has limits both as
$t\searrow0$ and as $t\nearrow\infty$.

The estimate on curvature from below follows from Proposition \ref{PNunifpresconv}.
This is because $kL$ is scale-invariant, so the main estimate from Proposition \ref{PNunifpresconv} yields $kL \ge c_0L^{-1} \ge c_0c_L^{-1}$.
The quantity $||k^{-1}k_\theta||_2^2$ is also scale-invariant, so its decay to zero also follows from Proposition \ref{PNunifpresconv}.
With this fact in hand, we can repeat the final steps of the proof of
Proposition \ref{PNunifpresconv} but for the rescaled flow to conclude that
$\log k$ converges to its average (which is uniformly bounded along the
rescaled flow and thus convergent).
\end{proof}


We have identification of the limit of the rescaling as a standard round $\omega$-circle.
Our final task is to establish convergence in the smooth topology.
For this we will focus on the rescaled support function.

\begin{prop}
\label{PNrescexpdecay}
Consider the rescaling of an entropy flow generated by Theorem \ref{TMTM21}.
Then $\vn{h}_2^2(t) \rightarrow 2\omega\pi$, and
\[
	\vn{h_\theta}_2^2(t) \le ||h_\theta||_2^2(0) e^{-2t}\,.
\]
\end{prop}
\begin{proof}
First, compute
\begin{align*}
\frac{d}{dt} \int h^2\,d\theta
	&= 2\int h(k_{\theta\theta} + k - h)\,d\theta
\\
	&= 2\int 1 - h^2\,d\theta = 4\omega\pi - 2\int h^2\,d\theta
\,,
\end{align*}
so
\[
	e^{2t} ||h||_2^2(t) - ||h||_2^2(0) = 2\omega\pi (e^{2t} - 1)
\]
which implies
\[
	 ||h||_2^2(t) = 2\omega\pi + (||h||_2^2(0) + 2\omega\pi)e^{-2t} 
\,.
\]
This equation (note that $\vn{h}_2^2$ is monotone for the original flow) implies that
$\vn{h}_2^2 \rightarrow 2\omega\pi$ as $t\rightarrow\infty$.

For $\vn{h_\theta}_2^2$, we calculate
\begin{align}
\frac{d}{dt} \int h_\theta^2\,d\theta
	&= -2\int h_{\theta\theta}(k_{\theta\theta} + k - h)\,d\theta
	 = -2\int (1/k)_{\theta\theta}k\,d\theta
	   -2\int h_{\theta}^2\,d\theta
\notag\\
	&= -2\int k_\theta^2/k^2\,d\theta
	   -2\int h_{\theta}^2\,d\theta
	\le
		-2\int h_{\theta}^2\,d\theta
\,.
\label{EQht}
\end{align}
Therefore
\begin{equation}
\label{EQhtexpdec}
	\vn{h_\theta}_2^2(t) \le ||h_\theta||_2^2(0) e^{-2t}
\,,
\end{equation}
for $t\ge 0$.
\end{proof}

Now, as is standard, in order to obtain exponential decay for quantities of the
form $||h_{\theta^p}||_2^2$ it is enough to show that they are uniformly
bounded and then apply interpolation with \eqref{EQhtexpdec}.

\begin{prop}
\label{PNrescallest}
Consider the rescaling of an entropy flow generated by Theorem \ref{TMTM21}.
There exist $t_p\ge0$ such that for all $p\in\N$ we have
\[
	\vn{h_{\theta^p}}_2^2(t) \le C\,,\quad \text{for} \quad t\ge t_p
	\,.
\]
\end{prop}
\begin{proof}
Many of the estimates and calculations in this proof are in a sense rescaled analogies of the proof of Proposition \ref{PNallest}.
Proposition \ref{PNrescexpdecay} covers the case $p=1$.

{\bf Case $p=2$.} 
We calculate
\begin{align}
\frac{d}{dt} \int h_{\theta\theta}^2\,d\theta
	&= 2\int h_{\theta^4}(k_{\theta\theta} + k - h)\,d\theta
\notag\\
	&= 2\int (1/k)_{\theta^4}k\,d\theta
	   -2\int h_{\theta\theta}^2\,d\theta
\notag\\
	&= 2\int (1/k)_{\theta\theta}k_{\theta\theta}\,d\theta
	   -2\int h_{\theta\theta}^2\,d\theta
\notag\\
	&= -2\int (k_{\theta\theta}k^{-2} - 2k_\theta^2k^{-3})k_{\theta\theta}\,d\theta
	   -2\int h_{\theta\theta}^2\,d\theta
\notag\\
	&= -2\int k_{\theta\theta}^2k^{-2}\,d\theta
	   +4\int k_\theta^4k^{-4}\,d\theta
	   -2\int h_{\theta\theta}^2\,d\theta
\label{EQhtt}
\,.
\end{align}
Corollary \ref{CYrescaledfacts} implies that $\log k$
is convergent to its average.
The average of $\log k$ is $\frac1{2\omega\pi}\SE(t)$.
Applying the proof of Lemma \ref{LM4} gives a uniform bound for $\SE$ from below, due to the uniform length bound from Corollary \ref{CYrescaledfacts}.
For the bound from above, note that at each $t$ we have (using $k_0(t) = \inf k(\cdot,t) \ge c_0c_L^{-1}$)
\[
	|\log k - \log k_0| \le \sqrt{2\omega\pi}\sqrt{1/108}
\]
which implies a uniform bound from above for $k$ and therefore also a uniform bound from above for $\SE$.
Thus the average of $\log k$ is bounded, and $\log k$ converges to it.
Therefore $k$ is convergent to a constant (also its average).
Then
\begin{align*}
   -\int k_{\theta\theta}^2k^{-2}\,d\theta
   +2\int k_\theta^4k^{-4}\,d\theta
	\le 
		- c||k_{\theta\theta}||_2^2
		+ C||k^{-1}k_\theta||_2^2 
\,.
\end{align*}
Using \eqref{EQht} we have
\[
   \int_0^\infty \int k_\theta^4k^{-4}\,d\theta\,dt
	\le C
\]
so, integrating \eqref{EQhtt} we find
\begin{align*}
 ||h_{\theta\theta}||_2^2(t)
	\le C
\,.
\end{align*}
This holds for all $t\ge0$, so we can pick $t_2=t_1=0$.
%

{\bf Case $p=3$.} 
For the next case, we calculate
\begin{align}
\frac{d}{dt} \int h_{\theta^3}^2\,d\theta
	&= 2\int (1/k)_{\theta^3}k_{\theta^3}\,d\theta
	   -2\int h_{\theta^3}^2\,d\theta
\notag\\
	&\le -c\int k_{\theta^3}^2\,d\theta
		+ C\int k_{\theta\theta}^2k_\theta^2 + k_\theta^6\,d\theta
		- 2\int h_{\theta^3}^2\,d\theta
\,.
\label{EQh3main}
\end{align}
We find
\begin{align*}
\int k_{\theta\theta}^2k_\theta^2\,d\theta
	&= -\int k_{\theta^3}k_\theta^3
	         + 2k_{\theta\theta}^2k_\theta^2\,d\theta
\end{align*}
which implies, after factorising and interpolating,
\begin{equation}
\label{EQh31}
\int k_{\theta\theta}^2k_\theta^2\,d\theta
		\le \delta\int k_{\theta^3}^2\,d\theta
			+ C(\delta)\int k_\theta^6\,d\theta
\,.
\end{equation}
Observe $k_\theta = (k-\overline{k})_\theta$ where $\overline{k}$ is the average of $k$.
Then, integrating by parts, interpolating, and using the convergence of $k$ to its average, we find
\begin{align*}
\int k_\theta^6\,d\theta
	&= -5\int k_{\theta\theta}k_{\theta}^4(k-\overline{k})\,d\theta
	\\&\le
		\delta \int k_{\theta\theta}^2k_{\theta}^2\,d\theta
		+ C(\delta) \int k_{\theta}^6(k-\overline{k})^2\,d\theta
	\\&\le
		\delta\int k_{\theta^3}^2\,d\theta
	        + \frac12\int k_\theta^6\,d\theta
\end{align*}
for $t>t_3$.
This is the definition of $t_3$.

Thus
\begin{equation}
\label{EQh32}
\int k_\theta^6\,d\theta
	\le
		\delta\int k_{\theta^3}^2\,d\theta
\end{equation}
for $t>t_3$.
Combining \eqref{EQh31}, \eqref{EQh32} and \eqref{EQh3main} yields
\begin{align*}
\frac{d}{dt} \int h_{\theta^3}^2\,d\theta
	&\le 0
\end{align*}
for $t>t_3$, which implies
$||h_{\theta^3}||_2^2(t) \le 
||h_{\theta^3}||_2^2(t_3)$.

{\bf Case $p=4$.} 
For $||h_{\theta^4}||_2^2$ we calculate
\begin{align}
\frac{d}{dt} \int h_{\theta^4}^2\,d\theta
	&= 2\int (1/k)_{\theta^4}k_{\theta^4}\,d\theta
	   -2\int h_{\theta^4}^2\,d\theta
\notag\\
	&\le -c\int k_{\theta^4}^2\,d\theta
		+ C\int 
			k_{\theta^3}^2k_\theta^2
			+ k_{\theta\theta}^2k_\theta^4
			+ k_{\theta\theta}^4
			+ k_{\theta}^8
			\,d\theta
	   -2\int h_{\theta^4}^2\,d\theta
\notag\\
	&\le -c\int k_{\theta^4}^2\,d\theta
		+ C\int 
			k_{\theta^3}^2k_\theta^2
			+ k_{\theta\theta}^4
			+ k_{\theta}^8
			\,d\theta
	   -2\int h_{\theta^4}^2\,d\theta
\,.
\label{EQcase4}
\end{align}


First, use
\begin{equation}
\int 
	k_{\theta^3}^2k_\theta^2
	\,d\theta
	\le
		\delta\int
			k_{\theta^4}^2
			+ k_{\theta\theta}^4\,d\theta
		+ C(\delta)\int
			k_\theta^8\,d\theta
\,,
\label{EQest1}
\end{equation}
where $\delta>0$ will be chosen.

To prove \eqref{EQest1}, first calculate
\begin{align*}
\int 
	k_{\theta^3}^2k_\theta^2
	\,d\theta
	&\le
		C\int
			|k_{\theta^4}|\,|k_{\theta\theta}|\,k_\theta^2
			+ |k_{\theta^3}|\,k_{\theta\theta}^2\,|k_\theta|
			\,d\theta
\\	&\le
		\delta\int
			k_{\theta^4}^2\,d\theta
		+ C(\delta)\int
			k_{\theta\theta}^2k_\theta^4\,d\theta
		+ \frac12\int k_{\theta^3}^2k_\theta^2\,d\theta
		+ C\int
			k_{\theta\theta}^4
			\,d\theta
\,.
\end{align*}
Then absorb to find
\begin{align}
\int 
	k_{\theta^3}^2k_\theta^2
	\,d\theta
	&\le
		\delta\int
			k_{\theta^4}^2\,d\theta
		+ C(\delta)\int
			k_{\theta\theta}^2k_\theta^4\,d\theta
		+ C\int
			k_{\theta\theta}^4
			\,d\theta
\notag\\	&\le
		\delta\int
			k_{\theta^4}^2
			+ k_{\theta\theta}^4\,d\theta
		+ C(\delta)\int
			k_\theta^8\,d\theta
\,,
\label{EQest2}
\end{align}
as required.

For the remaining terms, let us note two general estimates.
The first is
\begin{align}
\int
	k_{\theta}^{2q}
	\,d\theta
	&\le
		C(q)\int |k_{\theta\theta}|\,|k_\theta|^{2q-2}\,|k-\overline{k}|\,d\theta
\notag\\	&\le
		\delta\int k_{\theta\theta}^2\,k_\theta^{2q-4}\,d\theta
		+ C(\delta)||k-\overline{k}||_\infty \int k_\theta^{2q}\,\,d\theta
\,.
\label{EQestgen1}
\end{align}
The second requires two steps.
Start by estimating
\begin{align*}
\int 
	k_{\theta^m}^{2q}
	\,d\theta
	&\le
		C\int |k_{\theta^{m+1}}|\,k_{\theta^m}^{2q-2}\,|k_{\theta^{m-1}}|\,d\theta
\\	&\le
		\frac{1}{2}\int 
			k_{\theta^m}^{2q}
			\,d\theta
		+ C\int k_{\theta^{m+1}}^2k_{\theta^m}^{2q-4}k_{\theta^{m-1}}^2\,d\theta
\end{align*}
then absorb to find
\begin{align}
\int 
	k_{\theta^m}^{2q}
	\,d\theta
	&\le
		C\int k_{\theta^{m+1}}^2k_{\theta^m}^{2q-4}k_{\theta^{m-1}}^2\,d\theta
\label{EQestgen2}
\end{align}
Now applying \eqref{EQestgen1} with $q=4$ and then interpolating once more yields
\begin{equation*}
\int
	k_{\theta}^{8}
	\,d\theta
\le
		\delta\int k_{\theta\theta}^2\,k_\theta^{4}\,d\theta
		+ C(\delta)||k-\overline{k}||_\infty \int k_\theta^{8}\,\,d\theta
\le
		\frac{\delta^2}{2}\int k_{\theta\theta}^4\,d\theta
		+ (\frac12+C(\delta)||k-\overline{k}||_\infty) \int k_\theta^{8}\,\,d\theta
\end{equation*}
which is
\begin{equation}
\label{EQest3}
\int
	k_{\theta}^{8}
	\,d\theta
\le
		\delta^2\int k_{\theta\theta}^4\,d\theta
		+ C(\delta)||k-\overline{k}||_\infty \int k_\theta^{8}\,\,d\theta
\end{equation}
after absorption.

Applying \eqref{EQestgen2} with $q=2$ and $m=2$, then \eqref{EQest2}, then \eqref{EQest3} (using $\hat\delta$ there for clarity) gives
\begin{align}
\int
	k_{\theta\theta}^{4}
\,d\theta
	&\le
		C\int k_{\theta^3}^2k_\theta^2\,d\theta
\notag\\
	&\le
		\delta\int
			k_{\theta^4}^2
			+ k_{\theta\theta}^4\,d\theta
		+ C(\delta)\int
			k_\theta^8\,d\theta
\notag\\
	&\le
		(\delta+C(\delta)\hat\delta^2)\int
			k_{\theta^4}^2
			+ k_{\theta\theta}^4\,d\theta
		+ C(\delta)C(\hat\delta)||k-\overline{k}||_\infty \int k_\theta^{8}\,\,d\theta
\,.
\label{EQest4}
\end{align}


Add the estimates \eqref{EQest2}, \eqref{EQest3}, \eqref{EQest4} together to find (also use \eqref{EQest3} once more)
\begin{align*}
C\int 
	k_{\theta^3}^2k_\theta^2
	+ k_{\theta\theta}^4
	+ k_{\theta}^8
	\,d\theta
	\le
		(2\delta+\delta^2+C(\delta)\hat\delta^2)\int
			k_{\theta^4}^2
			+ k_{\theta\theta}^4\,d\theta
		+ C(\delta)C(\hat\delta)||k-\overline{k}||_\infty \int k_\theta^{8}\,\,d\theta
\,.
\end{align*}
Choose $\delta$, $\hat\delta$ small enough and let $t_4>0$ be large enough (recall $k$ is converging pointwise to its average) such that
\begin{align*}
C\int 
	k_{\theta^3}^2k_\theta^2
	+ k_{\theta\theta}^4
	+ k_{\theta}^8
	\,d\theta
	\le
		\frac{c}{2}\delta\int
			k_{\theta^4}^2
			\,d\theta
		+ \frac12\int 
			k_{\theta^3}^2k_\theta^2
			+ k_{\theta\theta}^4
			+ k_{\theta}^8
			\,d\theta
\,.
\end{align*}
Absorbing then combining with \eqref{EQcase4} gives
\begin{align*}
\frac{d}{dt} \int h_{\theta^4}^2\,d\theta
	&\le -\frac{c}{2}\int k_{\theta^4}^2\,d\theta
	   -2\int h_{\theta^4}^2\,d\theta
\,,
\end{align*}
for $t\ge t_4$, which clearly implies
\[
||h_{\theta^4}||_2^2(t)
\le 
||h_{\theta^4}||_2^2(t_4)
\]
as required.

{\bf Case $p=5$.} 
For the next case we calculate
\begin{align}
\frac{d}{dt} \int h_{\theta^5}^2\,d\theta
	&= 2\int (1/k)_{\theta^5}k_{\theta^5}\,d\theta
	   -2\int h_{\theta^5}^2\,d\theta
\notag\\
	&\le -c\int k_{\theta^5}^2\,d\theta
		+ C\int 
			k_{\theta^4}^2k_\theta^2
			+ k_{\theta^3}^2k_{\theta\theta}^2
			+ k_{\theta^3}^2k_\theta^4
			+ k_{\theta\theta}^4k_\theta^2
			+ k_{\theta\theta}^2k_\theta^6
			+ k_{\theta}^{10}
			\,d\theta
	   -2\int h_{\theta^5}^2\,d\theta
\notag\\
	&\le -c\int k_{\theta^5}^2\,d\theta
		+ C\int 
			k_{\theta^4}^2
			+ k_{\theta^3}^2k_{\theta\theta}^2
			+ k_{\theta^3}^2
			+ k_{\theta\theta}^4
			+ k_{\theta\theta}^2
			\,d\theta
		+ C
	   -2\int h_{\theta^5}^2\,d\theta
\,.
\label{EQcase5}
\end{align}
In the last step we used the fact that $||h_{\theta^4}||_2^2(t)$ bounded
implies $||k_{\theta}||_\infty$ bounded (see the proof of Proposition
\ref{PNallest}).

Interpolation and the estimate \eqref{EQestgen2} gives
\begin{align*}
	C\int 
		k_{\theta^4}^2
		+ k_{\theta^3}^2k_{\theta\theta}^2
		+ k_{\theta^3}^2
		+ k_{\theta\theta}^4
		+ k_{\theta\theta}^2
		\,d\theta
	\le -\frac{c}{4}\int k_{\theta^5}^2\,d\theta
		+ C\int 
		  k_{\theta^3}^2k_{\theta\theta}^2
		+ k_{\theta\theta}^4
		\,d\theta
		+ C
\,.
\end{align*}
Using now \eqref{EQestgen2} with $m=2$ and $p=2$ gives $||k_{\theta\theta}||_4^4(t) \le 
C||k_{\theta^3}||_2^2(t)$ for $t>t_4$, which can be interpolated.
For the remaining term, we estimate
\[
\int k_{\theta^3}^2k_{\theta\theta}^2\,d\theta
\le 
	\int |k_{\theta^4}|\,|k_{\theta\theta}^3|\,d\theta
	- 2\int k_{\theta^3}^2k_{\theta\theta}^2\,d\theta
\]
which implies
\begin{align*}
\int k_{\theta^3}^2k_{\theta\theta}^2\,d\theta
&\le
	\frac13\int |k_{\theta^4}|\,|k_{\theta\theta}^3|\,d\theta
\\&\le 
	\delta \int k_{\theta\theta}^6\,d\theta
	+ C(\delta)\int k_{\theta^4}^2\,d\theta
\\&\le 
	\delta \int k_{\theta^3}^2k_{\theta\theta}^2k_{\theta}^2\,d\theta
	+ C(\delta)\int k_{\theta^4}^2\,d\theta
\\&\le 
	C\delta \int k_{\theta^3}^2k_{\theta\theta}^2\,d\theta
	+ C(\delta)\int k_{\theta^4}^2\,d\theta
\,,
\end{align*}
for $t>t_4$.
Above we used \eqref{EQestgen2} with $m=2$, $q=3$ and then the boundedness of
$||k_\theta||_\infty$.
Absorbing, we find
\[
\int k_{\theta^3}^2k_{\theta\theta}^2\,d\theta
\le
	C(\delta)\int k_{\theta^4}^2\,d\theta
\le
	\frac{c}{4}\int k_{\theta^5}^2\,d\theta
	+ C
\,.
\]
%
These estimates allow us to conclude, with \eqref{EQcase5}, that
$||h_{\theta^5}||_2^2(t) \le C$ for $t>t_5$, where we
simply choose $t_5=t_4$.\footnote{This uses the negative term on the RHS of
\eqref{EQcase5}.}

{\bf Case $p=6$.} 
For the next case we calculate
\begin{align}
\frac{d}{dt} \int h_{\theta^6}^2\,d\theta
	&= 2\int (1/k)_{\theta^6}k_{\theta^6}\,d\theta
	   -2\int h_{\theta^6}^2\,d\theta
\notag\\
	&\le -c\int k_{\theta^6}^2\,d\theta
	   -2\int h_{\theta^6}^2\,d\theta
\notag\\&\qquad
		+ C\int 
			k_{\theta^5}^2k_\theta^2
			+ k_{\theta^4}^2k_{\theta\theta}^2
			+ k_{\theta^4}^2k_\theta^4
			+ k_{\theta^3}^4
			+ k_{\theta^3}^2k_{\theta\theta}^2k_\theta^2
			+ k_{\theta^3}^2k_\theta^6
			+ k_{\theta^2}^6
			+ k_{\theta^2}^4k_\theta^4
			+ k_{\theta^2}^2k_\theta^8
			+ k_{\theta}^{12}
			\,d\theta
\notag\\
	&\le -c\int k_{\theta^6}^2\,d\theta
		+ C\int 
			k_{\theta^5}^2
			+ k_{\theta^4}^2
			+ k_{\theta^3}^4
			+ k_{\theta^3}^2
			\,d\theta
		+ C
	   -2\int h_{\theta^6}^2\,d\theta
\,.
\label{EQcase6}
\end{align}
In the last step we used the fact that $||h_{\theta^5}||_2^2(t)$ bounded
implies $||k_{\theta\theta}||_\infty$ bounded (see the proof of Proposition
\ref{PNallest}).

Interpolation and the estimate \eqref{EQestgen2} (for $q=2$, $m=3$) gives
\begin{align*}
	C\int 
		k_{\theta^5}^2
		+ k_{\theta^4}^2
		+ k_{\theta^3}^4
		+ k_{\theta^3}^2
		\,d\theta
	\le -\frac{c}{4}\int k_{\theta^5}^2\,d\theta
		+ C\int 
		  k_{\theta^4}^2k_{\theta\theta}^2
		\,d\theta
		+ C
	\le -\frac{c}{2}\int k_{\theta^5}^2\,d\theta
		+ C
\,.
\end{align*}
These estimates allow us to conclude, with \eqref{EQcase6}, that
$||h_{\theta^6}||_2^2(t) \le C$ for $t\ge t_6$, where we
simply choose $t_6=t_5$.

{\bf Final case: Induction for large $p$.} 
Suppose there exist times $\{t_1,\ldots,t_p\}$ such that
\[
	\vn{h_{\theta^p}}_2^2(t) \le 
	\vn{h_{\theta^p}}_2^2(t_p)
\]
for all $t>t_p$.
We wish to show that this implies there exists a $t_{p+1}$ such that
\[
	\vn{h_{\theta^{p+1} }}_2^2(t) \le 
	\vn{h_{\theta^{p+1} }}_2^2(t_{p+1})
\]
for all $t>t_{p+1}$.

As before, previous steps imply
\begin{equation}
\label{EQresckhighinfinity}
	||k_{\theta^{p-3}}||_\infty \le C
\,.
\end{equation}
Note that as $p \ge 7$, we know at least that the first four derivatives of $k$ are bounded in $L^\infty$.

Recall \eqref{EQgeneralrecipderiv}, which we copy here for the convenience of the reader:
\begin{equation*}
(k^{-1})_{\theta^m}
 = 
	- k^{-2}k_{\theta^m}
	+ \sum_{q=2}^m k^{-(q+1)} \sum_{i_1+\ldots+i_q=m} c(i_1,\ldots,i_q) k_{\theta^{i_1}} \cdots k_{\theta^{i_q}}
\,.
\end{equation*}
We compute
\begin{align*}
\frac{d}{dt} \int h_{\theta^{p+1}}^2\,d\theta
&=
	2\int (k^{-1})_{\theta^{p+1}}k_{\theta^{p+1}}\,d\theta
	- 2\int h_{\theta^{p+1}}^2\,d\theta
\,.
\end{align*}
Using \eqref{EQgeneralrecipderiv} and then \eqref{EQresckhighinfinity}, we estimate
\begin{align*}
	-2\int &(k^{-1})_{\theta^{p+1}}k_{\theta^{p+1}}\,d\theta
\le
	-\int k^{-2}k_{\theta^{(p+1)}}^2\,d\theta
	+C\sum_{q=2}^{p+1}
		\sum_{i_1+\ldots+i_q={p+1}} c(i_1,\ldots,i_q)
	\int 
	 k_{\theta^{i_1}}^2 \cdots k_{\theta^{i_q}}^2
	 \,d\theta
\\
&\le
	-\int k^{-2}k_{\theta^{(p+1)}}^2\,d\theta
	+ C
	+ C\sum_{q=2}^{4}
		\sum_{i_1+\ldots+i_q={p+1}} c(i_1,\ldots,i_q)
	\int 
	 k_{\theta^{i_1}}^2 \cdots k_{\theta^{i_q}}^2
	 \,d\theta
\\
&\le
	-\int k^{-2}k_{\theta^{(p+1)}}^2\,d\theta
	+ C
\\
&\qquad
	+ C\sum_{i_1+i_2={p+1}} 
	 \int 
	 	k_{\theta^{i_1}}^2 k_{\theta^{i_2}}^2
		\,d\theta
	+ C\sum_{i_1+i_2+i_3={p+1}} 
	 \int 
	 	k_{\theta^{i_1}}^2 k_{\theta^{i_2}}^2 k_{\theta^{i_3}}^2
		\,d\theta
\\
&\qquad
	+ C\sum_{i_1+i_2+i_3+i_4={p+1}}
	 \int 
	 	k_{\theta^{i_1}}^2 k_{\theta^{i_2}}^2 k_{\theta^{i_3}}^2 k_{\theta^{i_4}}^2
		\,d\theta
\,.
\\
\intertext{In the summations above we use again \eqref{EQresckhighinfinity} to estimate all factors of the form $k_{\theta^m}$ for $m\le p-3$ (note that $m$ is at least 3), and find}
\frac{d}{dt} \int h_{\theta^{(p+1)}}^2\,d\theta
&\le
	-c\int k^{-2}k_{\theta^{(p+1)}}^2\,d\theta
	+ C
	- 2\int h_{\theta^{(p+1)}}^2\,d\theta
\\
&\qquad
	 + C\int 
		k_{\theta^{p}}^2 k_{\theta}^2 
	 	+ k_{\theta^{p-1}}^2 k_{\theta^2}^2  
	 	+ k_{\theta^{p-2}}^2 k_{\theta^3}^2 
		\,d\theta
\\
&\qquad
	+ C\int 
	 	k_{\theta^{p-1}}^2 k_{\theta}^2 k_{\theta}^2
	 	+ k_{\theta^{p-2}}^2 k_{\theta^2}^2 k_{\theta}^2
		\,d\theta
\\
&\qquad
	+ C\int 
	 	k_{\theta^{p-2}}^2 k_{\theta}^2 k_{\theta}^2 k_{\theta}^2
		\,d\theta
\\
&\le
	-\int k^{-2}k_{\theta^{(p+1)}}^2\,d\theta
	+C\int k_{\theta^{p}}^2 \,d\theta
	+C\int k_{\theta^{(p-1)}}^2 \,d\theta
	+C\int k_{\theta^{(p-2)}}^2 \,d\theta
	+C
	- 2\int h_{\theta^{(p+1)}}^2\,d\theta
\,.
\intertext{Corollary \ref{CYrescaledfacts}, integration by parts and interpolation yields}
\frac{d}{dt} \int h_{\theta^{(p+1)}}^2\,d\theta
&\le
	-c\int k_{\theta^{(p+1)}}^2\,d\theta
	+C\int k_{\theta^{p}}^2 \,d\theta
	+C
	- 2\int h_{\theta^{(p+1)}}^2\,d\theta
\le
	-c\int k_{\theta^{p+1}}^2\,d\theta
	+ C
	- 2\int h_{\theta^{(p+1)}}^2\,d\theta
\,.
\end{align*}
This implies $||h_{\theta^{(p+1)}}||_2^2(t) \le C$ for $t\ge t_{p+1}$, where we choose $t_{p+1} = t_p$.
\end{proof}

Finally, we give uniform exponential decay estimates for all derivatives of the rescaled support function.


\begin{prop}
\label{PNallconv}
Consider the rescaling of an entropy flow generated by Theorem \ref{TMTM21}.
There exist $t_0>0$ such that for all $p\in\N$ we have
\[
	\vn{h_{\theta^p}}_2^2(t) \le C\,e^{-t}\,,\quad \text{for} \quad t\ge t_0
	\,.
\]
\end{prop}
\begin{proof}
First, note that for Proposition \ref{PNrescallest} the times $t_p$ satisfy
$t_p = t_{p-1}$ for $p\ge5$, so we may take $t_0 = \max\{t_1,t_2,t_3,t_4\}$.

For $p\ge 2$ we interpolate to find
\[
	\vn{h_{\theta^p}}_2^2(t) \le C\vn{h_{\theta^{2p-2}}}_2\vn{h_\theta}_2
	\le C\,e^{-t}
\,,
\]
for $t\ge t_0$.
This uses Proposition \ref{PNrescexpdecay} and Proposition \ref{PNrescallest}.
\end{proof}

Proposition \ref{PNallconv} implies convergence as $t\nearrow\infty$ of the rescaled support function to that of a round circle in $C^\infty(d\theta)$ by a standard argument (c.f. \cite{AndrewsCC}).

\appendix


\section{Local well-posedness}

%
%

\label{SClocalwellp}

In this section, we establish the following local well-posedness theorem.
This will be used in the compactness and smoothing argument in Section \ref{SCfundest}.

\begin{thm}
\label{TM1}
Suppose $\gamma_0:\S^1\rightarrow\R^2$ is a locally convex curve of class $W^{3,\infty}(du)$.
Then there exists a $T>0$ and unique corresponding entropy flow $\gamma:\S^1\times(0,T)\rightarrow\R^2$ such that
\begin{itemize}
\item $\lim_{t\searrow0}\gamma = \gamma_0$ in $W^{3,\infty}(du)$
\item $\gamma(\cdot,t)$ is of class $C^\infty$ and locally convex
\item $T<\infty$ implies $\lim_{t\nearrow\infty} \vn{\gamma(\cdot,t)}_{W^{3,\infty}(du)} = +\infty$.
\end{itemize}
\end{thm}

We will in fact establish existence for initial strictly convex immersed data of class $W^{4-4/p,p}$ with $p>5/2$. Of course the $W^{3,\infty}$ hypothesis implies this with for instance $p=4$.
Formally Theorem \ref{TM1} is a corollary of Theorem \ref{TMlod}.

	Suppose that $\gamma : \mathbb{S}^1 \times [0, T) \rightarrow \mathbb{R}^2$ lies in the maximal regularity space
\[
	\mathbb{E}^{1,p}((0,T); L^p(\mathbb{S}^1; \mathbb{R}^2)) = W^{1,p}((0,T); L^p(\mathbb{S}^1; \mathbb{R}^2)) \cap L^p((0,T); W^{4,p}(\mathbb{S}^1; \mathbb{R}^2))
\]
	for some $1 < p < \infty$ and it satisfies \eqref{EF}. $\mathbb{E}^{1,p}((0,T); L^p(\mathbb{S}^1; \mathbb{R}^2))$ is a Banach space when equipped with the norm
\[
	\| \gamma \|_{\mathbb{E}^{1,p}((0,T); L^p(\mathbb{S}^1; \mathbb{R}^2))} := \max\left\{ \| \gamma \|_{W^{1,p}((0,T); L^p(\mathbb{S}^1; \mathbb{R}^2))} , \| \gamma \|_{L^p((0,T); W^{4,p}(\mathbb{S}^1; \mathbb{R}^2))} \right\}.
\]
	The optimal space of initial data in the class of maximal $L^p$-regularity is precisely immersed curves in
\[
	(L^p(\mathbb{S}^1; \mathbb{R}^2), W^{4,p}(\mathbb{S}^1; \mathbb{R}^2))_{1 - \frac{1}{p},p} = W^{4(1 - \frac{1}{p}),p}(\mathbb{S}^1; \mathbb{R}^2)
\]
	for which we shall also make the additional assumption that $\gamma_0$ is strictly locally convex so that the entropy is initially well defined.
	By expanding out \eqref{EQefingenlparam} we see that the entropy flow in a general parameter $u$ is equivalent to solving
\begin{align}
	\begin{cases}
	\quad \partial_t \gamma &= - \frac{1}{\langle \gamma_{uu}, \nu \rangle^2} \gamma_{u^4}^{\perp} - \frac{\langle \gamma_{u^3}, \nu \rangle}{|\gamma_u| \langle \gamma_{uu}, \nu \rangle}\gamma_{u^3}^{\perp} - \frac{1}{|\gamma_u|^2} \gamma_{uu}^{\perp} - \frac{4\langle \gamma_{u^3}, \gamma_u \rangle}{|\gamma_u|^2 \langle \gamma_{uu}, \nu \rangle}\nu + \frac{6 \langle \gamma_{uu}, \gamma_u \rangle^2}{|\gamma_u|^4 \langle \gamma_{uu}, \nu \rangle} \nu, \quad t > 0 \\
		\gamma(\cdot,0) &= \gamma_0(\cdot).
	\end{cases}
\end{align}

	We shall linearise this operator by freezing the highest order term at the initial datum $\gamma_0$ such that we instead will consider
\begin{align}
\begin{cases}
	\partial_t \gamma + \frac{1}{|(\gamma_0)_{uu}^{\perp}|^2} \gamma_{u^4}^{\perp} &= \left(\frac{1}{|(\gamma_0)_{uu}^{\perp}|^2} - \frac{1}{|\gamma_{uu}^{\perp}|^2} \right)\gamma_{u^4}^{\perp} + F(\partial^3_u \gamma, \partial^2_u \gamma, \partial_u \gamma), \quad t > 0 \label{2} \\
	\qquad \qquad \ \ \gamma(\cdot,0) &= \gamma_0(\cdot)
\end{cases}
\end{align}
	where
\[
	F(\partial^3_u \gamma, \partial^2_u \gamma, \partial_u \gamma) = - \frac{\langle \gamma_{u^3}, \nu \rangle}{|\gamma_u| \langle \gamma_{uu}, \nu \rangle}\gamma_{u^3}^{\perp} - \frac{1}{|\gamma_u|^2} \gamma_{uu}^{\perp} - \frac{4\langle \gamma_{u^3}, \gamma_u \rangle}{|\gamma_u|^2 \langle \gamma_{uu}, \nu \rangle}\nu + \frac{6 \langle \gamma_{uu}, \gamma_u \rangle^2}{|\gamma_u|^4 \langle \gamma_{uu}, \nu \rangle} \nu.
\]

	To establish local well-posedness we follow \cite{RS} by using contraction estimates on an appropriate subset of $\mathbb{E}^{1,p}((0,T); L^p(\mathbb{S}^1; \mathbb{R}^2))$ and extending to all of $\mathbb{E}^{1,p}((0,T); L^p(\mathbb{S}^1; \mathbb{R}^2))$ with an appropriate contradiction argument. To that end suppose for a given $R > 0$ we define
\[
	\overline{B}_R^{\mathbb{E}^{1,p}}(\bar \gamma) := \{ \gamma \in \mathbb{E}^{1,p}((0,T); L^p(\mathbb{S}^1; \mathbb{R}^2)) : \| \gamma - \bar \gamma \|_{\mathbb{E}^{1,p}((0,T); L^p(\mathbb{S}^1; \mathbb{R}^2))} \leq R, \gamma(\cdot,0) = \bar \gamma(\cdot,0)\}
\]
	where $\bar \gamma$ is a specified reference flow. We shall assume that $R, T$ are bounded apriori by $R_0, T_0$ respectively and assume that $\bar \gamma$ is the solution to the linear evolution equation
\begin{align}
	\begin{cases}
		\partial_t \bar \gamma + \frac{1}{|(\gamma_0)_{uu}^{\perp}|^2} \gamma_{u^4}^{\perp} &= 0, \quad t > 0 \\
		\qquad \qquad \ \ \bar \gamma(\cdot,0) &= \gamma_0(\cdot).
	\end{cases}
\end{align}
	Such a reference flow is locally well-posed from the following theorem.

\begin{theorem}\label{Thm1}
	Let $1 < p < \infty, 0 < T \leq T_0$ and suppose that $a \in C^0(\mathbb{S}^1; \mathbb{R}^2)$ with 
\[
	(f,\gamma_0) \in L^p((0,T); L^p(\mathbb{S}^1; \mathbb{R}^2)) \times W^{4(1 - \frac{1}{p}),p}(\mathbb{S}^1; \mathbb{R}^2).
\]
	Then there exists a unique $\sigma \in \mathbb{E}^{1,p}((0,T); L^p(\mathbb{S}^1; \mathbb{R}^2))$ satisfying
\begin{align}\label{4}
	\begin{cases}
		\partial_t \sigma + a \sigma_{u^4} &= f, \quad t > 0 \\
		\quad \ \ \ \sigma(\cdot,0) &= \gamma_0(\cdot)
	\end{cases}
\end{align}
	and there exists $C = C(T,p,a) > 0$ such that
\begin{align}\label{5}
	\| \sigma \|_{\mathbb{E}^{1,p}((0,T); L^p(\mathbb{S}^1; \mathbb{R}^2))} \leq C\left(\| f \|_{L^p((0,T); L^p(\mathbb{S}^1; \mathbb{R}^2))} + \| \gamma_0 \|_{W^{4(1 - \frac{1}{p}),p}(\mathbb{S}^1; \mathbb{R}^2)}\right).
\end{align}
\end{theorem}
\begin{proof}
	This follows from the maximal regularity of the operator $a\partial^4_u$ as seen in \cite{MIQPEE} Theorem 6.3.2 with zero boundary data. 
\end{proof}
	
	Our present goal then is to control the non-linearities present in $F$ but also in controlling $|(\gamma_0)_{uu}^{\perp}|^{-2} - |\gamma_{uu}^{\perp}|^{-2}$. To do so, we shall employ the following lemmata and proposition.

\begin{lemma}[\cite{RS} Proposition B.3]\label{Lem1}
	Let $k \in \{1,2,3,4\}, 1 < p < \infty$ and $\rho_1,\rho_2 \in [1, \infty)$. Suppose there exists some $\theta \in [0,1]$ such that for $k \in \{1,2,3,4\}$
\[
	\frac{4 - k}{4}\theta - \frac{1}{p} \geq - \frac{1}{\rho_1}, \quad (4-k)(1- \theta) - \frac{1}{p} \geq - \frac{1}{\rho_2}
\]
	then for any $0 < T \leq T_0$,
\begin{align*}
	\partial^k_u : \mathbb{E}^{1,p}((0,T); L^p(\mathbb{S}^1; \mathbb{R}^2)) \hookrightarrow L^{\rho_1}((0,T); L^{\rho_2}(\mathbb{S}^1; \mathbb{R}^2))
\end{align*}
	with the estimate
\begin{align*}
	\| \partial^k_u \gamma \|_{L^{\rho_1}((0,T); L^{\rho_2}(\mathbb{S}^1; \mathbb{R}^2))} \leq C(T_0,k,\theta,p,\rho_1,\rho_2) \| \gamma \|_{\mathbb{E}^{1,p}((0,T); L^p(\mathbb{S}^1; \mathbb{R}^2))}.
\end{align*}
\end{lemma}

\begin{lemma}\label{Lem2}
	Let $X$ be a Banach space of class $\mathcal{HT}$ i.e. a Banach space for which the Hilbert Transform is a bounded linear operator on $L^2(\mathbb{R}; X)$ and $I \subset \mathbb{R}$. Then
\begin{enumerate}[(i)]
	\item For $s \in (0,\infty), 1 \leq p \leq \infty$
		\[
			W^{s,p}(I; X) \hookrightarrow C^{[\alpha], \alpha - [\alpha]}(\overline I; X), \quad \mathbb{N} \not \ni \alpha = s - \frac{1}{p}.
		\]
	\item If $k + \frac{1}{p} < s < k + 1 + \frac{1}{p}$ for some $k \in \mathbb{N}_0$
		\[
			W^{s,p}(I; X) \hookrightarrow BUC^k(\bar I; X).
		\]
			The same assertion holds if we replace $W^{s,p}(I; X)$ by $H^{s,p}(I; X)$.
	\item For $l \in \{0,1,2,3\}$, $1 \leq p \leq \infty$ such that
		\[
			\mathbb{E}^{1,p}((0,T); L^p(\mathbb{S}^1; \mathbb{R}^2)) \hookrightarrow C^{0,\alpha}([0,T]; C^{l,\alpha}(\mathbb{S}^1; \mathbb{R}^2))
		\]
			for $p > \frac{5}{4 - l}$. Moreover there exists some $C = C(T_0,p,l) > 0$ such that
		\[
			\| \gamma \|_{C^{0,\alpha}([0,T]; C^{l,\alpha}(\mathbb{S}^1; \mathbb{R}^2))} \leq C(T_0,p,l) \| \gamma \|_{\mathbb{E}^{1,p}((0,T); L^p(\mathbb{S}^1; \mathbb{R}^2))}.
		\]
\end{enumerate}
\begin{proof}
	\begin{enumerate}[(i)]
		\item By characterisation of $W^{s,p}(I;X)$ as an interpolation space, the monotonicity and embeddings of Besov spaces (\cite{Amann} 3.2, 3.3 and \cite{Triebel} Theorem 1.5.1 (ii)) it holds that
			\[
				W^{s,p}(I;X) = B^s_{p,p}(I;X) \hookrightarrow B^s_{p,\infty}(I; X) \hookrightarrow B^{\alpha}_{\infty, \infty}(\overline I;X) = \mathcal{C}^{\alpha}(\overline I; X)
			\]
		where $\mathcal{C}^{\alpha}$ is the H\"older-Zygmund space of order $\alpha$. The claim then follows since by \cite{Triebel}  Theorem 1.2.2 $\mathcal{C}^{\alpha}$ coincides with $C^{[\alpha], \alpha - [\alpha]}$ for $\mathbb{N} \not \ni \alpha > 0$.
		\item Follows directly from \cite{MS} Proposition 2.10.
		\item By properties of interpolation spaces for $\theta \in (0,1)$ such that $2m(1 - \theta) \notin \mathbb{N}_0$
			\begin{align*}
				\mathbb{E}^{1,p}((0,T); L^p(\mathbb{S}^1; \mathbb{R}^2)) &\hookrightarrow (W^{1,p}((0,T); L^p(\mathbb{S}^1; \mathbb{R}^2)), L^p((0,T); W^{4,p}(\mathbb{S}^1; \mathbb{R}^2)))_{\theta,p} \\
						&= W^{\theta,p}((0,T); W^{4(1 - \theta),p}(\mathbb{S}^1; \mathbb{R}^2))
			\end{align*}
		where we have identified the two spaces using \cite{Amann} Theorem 3.1.
			By (i) for $\theta > \frac{1}{p}$ there exists some $(0,1) \ni \alpha_1 = \theta - \frac{1}{p}$ such that
				\[
					\mathbb{E}^{1,p}((0,T); L^p(\mathbb{S}^1; \mathbb{R}^2)) \hookrightarrow C^{0,\alpha_1}([0,T]; W^{2m(1 - \theta),p}(\mathbb{S}^1; \mathbb{R}^2)).
				\]
			Any $\gamma \in \mathbb{E}^{1,p}((0,T); L^p(\mathbb{S}^1; \mathbb{R}^2))$ has a weak $l$-th order spatial derivative with \newline $\partial^l_u \gamma \in C^{0,\alpha_1}([0,T]; W^{4 - l - 4\theta, p}(\mathbb{S}^1; \mathbb{R}^2))$ which embeds into
			\[
				C^{0, \alpha_1}([0,T]; C^{[\alpha_2],\alpha_2 - [\alpha_2]}(\mathbb{S}^1; \mathbb{R}^2))
			\]
			for some $\alpha_2 = 4 - l - 4\theta - \frac{1}{p} > 0$ which requires $\theta < 1 - \left(\frac{l}{4} + \frac{1}{4p}\right)$. Thus for the embeddings to hold we will need to choose $\theta$ such that $4(1 - \theta) \notin \mathbb{N}_0$ and
			\[
				\theta \in \left(\frac{1}{p}, 1 - \left(\frac{l}{4} + \frac{1}{4p}\right)\right)
			\]
				for this interval to be well defined we must have $p > \frac{5}{4 - l}$. The choice of $\theta = \frac{1}{2} - \frac{l}{8} - \frac{5}{8p}$ is well defined for every $p > \frac{5}{4 - l}$ with the exception of $p = 5$ when $l = 1$, else the choice of $\theta = \frac{23}{80} - \frac{l}{16}$ in the case that $p = 5$ works. By identifying $B^{\alpha}_{\infty, \infty}$ with the H\"older-Zygmund spaces $\mathcal{C}^{\alpha}$ and the fact that $B^{s_0}_{p,q} \hookrightarrow B^{s_1}_{p,q}$ for any $s_0 \geq s_1, 1 \leq p,q \leq \infty$ we have that for $\alpha := \min\{\alpha_1, \alpha_2\}$
			\begin{align*}
				C^{[\alpha_2], \alpha_2 - [\alpha_2]}(\mathbb{S}^1; \mathbb{R}^2) = \mathcal{C}^{\alpha_2}(\mathbb{S}^1; \mathbb{R}^2) &= B^{\alpha_2}_{\infty,\infty}(\mathbb{S}^1; \mathbb{R}^2) \\
						&\hookrightarrow B^{\alpha}_{\infty, \infty}(\mathbb{S}^1; \mathbb{R}^2) = \mathcal{C}^{\alpha}(\mathbb{S}^1; \mathbb{R}^2) = C^{0, \alpha}(\mathbb{S}^1; \mathbb{R}^2)
			\end{align*}
	and thus there exists some $\alpha \in (0,1)$ such that
	\[
		\mathbb{E}^{1,p}((0,T); L^p(\mathbb{S}^1; \mathbb{R}^2)) \hookrightarrow C^{0, \alpha}([0,T]; C^{l,\alpha}(\mathbb{S}^1; \mathbb{R}^2)).
	\]
	\end{enumerate}
\end{proof}
\end{lemma}

\begin{lemma}\label{Lem3}
	Let $l \in \{0,1,2,3\}$. For every $p > \frac{5}{4 - l}$, there exists some $T = T(T_0,R_0,p,l,\bar \gamma) > 0$ such that for some $R \in (0, R_0]$ and any $\gamma \in \overline{B}_R^{\mathbb{E}^{1,p}}(\bar \gamma)$
\[
	|\partial^l_u \gamma(u,t)| \geq \frac{1}{2} \inf_{v \in \mathbb{S}^1} |\partial^l_u \gamma_0(v)|
\]
	for every $(u,t) \in \mathbb{S}^1 \times [0, T)$. In particular if $|\partial^l_u \gamma_0| > 0$ then $|\partial^l_u \gamma(\cdot,t)| > 0$ for every $t \in [0, T)$.
\begin{proof}
	By \ref{Lem2} for every $p > \frac{5}{4 - l}$ we have that 
\[
	\mathbb{E}^{1,p}((0,T); L^p(\mathbb{S}^1; \mathbb{R}^2)) \hookrightarrow C^{0,\alpha}([0,T]; C^{l,\alpha}(\mathbb{S}^1; \mathbb{R}^2))
\]
	 which implies that $\gamma \in C^{0,\alpha}([0,T]; C^l(\mathbb{S}^1; \mathbb{R}^2))$ and there exists some $C = C(T_0,p,l) > 0$ such that for any $\gamma \in \overline{B}_R^{\mathbb{E}^{1,p}}(\bar \gamma)$
\begin{align*}
	\| \gamma - \gamma_0 \|_{C^l(\mathbb{S}^1; \mathbb{R}^2)} &\leq \| \gamma \|_{C^{0,\alpha}([0,T]; C^l(\mathbb{S}^1; \mathbb{R}^2))} T^{\alpha} \\
					&\leq \left(\| \gamma - \bar \gamma \|_{C^{0,\alpha}([0,T]; C^l(\mathbb{S}^1; \mathbb{R}^2))} + \| \bar \gamma \|_{C^{0,\alpha}([0,T]; C^l(\mathbb{S}^1; \mathbb{R}^2))} \right)T^{\alpha} \\
					&\leq \left(C(T_0,p,l) \| \gamma - \bar \gamma \|_{\mathbb{E}^{1,p}((0,T); L^p(\mathbb{S}^1; \mathbb{R}^2))} + C(T_0,p,l) \| \bar \gamma \|_{\mathbb{E}^{1,p}((0,T); L^p(\mathbb{S}^1; \mathbb{R}^2))} \right)T^{\alpha} \\
					&\leq C(T_0,R_0,p,l, \bar \gamma)T^{\alpha}
\end{align*}
	where we have used the fact that $\gamma - \bar \gamma$ has zero temporal trace. Thus if we choose $T = T(T_0,R_0,p,l, \bar \gamma) > 0$ small enough
\[
	|\partial^l_u \gamma(u,t)| \geq \inf_{\mathbb{S}^1} |\partial^l_u \gamma_0(u)| - |\partial^l_u \gamma(u,t) - \partial^l_u \gamma_0(u)| \geq \frac{1}{2} \inf_{\mathbb{S}^1} |\partial^l_u \gamma_0|.
\]
\end{proof}
\end{lemma}

\begin{lemma}\label{Lem4}
	Let $\gamma \in B_R^{\mathbb{E}^{1,p}}(\bar \gamma)$ for some $p > \frac{5}{2}$. Then there exists some $T = T(T_0,R_0,p,\bar \gamma) > 0$ such that
\[
	|\gamma_{uu}^{\perp}(u,t)| \geq \frac{1}{2} \inf_{v \in \mathbb{S}^1}|\gamma_{uu}^{\perp}(v,0)|
\]
	for every $(u,t) \in \mathbb{S}^1 \times [0, T)$. In particular if $|\gamma_{uu}^{\perp}(\cdot,0)| > 0$ then $|\gamma_{uu}^{\perp}(\cdot,t)| > 0$ for every $t \in [0,T)$.
\begin{proof}
	Let $\gamma \in B_R^{\mathbb{E}^{1,p}}(\bar \gamma)$ then the following holds
\begin{align}
	|\langle \gamma_{uu}, \nu \rangle \nu - \langle (\gamma_0)_{uu}, \nu_0 \rangle \nu_0| &\leq | \gamma_{uu} - (\gamma_0)_{uu}| + |\langle \gamma_{uu}, \tau \rangle \tau - \langle (\gamma_0)_{uu}, \tau_0 \rangle \tau_0| \nonumber\\
		&\leq |\gamma_{uu} - (\gamma_0)_{uu}| + |\langle \gamma_{uu}, \tau \rangle \tau - \langle (\gamma_0)_{uu}, \tau \rangle \tau| \nonumber  \\
				&\quad \ + |\langle (\gamma_0)_{uu}, \tau \rangle \tau - \langle (\gamma_0)_{uu}, \tau_0 \rangle \tau_0| \nonumber \\
		&\leq 2|\gamma_{uu} - (\gamma_0)_{uu}| + |\langle (\gamma_0)_{uu}, \tau \rangle \tau - \langle (\gamma_0)_{uu}, \tau_0 \rangle \tau| \nonumber \\
				&\quad \ + |\langle (\gamma_0)_{uu}, \tau_0 \rangle \tau - \langle (\gamma_0)_{uu}, \tau_0 \rangle \tau_0| \nonumber \\
		&\leq 2|\gamma_{uu} - (\gamma_0)_{uu}| + 2|(\gamma_0)_{uu}| |\tau - \tau_0| \label{6}.
\end{align}
	
	To control $\tau - \tau_0$ we shall use the Mean Value Theorem to find
\begin{align}
	|\tau - \tau_0| &= \left|\frac{\gamma_u}{|\gamma_u|} - \frac{(\gamma_0)_u}{|(\gamma_0)_u|}\right| \nonumber \\
						&\leq \left|\frac{1}{|(\gamma_0)_u|} \right| |\gamma_u - (\gamma_0)_u| + \left|\left(\frac{1}{|(\gamma_0)_u|} - \frac{1}{|\gamma_u|}\right) \gamma_u \right| \label{7} \\
						&\leq \left(\inf_{\mathbb{S}^1} |(\gamma_0)_u|\right)^{-1} |\gamma_u - (\gamma_0)_u| + \left(\inf_{\mathbb{S}^1} |(\gamma_0)_u|\right)^{-2} |\gamma_u - (\gamma_0)_u| |\gamma_u| \nonumber \\
						&\leq C(\gamma_0) |\gamma_u - (\gamma_0)_u| + C(\gamma_0) |\gamma_u - (\gamma_0)_u|^2 + C(\gamma_0) |\gamma_u - (\gamma_0)_u| | (\gamma_0)_u| \nonumber .
\end{align}
	If we let $\alpha_1, \alpha_2$ be chosen such that the embeddings
\[
	\mathbb{E}^{1,p}((0,T); L^p(\mathbb{S}^1; \mathbb{R}^2)) \hookrightarrow C^{0,\alpha_1}([0,T]; C^{1,\alpha_1}(\mathbb{S}^1; \mathbb{R}^2))
\]
\[
	\mathbb{E}^{1,p}((0,T); L^p(\mathbb{S}^1; \mathbb{R}^2)) \hookrightarrow C^{0,\alpha_2}([0,T]; C^{2,\alpha_2}(\mathbb{S}^1; \mathbb{R}^2))
\]
	hold and let $\alpha = \min\{\alpha_1,\alpha_2\}$ it then follows by \ref{Lem2} (iii)
\begin{align}
	\| \langle \gamma_{uu}, \nu \rangle \nu - \langle (\gamma_0)_{uu}, \nu_0 \rangle &\nu_0 \|_{C^0(\mathbb{S}^1; \mathbb{R}^2)} \nonumber \\ &\leq 2 \| \gamma_{uu} - (\gamma_0)_{uu} \|_{C^0(\mathbb{S}^1; \mathbb{R}^2)} \nonumber \\
				&\quad \ + C(\gamma_0) \| (\gamma_0)_{uu} \|_{C^0(\mathbb{S}^1; \mathbb{R}^2)} \| \gamma_u - (\gamma_0)_u \|_{C^0(\mathbb{S}^1; \mathbb{R}^2)}\nonumber  \\
				&\quad \ + C(\gamma_0) \| \gamma_u - (\gamma_0)_u \|_{C^0(\mathbb{S}^1; \mathbb{R}^2)}^2 \nonumber \\
				&\quad \ + C(\gamma_0) \| \gamma_u - (\gamma_0)_u \|_{C^0(\mathbb{S}^1; \mathbb{R}^2)} \| (\gamma_0)_u \|_{C^0(\mathbb{S}^1; \mathbb{R}^2)} \nonumber \\
		&\leq 2 \| \gamma \|_{C^{0,\alpha}([0,T]; C^2(\mathbb{S}^1; \mathbb{R}^2))} T^{\alpha} \nonumber \\
				&\quad \ + C(\gamma_0) \| \gamma \|_{C^{0,\alpha}([0,T]; C^1(\mathbb{S}^1; \mathbb{R}^2))}T^{\alpha} \nonumber \\
				&\quad \ + C(\gamma_0) \| \gamma \|_{C^{0,\alpha}([0,T]; C^1(\mathbb{S}^1; \mathbb{R}^2))}^2T^{2\alpha} \nonumber \\
				&\quad \ + C(\gamma_0) \| \gamma \|_{C^{0,\alpha}([0,T]; C^1(\mathbb{S}^1; \mathbb{R}^2))}T^{\alpha} \nonumber \\
		&\leq C(T_0,p,\gamma_0) \| \gamma \|_{\mathbb{E}^{1,p}((0,T); L^p(\mathbb{S}^1; \mathbb{R}^2))} T^{\alpha} \nonumber \\
				&\quad \ + C(T_0,p,\gamma_0) \| \gamma \|_{\mathbb{E}^{1,p}((0,T); L^p(\mathbb{S}^1; \mathbb{R}^2))}^2 T_0^{\alpha}T^{\alpha} \nonumber \\
		&\leq C(T_0,p,\gamma_0) \| \gamma - \bar \gamma \|_{\mathbb{E}^{1,p}((0,T); L^p(\mathbb{S}^1; \mathbb{R}^2))} T^{\alpha}\nonumber \\
				&\quad \ + C(T_0,p,\gamma_0) \| \bar \gamma \|_{\mathbb{E}^{1,p}((0,T); L^p(\mathbb{S}^1; \mathbb{R}^2))} T^{\alpha} \nonumber \\
				&\quad \ + C(T_0,p,\gamma_0) \| \gamma - \bar \gamma \|_{\mathbb{E}^{1,p}((0,T); L^p(\mathbb{S}^1; \mathbb{R}^2))}^2 T^{\alpha}\nonumber \\
				&\quad \ + C(T_0,p,\gamma_0) \| \gamma - \bar \gamma \|_{\mathbb{E}^{1,p}((0,T); L^p(\mathbb{S}^1; \mathbb{R}^2))} \| \bar \gamma \|_{\mathbb{E}^{1,p}((0,T); L^p(\mathbb{S}^1; \mathbb{R}^2))}T^{\alpha} \nonumber \\
				&\quad \ + C(T_0,p,\gamma_0) \| \bar \gamma \|_{\mathbb{E}^{1,p}((0,T); L^p(\mathbb{S}^1; \mathbb{R}^2))}^2 T^{\alpha} \nonumber \\
		&\leq C(T_0,R_0,p,\bar \gamma)T^{\alpha} \label{8}.
\end{align}
	In particular for $T = T(T_0,R_0,p, \bar \gamma) > 0$ sufficiently small
\begin{align*}
	|\langle \gamma_{uu}, \nu \rangle| &= |\gamma_{uu}^{\perp}| \\
												&\geq \inf_{\mathbb{S}^1} |\gamma_{uu}^{\perp}(\cdot,0)| - |\gamma_{uu}^{\perp}(\cdot,t) - \gamma_{uu}^{\perp}(\cdot,0)| \\
												&\geq \frac{1}{2} \inf_{\mathbb{S}^1} |\gamma_{uu}^{\perp}(\cdot,0)| = \frac{1}{2} \inf_{\mathbb{S}^1} |\langle (\gamma_0)_{uu}, \nu_0 \rangle|.
\end{align*}
\end{proof}
\end{lemma}

	We first shall tackle the highest order term as it will require the least machinery to control.
\begin{prop}\label{Prop1}
	Suppose that $\gamma \in B_R^{\mathbb{E}^{1,p}}(\bar \gamma)$. Then for sufficiently small $T,R > 0$ the map
\[
	\gamma \mapsto \left(\frac{1}{|(\gamma_0)_{uu}^{\perp}|^2} - \frac{1}{|\gamma_{uu}^{\perp}|^2} \right)\gamma_{u^4} : B_R^{\mathbb{E}^{1,p}}(\bar \gamma) \rightarrow L^p((0,T); L^p(\mathbb{S}^1; \mathbb{R}^2))
\]
	is a well defined $\varepsilon$-contraction i.e. a Lipschitz map with Lipschitz constant $\varepsilon$.
\begin{proof}
	Let $\gamma, \tilde \gamma \in B_R^{\mathbb{E}^{1,p}}(\bar \gamma)$. Observe by the triangle inequality that
\begin{align*}
	\left|\left(\frac{1}{|(\gamma_0)_{uu}^{\perp}|^2} - \frac{1}{|\gamma_{uu}^{\perp}|^2} \right)\gamma_{u^4} - \left(\frac{1}{|(\gamma_0)_{uu}^{\perp}|^2} - \frac{1}{|\tilde \gamma_{uu}^{\perp}|^2}\right) \tilde \gamma_{u^4} \right| &\leq \left|\frac{1}{|(\gamma_0)_{uu}^{\perp}|^2} - \frac{1}{|\gamma_{uu}^{\perp}|^2} \right| |\gamma_{u^4} - \tilde \gamma_{u^4}| \\
					&\quad \ + \left|\frac{1}{|\gamma_{uu}^{\perp}|^2} - \frac{1}{|\tilde \gamma_{uu}^{\perp}|^2}\right| |\tilde \gamma_{u^4}|
\end{align*}
	from which it holds
\begin{align*}
	\left \| \left(\frac{1}{|(\gamma_0)_{uu}^{\perp}|^2} - \frac{1}{|\gamma_{uu}^{\perp}|^2} \right)\gamma_{u^4} - \left(\frac{1}{|(\gamma_0)_{uu}^{\perp}|^2} - \frac{1}{|\tilde \gamma_{uu}^{\perp}|^2}\right) \tilde \gamma_{u^4} \right \|_{L^p((0,T); L^p(\mathbb{S}^1;\mathbb{R}^2))} \\
	\leq \left \| \frac{1}{|(\gamma_0)_{uu}^{\perp}|^2} - \frac{1}{|\gamma_{uu}^{\perp}|^2} \right \|_{L^{\infty}((0,T); L^p(\mathbb{S}^1))} \| \gamma_{u^4} - \tilde \gamma_{u^4} \|_{L^p((0,T); L^p(\mathbb{S}^1; \mathbb{R}^2))} \\
	\quad \ + \left \| \frac{1}{|\gamma_{uu}^{\perp}|^2} - \frac{1}{|\tilde \gamma_{uu}^{\perp}|^2} \right \|_{L^{\infty}((0,T); L^p(\mathbb{S}^1))} \| \tilde \gamma_{u^4} \|_{L^p((0,T); L^p(\mathbb{S}^1; \mathbb{R}^2))}.
\end{align*}
	We may improve this further by utilising the Mean Value Theorem and \ref{Lem4}
\begin{align}
	\left|\frac{1}{|(\gamma_0)_{uu}^{\perp}|^2} - \frac{1}{|\gamma_{uu}^{\perp}|^2}\right| \leq 2\left(\frac{1}{2} \inf_{\mathbb{S}^1}|(\gamma_0)_{uu}^{\perp}|\right)^{-3} |(\gamma_0)_{uu}^{\perp} - \gamma_{uu}^{\perp}| \leq C(\gamma_0)|\gamma_{uu}^{\perp} - (\gamma_0)_{uu}^{\perp}| \label{9}
\end{align}
	to find that
\begin{align*}
	\left \| \frac{1}{|(\gamma_0)_{uu}^{\perp}|^2} - \frac{1}{|\gamma_{uu}^{\perp}|^2} \right \|_{L^{\infty}((0,T); L^p(\mathbb{S}^1))} &\leq C(\gamma_0) \sup_{t \in [0, T]} \| \gamma_{uu}^{\perp}(\cdot,t) - \gamma_{uu}^{\perp}(\cdot,0)\|_{C^0(\mathbb{S}^1; \mathbb{R}^2)} \\
							&\leq C(T_0,R_0,p,\bar \gamma)T^{\alpha}
\end{align*}
	where $\alpha$ is the same as in \ref{Lem4}. Combining this with the simple estimate
\[
	\| \partial^4_u \gamma - \partial^4_u \tilde \gamma \|_{L^p((0,T); L^p(\mathbb{S}^1; \mathbb{R}^2))} \leq \| \gamma - \tilde \gamma \|_{\mathbb{E}^{1,p}((0,T); L^p(\mathbb{S}^1; \mathbb{R}^2))}
\]
	shows that for sufficiently small $T = T(T_0,R_0,p,\varepsilon,\bar \gamma) > 0$
\[
	\left \| \frac{1}{|(\gamma_0)_{uu}^{\perp}|^2} - \frac{1}{|\gamma_{uu}^{\perp}|^2} \right \|_{L^{\infty}((0,T); L^p(\mathbb{S}^1))} \| \gamma_{u^4} - \tilde \gamma_{u^4} \|_{L^p(\mathbb{S}^1; \mathbb{R}^2)} \leq \frac{\varepsilon}{2} \| \gamma - \tilde \gamma \|_{\mathbb{E}^{1,p}((0,T); L^p(\mathbb{S}^1; \mathbb{R}^2))}.
\]
	For the remaining term we use \eqref{9} in conjunction with \eqref{6}, \eqref{7} replacing $\gamma_0$ with $\tilde \gamma$ where appropriate to find that
\begin{align*}
	&\left \| \frac{1}{|\gamma_{uu}^{\perp}|^2} - \frac{1}{|\tilde \gamma_{uu}^{\perp}|^2} \right \|_{L^{\infty}((0,T); L^p(\mathbb{S}^1))} \\
		&\leq 2\left(\inf_{\mathbb{S}^1} |(\gamma_0)_{uu}^{\perp}|\right)^{-3} \| \gamma_{uu}^{\perp} - \tilde \gamma_{uu}^{\perp} \|_{L^{\infty}((0,T); L^p(\mathbb{S}^1; \mathbb{R}^2))} \\
		&\leq C(\gamma_0) \| \gamma_{uu} - \tilde \gamma_{uu} \|_{L^{\infty}((0,T); L^p(\mathbb{S}^1; \mathbb{R}^2))} \\
				&\quad \ + C(\gamma_0) \|\tilde \gamma_{uu}\|_{L^{\infty}((0,T); L^p(\mathbb{S}^1; \mathbb{R}^2))} \| \tau - \tilde \tau \|_{L^{\infty}((0,T); L^p(\mathbb{S}^1; \mathbb{R}^2))} \\
		&\leq C(p,\gamma_0) \| \gamma - \tilde \gamma \|_{C^{0,\alpha}((0,T); C^{2,\alpha}(\mathbb{S}^1; \mathbb{R}^2))} \\
				&\quad \ + C(p,\gamma_0) \| \tilde \gamma \|_{C^{0,\alpha}((0,T); C^{2,\alpha}(\mathbb{S}^1; \mathbb{R}^2))} \| \gamma - \tilde \gamma \|_{C^{0,\alpha}((0,T); C^{1,\alpha}(\mathbb{S}^1; \mathbb{R}^2))} \\
				&\quad \ + C(p,\gamma_0) \| \tilde \gamma \|_{C^{0,\alpha}((0,T); C^{2,\alpha}(\mathbb{S}^1; \mathbb{R}^2))} \| \gamma - \tilde \gamma \|_{C^{0,\alpha}((0,T); C^{1,\alpha}(\mathbb{S}^1; \mathbb{R}^2))} \| \tilde \gamma \|_{C^{0,\alpha}((0,T); C^{1,\alpha}(\mathbb{S}^1; \mathbb{R}^2))} \\
		&\leq C(T_0,p,\gamma_0) \| \gamma - \tilde \gamma \|_{\mathbb{E}^{1,p}((0,T); L^p(\mathbb{S}^1; \mathbb{R}^2))}T^{\alpha} \\
				&\quad \ + C(T_0,p,\gamma_0) \| \tilde \gamma - \bar \gamma \|_{C^{0,\alpha}([0,T]; C^{2,\alpha}(\mathbb{S}^1; \mathbb{R}^2))} \| \gamma - \tilde \gamma \|_{\mathbb{E}^{1,p}((0,T); L^p(\mathbb{S}^1; \mathbb{R}^2))} T^{\alpha} \\
				&\quad \ + C(T_0,p,\gamma_0) \| \bar \gamma \|_{C^{0,\alpha}([0,T]; C^{2,\alpha}(\mathbb{S}^1; \mathbb{R}^2))} \| \gamma - \tilde \gamma \|_{\mathbb{E}^{1,p}((0,T); L^p(\mathbb{S}^1; \mathbb{R}^2))} T^{\alpha} \\
				&\quad \ + C(T_0,p, \gamma_0) \| \tilde \gamma - \bar \gamma \|_{C^{0,\alpha}([0,T]; C^{2,\alpha}(\mathbb{S}^1; \mathbb{R}^2))} \| \gamma - \tilde \gamma \|_{\mathbb{E}^{1,p}((0,T); L^p(\mathbb{S}^1; \mathbb{R}^2))} \| \tilde \gamma \|_{C^{0,\alpha}([0,T]; C^{1,\alpha}(\mathbb{S}^1; \mathbb{R}^2))}T^{\alpha} \\
				&\quad \ + C(T_0,p,\gamma_0) \| \bar \gamma \|_{C^{0,\alpha}([0,T]; C^{2,\alpha}(\mathbb{S}^1; \mathbb{R}^2))} \| \gamma - \tilde \gamma \|_{\mathbb{E}^{1,p}((0,T); L^p(\mathbb{S}^1; \mathbb{R}^2))} \| \tilde \gamma \|_{C^{0,\alpha}([0,T]; C^{2,\alpha}(\mathbb{S}^1; \mathbb{R}^2))}T^{\alpha} \\
		&\leq C(T_0,p,\gamma_0) \| \gamma - \tilde \gamma \|_{\mathbb{E}^{1,p}((0,T); L^p(\mathbb{S}^1; \mathbb{R}^2))}T^{\alpha} \\
				&\quad \ + C(T_0,p,\gamma_0) \| \tilde \gamma - \bar \gamma \|_{\mathbb{E}^{1,p}((0,T); L^p(\mathbb{S}^1; \mathbb{R}^2))} \| \gamma - \tilde \gamma \|_{\mathbb{E}^{1,p}((0,T); L^p(\mathbb{S}^1; \mathbb{R}^2))} T^{2\alpha} \\
				&\quad \ + C(T_0,p,\bar \gamma) \| \gamma - \tilde \gamma \|_{\mathbb{E}^{1,p}((0,T); L^p(\mathbb{S}^1; \mathbb{R}^2))} T^{\alpha} \\
				&\quad \ + C(T_0,p, \gamma_0) \| \tilde \gamma - \bar \gamma \|_{\mathbb{E}^{1,p}((0,T); L^p(\mathbb{S}^1; \mathbb{R}^2))} \| \gamma - \tilde \gamma \|_{\mathbb{E}^{1,p}((0,T); L^p(\mathbb{S}^1; \mathbb{R}^2))} \| \tilde \gamma \|_{C^{0,\alpha}([0,T]; C^{1,\alpha}(\mathbb{S}^1; \mathbb{R}^2))}T^{2\alpha} \\
				&\quad \ + C(T_0,p,\bar \gamma) \| \gamma - \tilde \gamma \|_{\mathbb{E}^{1,p}((0,T); L^p(\mathbb{S}^1; \mathbb{R}^2))} \| \tilde \gamma \|_{C^{0,\alpha}([0,T]; C^{1,\alpha}(\mathbb{S}^1; \mathbb{R}^2))}T^{\alpha} \\
		&\leq C(T_0,p,\bar \gamma) \| \gamma - \tilde \gamma \|_{\mathbb{E}^{1,p}((0,T); L^p(\mathbb{S}^1; \mathbb{R}^2))}T^{\alpha} \\
				&\quad \ + C(T_0,R_0,p,\gamma_0) \| \gamma - \tilde \gamma \|_{\mathbb{E}^{1,p}((0,T); L^p(\mathbb{S}^1; \mathbb{R}^2))} T_0^{\alpha}T^{\alpha} \\
				&\quad \ + C(T_0,p, \gamma_0) \| \tilde \gamma - \bar \gamma \|_{\mathbb{E}^{1,p}((0,T); L^p(\mathbb{S}^1; \mathbb{R}^2))} \| \gamma - \tilde \gamma \|_{\mathbb{E}^{1,p}((0,T); L^p(\mathbb{S}^1; \mathbb{R}^2))} \| \tilde \gamma \|_{C^{0,\alpha}([0,T]; C^{1,\alpha}(\mathbb{S}^1; \mathbb{R}^2))}T_0^{\alpha}T^{\alpha} \\
				&\quad \ + C(T_0,p,\bar \gamma) \| \gamma - \tilde \gamma \|_{\mathbb{E}^{1,p}((0,T); L^p(\mathbb{S}^1; \mathbb{R}^2))} \| \tilde \gamma - \bar \gamma \|_{C^{0,\alpha}([0,T]; C^{1,\alpha}(\mathbb{S}^1; \mathbb{R}^2))}T^{\alpha} \\
				&\quad \ + C(T_0,p,\bar \gamma) \| \gamma - \tilde \gamma \|_{\mathbb{E}^{1,p}((0,T); L^p(\mathbb{S}^1; \mathbb{R}^2))} \| \bar \gamma \|_{C^{0,\alpha}([0,T]; C^{1,\alpha}(\mathbb{S}^1; \mathbb{R}^2))}T^{\alpha} \\
		&\leq C(T_0,R_0,p,\bar \gamma) \| \gamma - \tilde \gamma \|_{\mathbb{E}^{1,p}((0,T); L^p(\mathbb{S}^1; \mathbb{R}^2))} T^{\alpha} \\
				&\quad \ + C(T_0,R_0,p,\gamma_0) \| \gamma - \tilde \gamma \|_{\mathbb{E}^{1,p}((0,T); L^p(\mathbb{S}^1; \mathbb{R}^2))} \| \tilde \gamma - \bar \gamma \|_{C^{0,\alpha}([0,T]; C^{1,\alpha}(\mathbb{S}^1; \mathbb{R}^2))} T^{\alpha} \\
				&\quad \ + C(T_0,R_0,p,\gamma_0) \| \gamma - \tilde \gamma \|_{\mathbb{E}^{1,p}((0,T); L^p(\mathbb{S}^1; \mathbb{R}^2))} \| \bar \gamma \|_{C^{0,\alpha}([0,T]; C^{1,\alpha}(\mathbb{S}^1; \mathbb{R}^2))} T^{\alpha} \\
				&\quad \ + C(T_0,p,\bar \gamma) \| \gamma - \tilde \gamma \|_{\mathbb{E}^{1,p}((0,T); L^p(\mathbb{S}^1; \mathbb{R}^2))} \| \tilde \gamma - \bar \gamma \|_{\mathbb{E}^{1,p}((0,T); L^p(\mathbb{S}^1; \mathbb{R}^2))} T^{2\alpha} \\
		&\leq  C(T_0,R_0,p,\bar \gamma) \| \gamma - \tilde \gamma \|_{\mathbb{E}^{1,p}((0,T); L^p(\mathbb{S}^1; \mathbb{R}^2))} T^{\alpha} \\
				&\quad \ + C(T_0,R_0,p,\gamma_0) \| \gamma - \tilde \gamma \|_{\mathbb{E}^{1,p}((0,T); L^p(\mathbb{S}^1; \mathbb{R}^2))} \| \tilde \gamma - \bar \gamma \|_{\mathbb{E}^{1,p}((0,T); L^p(\mathbb{S}^1; \mathbb{R}^2))} T^{2\alpha} \\
		&\leq C(T_0,R_0,p,\bar \gamma) \| \gamma - \tilde \gamma \|_{\mathbb{E}^{1,p}((0,T); L^p(\mathbb{S}^1; \mathbb{R}^2))} T^{\alpha}
\end{align*}
	where we have used the fact that $\gamma(\cdot,0) = \tilde \gamma(\cdot,0) = \bar \gamma(\cdot,0) = \gamma_0(\cdot)$.
	As before we have that 
\[
	\| \tilde \gamma_{u^4} \|_{L^p((0,T); L^p(\mathbb{S}^1; \mathbb{R}^2))} \leq \| \tilde \gamma - \bar \gamma \|_{\mathbb{E}^{1,p}((0,T); L^p(\mathbb{S}^1; \mathbb{R}^2))} + \| \bar \gamma \|_{\mathbb{E}^{1,p}((0,T); L^p(\mathbb{S}^1; \mathbb{R}^2))} \leq R_0 + C(\bar \gamma) \leq C(R_0,\bar \gamma)
\]
	so that it follows for sufficiently small $T = T(T_0,R_0,p,\varepsilon, \bar \gamma) > 0$
\begin{align*}
	\left \| \frac{1}{|\gamma_{uu}^{\perp}|^2} - \frac{1}{|\tilde \gamma_{uu}^{\perp}|^2} \right \|_{L^{\infty}((0,T); L^p(\mathbb{S}^1))} \| \tilde \gamma_{u^4} \|_{L^p((0,T); L^p(\mathbb{S}^1; \mathbb{R}^2))} \leq \frac{\varepsilon}{2} \| \gamma - \tilde \gamma \|_{\mathbb{E}^{1,p}((0,T); L^p(\mathbb{S}^1; \mathbb{R}^2))}.
\end{align*}
\end{proof}
\end{prop}

	Having shown the $\varepsilon$-contractivity of the highest order term of the operator on the right hand side of \eqref{2} we now turn our attention to controlling the non-linearities present in $F$. The main obstacle to this is the presence of $|\gamma_u|^{-l}\langle \gamma_{uu}, \nu \rangle^{-1}$ for some $l \in \mathbb{N}$ in most of the terms. The following lemma shows that we are able to control such quantities.
	
\begin{lemma}\label{Lem5}
	Let $l \in \mathbb{N}, 1 \leq q \leq \infty$. Suppose that $\varphi(\gamma_1, \cdots, \gamma_m)$ is a multilinear map such that for every $\gamma_1,\cdots,\gamma_l \in \mathbb{E}^{1,p}((0,T); L^p(\mathbb{S}^1; \mathbb{R}^2))$,
\begin{align}
	\| \varphi(\gamma_1, \cdots, \gamma_l) \|_{L^q((0,T); Z)} \leq C \prod_{i = 1}^l \| \partial_u^{\beta_i} \gamma_i \|_{X_i} \label{3.5}
\end{align}
	for some $\beta_1,\cdots,\beta_l \in \{0,1,2,3\}$ and $Z, X_1,\cdots,X_m$ are Banach spaces satisfying the following criteria:
\begin{enumerate}[(i)]
	\item $\partial_u^{\beta_i} : \mathbb{E}^{1,p}((0,T); L^p(\mathbb{S}^1; \mathbb{R}^2)) \rightarrow X_i$ and for $\gamma \in \prescript{}{0}{\mathbb{E}^{1,p}}((0,T); L^p(\mathbb{S}^1; \mathbb{R}^2))$ there exists some constant $C = C(T_0) > 0$ such that
\[
	\| \partial_u^{\beta_i} \gamma \|_{X_i} \leq C(T_0,p) \| \gamma \|_{\mathbb{E}^{1,p}((0,T); L^p(\mathbb{S}^1; \mathbb{R}^2))}
\]
	for each $i = 1,\cdots,l$.
	\item One of the following holds:
		\begin{enumerate}
			\item There exists some $\alpha > 0$ and a constant $C = C(T_0,p) > 0$ such that for every $\gamma \in \prescript{}{0}{\mathbb{E}^{1,p}}((0,T); L^p(\mathbb{S}^1; \mathbb{R}^2))$
			\[
				\| \partial_u^{\beta_i} \gamma \|_{X_i} \leq C(T_0,p)T^{\alpha} \| \gamma \|_{\mathbb{E}^{1,p}((0,T); L^p(\mathbb{S}^1; \mathbb{R}^2))}.
			\]
			\item There is some $j \neq l$ such that for every $\gamma \in \prescript{}{0}{\mathbb{E}^{1,p}}((0,T); L^p(\mathbb{S}^1; \mathbb{R}^2))$
\[
	\lim_{T \searrow 0} \| \partial_u^{\beta_j}\gamma \|_{X_j} = 0.
\]
		\end{enumerate}
\end{enumerate}
	Then if $\varphi(\gamma) := \varphi(\gamma,\cdots,\gamma)$ we have that $\varphi(\gamma) \in L^q((0,T); Z)$ for all \newline $\gamma \in \mathbb{E}^{1,p}((0,T); L^p(\mathbb{S}^1; \mathbb{R}^2))$ and for any $\varepsilon \in (0,1)$ there exists $R = R(T_0,R_0,p,q,l,\varepsilon, \bar \gamma), T = T(T_0,R_0,p,q,l,\varepsilon, \bar \gamma) > 0$ small enough such that for every $\gamma, \tilde \gamma \in B_R^{\mathbb{E}^{1,p}}(\bar \gamma)$,
\[
	\| \varphi(\gamma) - \varphi(\tilde \gamma) \|_{L^q((0,T); Z)} \leq \varepsilon \| \gamma - \tilde \gamma \|_{\mathbb{E}^{1,p}((0,T); L^p(\mathbb{S}^1; \mathbb{R}^2))}.
\]
	Moreover for any $l \in \mathbb{N}$ and for $Z = L^r(\mathbb{S}^1; \mathbb{R}^2)$ for some $1 \leq r \leq \infty$ there exists $R = R(T_0,R_0,p,q,r,l,\varepsilon, \bar \gamma)$, $T = T(T_0,R_0,p,q,r,l, \varepsilon, \bar \gamma) > 0$ small enough such that
\[
	\| |\gamma_u|^{-l}\langle \gamma_{uu}, \nu \rangle^{-1} \varphi(\gamma) - |\tilde \gamma_u|^{-l}\langle \tilde \gamma_{uu}, \tilde \nu \rangle^{-1} \varphi(\tilde \gamma) \|_{L^q((0,T); L^r(\mathbb{S}^1; \mathbb{R}^2))} \leq \varepsilon \| \gamma - \tilde \gamma \|_{\mathbb{E}^{1,p}((0,T); L^p(\mathbb{S}^1; \mathbb{R}^2))}.
\]
\begin{proof}
	The first part of the proof follows as in \cite{RS}. Let $l \in \mathbb{N}$ be arbitrary. The simple estimate
\[
	|\gamma_u|^{-l} \langle \gamma_{uu}, \nu \rangle^{-1}|\varphi(\gamma)| \leq \left(\frac{1}{2} \inf_{\mathbb{S}^1} |(\gamma_0)_u|\right)^{-l} \left(\frac{1}{2}\inf_{\mathbb{S}^1} \langle (\gamma_0)_{uu}, \nu_0 \rangle \right)^{-1} |\varphi(\gamma)|
\]
	shows that $\gamma \mapsto |\gamma_u|^{-l}(\langle \gamma_{uu}, \nu \rangle)^{-1}\varphi(\gamma)$ is well defined in the sense that it maps into the same space as $\varphi(\gamma)$. Let $\varepsilon \in (0,1)$ and $\gamma, \tilde \gamma \in B_R^{\mathbb{E}^{1,p}}(\bar \gamma)$. It then holds 
\begin{align*}
	&\| |\gamma_u|^{-l}\langle \gamma_{uu}, \nu \rangle^{-1} \varphi(\gamma) - |\tilde \gamma_u|^{-l}\langle \tilde \gamma_{uu}, \tilde \nu \rangle^{-1} \varphi(\tilde \gamma) \|_{L^q((0,T); L^r(\mathbb{S}^1; \mathbb{R}^2))} \\
		&\leq \| (|\gamma_u|^{-l} - |\tilde \gamma_u|^{-l})\langle \gamma_{uu}, \nu \rangle^{-1} \varphi(\gamma) \|_{L^q((0,T); L^r(\mathbb{S}^1; \mathbb{R}^2))} \\
				&\quad \ + \||\tilde \gamma_u|^{-l}(\langle \gamma_{uu}, \nu \rangle^{-1}\varphi(\gamma) - \langle \tilde \gamma_{uu}, \tilde \nu \rangle^{-1} \varphi(\tilde \gamma)) \|_{L^q((0,T); L^r(\mathbb{S}^1; \mathbb{R}^2))} \\
		&\leq \| |\gamma_u|^{-l} - |\tilde \gamma_u|^{-l} \|_{L^{\infty}((0,T); L^r(\mathbb{S}^1))} \| \langle \gamma_{uu}, \nu \rangle^{-1} \|_{L^{\infty}((0,T); L^r(\mathbb{S}^1))} \| \varphi(\gamma) \|_{L^q((0,T); L^r(\mathbb{S}^1; \mathbb{R}^2))} \\
				&\quad \ + \| |\tilde \gamma_u|^{-l} \|_{L^{\infty}((0,T); L^r(\mathbb{S}^1))} \| \langle \gamma_{uu}, \nu \rangle^{-1} - \langle \tilde \gamma_{uu}, \tilde \nu \rangle^{-1} \|_{L^{\infty}((0,T); L^r(\mathbb{S}^1))} \| \varphi(\gamma) \|_{L^q((0,T); L^r(\mathbb{S}^1; \mathbb{R}^2))} \\
				&\quad \ + \| |\tilde \gamma_u|^{-l} \|_{L^{\infty}((0,T); L^r(\mathbb{S}^1))} \| \langle \tilde \gamma_{uu}, \tilde \nu \rangle^{-1} \|_{L^{\infty}((0,T); L^r(\mathbb{S}^1))} \| \varphi(\gamma) - \varphi(\tilde \gamma) \|_{L^q((0,T); L^r(\mathbb{S}^1; \mathbb{R}^2))}.
\end{align*}
	We see from the Mean Value Theorem that
\[
	|\gamma_u|^{-l} - |\tilde \gamma_u|^{-l} \leq l\left(\frac{1}{2}\inf_{\mathbb{S}^1} |(\gamma_0)_u|\right)^{-(l+1)}|\gamma_u - \tilde \gamma_u|
\]
	from which it follows there exists some $\alpha \in (0,1)$ such that
\begin{align*}
	\| |\gamma_u|^{-l} - |\tilde \gamma_u|^{-l} \|_{L^{\infty}((0,T); L^r(\mathbb{S}^1))} &\leq C(l,\gamma_0) \| \gamma_u - \tilde \gamma_u \|_{L^{\infty}((0,T); L^r(\mathbb{S}^1; \mathbb{R}^2))} \\
		&\leq C(r,l,\gamma_0) \| \gamma - \tilde \gamma \|_{C^{0,\alpha}([0,T]; C^{1,\alpha}(\mathbb{S}^1; \mathbb{R}^2))} \\
		&\leq C(T_0,p,r,l,\gamma_0) \| \gamma - \tilde \gamma \|_{\mathbb{E}^{1,p}((0,T); L^p(\mathbb{S}^1; \mathbb{R}^2))}T^{\alpha}.
\end{align*}
	Similarly
\begin{align*}
	\langle \gamma_{uu}, \nu \rangle^{-1} - \langle \tilde \gamma_{uu}, \tilde \nu \rangle^{-1} &\leq \left(\inf_{\mathbb{S}^1} \langle (\gamma_0)_{uu}, \nu_0 \rangle \right)^{-2}|\langle \gamma_{uu}, \nu \rangle - \langle \tilde \gamma_{uu}, \tilde \nu \rangle| \\
		&\leq C(\gamma_0)(|\langle \gamma_{uu}, \nu \rangle - \langle \tilde \gamma_{uu}, \nu \rangle| + |\langle \tilde \gamma_{uu}, \nu \rangle - \langle \tilde \gamma_{uu}, \tilde \nu \rangle|) \\
		&\leq C(\gamma_0)|\gamma_{uu} - \tilde \gamma_{uu}| + C(\gamma_0)|\tilde \gamma_{uu}||\nu - \tilde \nu| \\
		&= C(\gamma_0)|\gamma_{uu} - \tilde \gamma_{uu}| + C(\gamma_0)|\tilde \gamma_{uu}|| R_{\frac{\pi}{2}}(\tau - \tilde \tau)| \\
		&= C(\gamma_0)|\gamma_{uu} - \tilde \gamma_{uu}| + C(\gamma_0)|\tilde \gamma_{uu}| |\tau - \tilde \tau|
\end{align*}
	where we have exploited the fact that the rotation $R_{\frac{\pi}{2}}$ is a linear isometry of $\mathbb{R}^2$. Using \eqref{7} replacing $\tau_0$ with $\tilde \tau$ where appropriate we find using the Mean Value Theorem again,
\begin{align*}
	|\tau - \tilde \tau| &\leq \left(\frac{1}{2}\inf_{\mathbb{S}^1} |(\gamma_0)_u| \right)^{-1}|\gamma_u - \tilde \gamma_u| + \left|\frac{1}{|\gamma_u|} - \frac{1}{|\tilde \gamma_u|} \right| |\gamma_u| \\
		&\leq \left(\frac{1}{2}\inf_{\mathbb{S}^1} |(\gamma_0)_u| \right)^{-1}|\gamma_u - \tilde \gamma_u| + \left(\frac{1}{2} \inf_{\mathbb{S}^1} |(\gamma_0)_u| \right)^{-2}|\gamma_u - \tilde \gamma_u| |\gamma_u|.
\end{align*}
	Thus
\begin{align*}
	&\| \langle \gamma_{uu}, \nu \rangle^{-1} - \langle \tilde \gamma_{uu}, \tilde \nu \rangle^{-1} \|_{L^{\infty}((0,T); L^r(\mathbb{S}^1))}  \\
		&\leq C(\gamma_0) \| \gamma_{uu} - \tilde \gamma_{uu} \|_{L^{\infty}((0,T); L^r(\mathbb{S}^1; \mathbb{R}^2))} \\
				&\quad \ + C(\gamma_0) \| \tilde \gamma_{uu} \|_{L^{\infty}((0,T); L^r(\mathbb{S}^1; \mathbb{R}^2))} \| \gamma_u - \tilde \gamma_u \|_{L^{\infty}((0,T); L^r(\mathbb{S}^1; \mathbb{R}^2))} \\
				&\quad \ + C(\gamma_0) \| \tilde \gamma_{uu} \|_{L^{\infty}((0,T); L^r(\mathbb{S}^1; \mathbb{R}^2))} \| \gamma_u - \tilde \gamma_u \|_{L^{\infty}((0,T); L^r(\mathbb{S}^1; \mathbb{R}^2))} \| \gamma_u \|_{L^{\infty}((0,T); L^r(\mathbb{S}^1; \mathbb{R}^2))} \\
		&\leq C(T_0,p,r,\gamma_0) \| \gamma - \tilde \gamma \|_{\mathbb{E}^{1,p}((0,T); L^p(\mathbb{S}^1; \mathbb{R}^2))}T^{\alpha} \\
				&\quad \ + C(T_0,p,r,\gamma_0) \| \tilde \gamma - \bar \gamma \|_{\mathbb{E}^{1,p}((0,T); L^p(\mathbb{S}^1; \mathbb{R}^2))} \| \gamma - \tilde \gamma \|_{\mathbb{E}^{1,p}((0,T); L^p(\mathbb{S}^1; \mathbb{R}^2))}T^{2\alpha} \\
				&\quad \ + C(T_0,p,r,\gamma_0) \| \bar \gamma \|_{C^{0,\alpha}([0,T]; C^{2,\alpha}(\mathbb{S}^1; \mathbb{R}^2))} \| \gamma - \tilde \gamma \|_{\mathbb{E}^{1,p}((0,T); L^p(\mathbb{S}^1; \mathbb{R}^2))} T^{\alpha} \\
				&\quad \ + C(T_0,p,r,\gamma_0) \| \tilde \gamma - \bar \gamma \|_{\mathbb{E}^{1,p}((0,T); L^p(\mathbb{S}^1; \mathbb{R}^2))}^2 \| \gamma - \tilde \gamma \|_{\mathbb{E}^{1,p}((0,T); L^p(\mathbb{S}^1; \mathbb{R}^2))} T^{3\alpha} \\
				&\quad \ + C(T_0,p,r,\gamma_0) \| \tilde \gamma - \bar \gamma \|_{\mathbb{E}^{1,p}((0,T); L^p(\mathbb{S}^1; \mathbb{R}^2))} \| \bar \gamma \|_{C^{0,\alpha}([0,T]; C^{2,\alpha}(\mathbb{S}^1; \mathbb{R}^2))} \| \gamma - \tilde \gamma \|_{\mathbb{E}^{1,p}((0,T); L^p(\mathbb{S}^1; \mathbb{R}^2))} T^{2\alpha} \\
				&\quad \ + C(T_0,p,r,\gamma_0) \| \tilde \gamma - \bar \gamma \|_{\mathbb{E}^{1,p}((0,T); L^p(\mathbb{S}^1; \mathbb{R}^2))} \| \bar \gamma \|_{C^{0,\alpha}([0,T]; C^{1,\alpha}(\mathbb{S}^1; \mathbb{R}^2))} \| \gamma - \tilde \gamma \|_{\mathbb{E}^{1,p}((0,T); L^p(\mathbb{S}^1; \mathbb{R}^2))} T^{2\alpha} \\
		&\leq C(T_0,R_0,p,r,\bar \gamma) T^{\alpha} \| \gamma - \tilde \gamma \|_{\mathbb{E}^{1,p}((0,T); L^p(\mathbb{S}^1; \mathbb{R}^2))}.
\end{align*}
	Similarly by the Mean Value Theorem and similar estimates as above
\begin{align*}
	\| \langle \gamma_{uu}, \nu \rangle^{-1} \|_{L^{\infty}((0,T); L^r(\mathbb{S}^1))} &\leq C(\gamma_0) \| \langle \gamma_{uu}, \nu \rangle \|_{L^{\infty}((0,T); L^r(\mathbb{S}^1))} \\
		&\leq C(\gamma_0) \| \gamma_{uu} \|_{L^{\infty}((0,T); L^r(\mathbb{S}^1; \mathbb{R}^2))} \\
		&\leq C(r,\gamma_0) \| \gamma - \bar \gamma \|_{C^{0,\alpha}([0,T]; C^{2,\alpha}(\mathbb{S}^1; \mathbb{R}^2))} \\
				&\quad \ + C(r,\gamma_0) \| \bar \gamma \|_{C^{0,\alpha}([0,T]; C^{2,\alpha}(\mathbb{S}^1; \mathbb{R}^2))} \\
		&\leq C(T_0,p,r,\gamma_0) \| \gamma - \bar \gamma \|_{\mathbb{E}^{1,p}((0,T); L^p(\mathbb{S}^1; \mathbb{R}^2))} T^{\alpha} + C(r,\bar \gamma) \\
		&\leq C(T_0,R_0,p,r,\bar \gamma)
\end{align*}
	with the same estimate replacing $\langle \gamma_{uu}, \nu \rangle^{-1}$ with $\langle \tilde \gamma_{uu}, \tilde \nu \rangle^{-1}$ and,
\begin{align*}
	\| |\tilde \gamma_u|^{-l} \|_{L^{\infty}((0,T); L^r(\mathbb{S}^1))} &\leq C(l, \gamma_0) \| \tilde \gamma_u \|_{L^{\infty}((0,T); L^r(\mathbb{S}^1; \mathbb{R}^2))} \\
		&\leq C(l,r, \gamma_0) \| \tilde \gamma - \bar \gamma \|_{C^{0,\alpha}([0,T]; C^{1,\alpha}(\mathbb{S}^1; \mathbb{R}^2))} \\
				&\quad \ + C(l,\gamma_0) \| \bar \gamma \|_{L^{\infty}((0,T); L^r(\mathbb{S}^1; \mathbb{R}^2))} \\
		&\leq C(T_0,p,r,l,\gamma_0) \| \tilde \gamma - \bar \gamma \|_{\mathbb{E}^{1,p}((0,T); L^p(\mathbb{S}^1; \mathbb{R}^2))} T^{\alpha} + C(l,\bar \gamma) \\
		&\leq C(T_0,R_0,p,r,l,\bar \gamma).
\end{align*}
	From the simple estimate
\begin{align*}
	\| \varphi(\gamma) \|_{L^q((0,T); L^r(\mathbb{S}^1; \mathbb{R}^2))} &\leq \| \varphi(\gamma) - \varphi(\bar \gamma) \|_{L^q((0,T); L^r(\mathbb{S}^1; \mathbb{R}^2))} + \| \varphi(\bar \gamma) \|_{L^q((0,T); L^r(\mathbb{S}^1; \mathbb{R}^2))} \\
		&\leq \bar \varepsilon \| \gamma - \bar \gamma \|_{\mathbb{E}^{1,p}((0,T); L^p(\mathbb{S}^1; \mathbb{R}^2))} + C(\bar \gamma) \\
		&\leq R_0 \bar \varepsilon + C(\bar \gamma)
\end{align*}
	for some $\bar \varepsilon > 0$ to be chosen, we thus find that
\begin{align*}
	\| |\gamma_u|^{-l}\langle \gamma_{uu}, \nu \rangle^{-1} \varphi(\gamma) &- |\tilde \gamma_u|^{-l}\langle \tilde \gamma_{uu}, \tilde \nu \rangle^{-1} \varphi(\tilde \gamma) \|_{L^q((0,T); L^r(\mathbb{S}^1; \mathbb{R}^2))} \\
	&\leq C(T_0,R_0,p,q,r,l,\bar \gamma)\bar \varepsilon T^{\alpha} \| \gamma - \tilde \gamma \|_{\mathbb{E}^{1,p}((0,T); L^p(\mathbb{S}^1; \mathbb{R}^2))} \\
			&\quad \ + C(T_0,R_0,p,q,r,l,\bar \gamma) T^{\alpha} \| \gamma - \tilde \gamma \|_{\mathbb{E}^{1,p}((0,T); L^p(\mathbb{S}^1; \mathbb{R}^2))} \\
			&\quad \ + C(T_0,R_0,p,q,r,l,\bar \gamma) \tilde \varepsilon T^{\alpha} \| \gamma - \tilde \gamma \|_{\mathbb{E}^{1,p}((0,T); L^p(\mathbb{S}^1; \mathbb{R}^2))} \\
	&\leq \varepsilon \| \gamma - \tilde \gamma \|_{\mathbb{E}^{1,p}((0,T); L^p(\mathbb{S}^1; \mathbb{R}^2))}
\end{align*}
for sufficiently small $\bar \varepsilon = \bar
\varepsilon(T_0,R_0,p,q,r,l,\varepsilon, \bar \gamma)$, $\tilde \varepsilon =
\tilde \varepsilon(T_0,R_0,p,q,r,l,\varepsilon, \bar \gamma) \in (0,1)$,
adjusting $T = T(T_0,R_0,p,q,r,l,\varepsilon, \bar \gamma)$, $R =
R(T_0,R_0,p,q,r,l,\varepsilon, \bar \gamma)$ if necessary.
\end{proof}
\end{lemma}
	This leads into contraction estimates on the operator $F$.

\begin{prop}\label{Prop2}
	Let $\varepsilon \in (0,1)$, $p > \frac{5}{2}$. For sufficiently small $T, R > 0$ the maps
	\begin{enumerate}[(i)]
		\item $\overline{B}_R^{\mathbb{E}^{1,p}}(\bar \gamma) \rightarrow L^p((0,T); L^p(\mathbb{S}^1; \mathbb{R}^2)) , \quad \gamma \mapsto - \frac{\langle \gamma_{u^3}, \nu\rangle}{|\gamma_u|\langle \gamma_{uu}, \nu \rangle} \gamma_{u^3}^{\perp}$
		\item $\overline{B}_R^{\mathbb{E}^{1,p}}(\bar \gamma) \rightarrow L^p((0,T); L^p(\mathbb{S}^1; \mathbb{R}^2)) , \quad \gamma \mapsto - \frac{1}{|\gamma_u|^2} \gamma_{uu}^{\perp}$
		\item $\overline{B}_R^{\mathbb{E}^{1,p}}(\bar \gamma) \rightarrow L^p((0,T); L^p(\mathbb{S}^1; \mathbb{R}^2)) , \quad \gamma \mapsto - \frac{4\langle \gamma_{u^3}, \gamma_u \rangle}{|\gamma_u|^2 \langle \gamma_{uu}, \nu \rangle}\nu$
		\item $\overline{B}_R^{\mathbb{E}^{1,p}}(\bar \gamma) \rightarrow L^p((0,T); L^p(\mathbb{S}^1; \mathbb{R}^2)) , \quad \gamma \mapsto \frac{6 \langle \gamma_{uu}, \gamma_u \rangle^2}{|\gamma_u|^4\langle \gamma_{uu}, \nu \rangle}\nu$
	\end{enumerate}
	are well defined $\varepsilon$-contractions.
\begin{proof}
	By \ref{Lem5} it is sufficient to show that the maps
\begin{enumerate}[(i)]
	\item $\overline{B}_R^{\mathbb{E}^{1,p}}(\bar \gamma) \rightarrow L^p((0,T); L^p(\mathbb{S}^1; \mathbb{R}^2)), \quad \gamma \mapsto \langle \gamma_{u^3}, \nu \rangle \gamma_{u^3}^{\perp}$
	\item $\overline{B}_R^{\mathbb{E}^{1,p}}(\bar \gamma) \rightarrow L^p((0,T); L^p(\mathbb{S}^1; \mathbb{R}^2)), \quad \gamma \mapsto \gamma_{uu}^{\perp}$
	\item $\overline{B}_R^{\mathbb{E}^{1,p}}(\bar \gamma) \rightarrow L^p((0,T); L^p(\mathbb{S}^1; \mathbb{R}^2)), \quad \gamma \mapsto \langle \gamma_{u^3}, \gamma_u \rangle \nu$
	\item $\overline{B}_R^{\mathbb{E}^{1,p}}(\bar \gamma) \rightarrow L^p((0,T); L^p(\mathbb{S}^1; \mathbb{R}^2)), \quad \gamma \mapsto \langle \gamma_{uu}, \gamma_u \rangle^2 \nu$
\end{enumerate}
	satisfy the assumption of \ref{Lem5} for which we shall be using \ref{Lem1} and H\dd{o}lder's Inequality in space and time to verify this. \newline
	(i) Observe that using H\dd{o}lder's Inequality in time along with the fact that $\nu$ is of unit length
\begin{align*}
	\| \langle \gamma_{u^3}, \nu \rangle \gamma_{u^3}^{\perp} \|_{L^p((0,T); L^p(\mathbb{S}^1; \mathbb{R}^2))} \leq \| \gamma_{u^3} \|_{L^{2p}((0,T); L^p(\mathbb{S}^1; \mathbb{R}^2))}^2
\end{align*}
	so that we are looking for $\theta_1 \in (0,1), \rho_1 \in [1, \infty)$ such that
\[
	\frac{\theta_1}{4} - \frac{1}{p} \geq - \frac{1}{2p}, \quad 1 - \theta_1 - \frac{1}{p} \geq - \frac{1}{\rho_1}.
\]
	from which we find that $\theta_1 \in [\frac{2}{p}, 1)$ which is well defined for $p > \frac{5}{2}$. Thus for a choice of $\theta_1 = \frac{2}{p}$ we find that
\[
	1 - \frac{2}{p} - \frac{1}{p} = 1 - \frac{3}{p} \geq - \frac{1}{\rho_1}.
\]
	For $p \geq 3$, $1 - \frac{3}{p} \geq 0$ so any choice of $\rho_1 > p$ say $\rho_1 = 2p$ will satisfy this inequality. Else for $\frac{5}{2} < p < 3$ we instead choose $\rho_1 = \frac{p}{3 - p}$ which is larger than $p$ for every $p > 2$ which of course holds by assumption. Thus applying H\dd{o}lder's Inequality in the space variable we find that for $p \geq 3$
\begin{align*}
	\| \gamma_{u^3} \|_{L^{2p}((0,T); L^p(\mathbb{S}^1; \mathbb{R}^2))}^2 &\leq C(p) \| \gamma_{u^3} \|_{L^{2p}((0,T); L^{2p}(\mathbb{S}^1; \mathbb{R}^2))}^2 \\
		&\leq C(T_0,p) \| \gamma \|_{\mathbb{E}^{1,p}((0,T); L^p(\mathbb{S}^1; \mathbb{R}^2))}^2
\end{align*}
	else for $\frac{5}{2} < p < 3$
\begin{align*}
	\| \gamma_{u^3} \|_{L^{2p}((0,T); L^p(\mathbb{S}^1; \mathbb{R}^2))}^2 &\leq C(p) \| \gamma_{u^3} \|_{L^{2p}((0,T); L^{\frac{p}{3 - p}}(\mathbb{S}^1; \mathbb{R}^2))}^2 \\
		&\leq C(T_0,p) \| \gamma \|_{\mathbb{E}^{1,p}((0,T); L^p(\mathbb{S}^1; \mathbb{R}^2))}^2
\end{align*}
	(ii) The estimate from applying H\dd{o}lder's Inequality in the time variable
\[
	\| \gamma_{uu}^{\perp} \|_{L^p((0,T); L^p(\mathbb{S}^1; \mathbb{R}^2))} \leq \| \gamma_{uu} \|_{L^p((0,T); L^p(\mathbb{S}^1; \mathbb{R}^2))} \leq C(T_0,p) \| \gamma_{uu} \|_{L^{2p}((0,T); L^p(\mathbb{S}^1; \mathbb{R}^2))}
\]
	shows that we must find $\theta_2 \in (0,1), \rho_2 \in [1, \infty)$ such that
\[
	\frac{\theta_2}{4} - \frac{1}{p} \geq - \frac{1}{2p}, \quad 2(1 - \theta_2) - \frac{1}{p} \geq - \frac{1}{\rho_2}.
\]
	For a choice of $\theta_2 = \frac{2}{p}$ we need $\rho_2$ such that
\[
	2 - \frac{4}{p} - \frac{1}{p} = 2 - \frac{5}{p} = \frac{2p - 5}{p} \geq - \frac{1}{\rho_2}
\]
	but since we have assumed $p > \frac{5}{2}, 2p - 5 > 0$ and thus any $\rho_2 > p$ can be chosen for which we again shall choose $\rho_2 = 2p$. Thus it holds
\[
	\| \gamma_{uu}^{\perp} \|_{L^p((0,T); L^p(\mathbb{S}^1; \mathbb{R}^2))} \leq C(T_0,p) \| \gamma_{uu} \|_{L^{2p}((0,T); L^{2p}(\mathbb{S}^1; \mathbb{R}^2))} \leq C(T_0,p) \| \gamma \|_{\mathbb{E}^{1,p}((0,T); L^p(\mathbb{S}^1; \mathbb{R}^2))}.
\]
	(iii) Observe that by using H\dd{o}lder's Inequality in time
\begin{align*}
	\| \langle \gamma_{u^3}, \gamma_u \rangle \nu \|_{L^p((0,T); L^p(\mathbb{S}^1; \mathbb{R}^2))} &\leq \| \gamma_u \|_{L^{\infty}((0,T); L^p(\mathbb{S}^1; \mathbb{R}^2))} \| \gamma_{u^3} \|_{L^p((0,T); L^p(\mathbb{S}^1; \mathbb{R}^2))} \\
		&\leq C(T_0,p) \| \gamma_u \|_{L^{\infty}((0,T); L^p(\mathbb{S}^1; \mathbb{R}^2))} \| \gamma_{u^3} \|_{L^{2p}((0,T); L^p(\mathbb{S}^1; \mathbb{R}^2))}
\end{align*}
	with the estimate
\[
	\| \gamma_u \|_{L^{\infty}((0,T); L^p(\mathbb{S}^1; \mathbb{R}^2))} \leq  C(T_0,p) \| \gamma \|_{C^{0,\alpha}([0,T]; C^{1,\alpha}(\mathbb{S}^1; \mathbb{R}^2))} \leq C(T_0,p)T^{\alpha} \| \gamma \|_{\mathbb{E}^{1,p}((0,T); L^p(\mathbb{S}^1; \mathbb{R}^2))}
\]
	for some $\alpha \in (0,1)$ it is then immediate from (i) that this term satisfies the hypotheses of \ref{Lem5}. \newline
	(iv) Applying H\dd{o}lder's Inequality in the time variable again we find for the final term
\begin{align*}
	\| \langle \gamma_{uu}, \gamma_u \rangle^2 \nu \|_{L^p((0,T); L^p(\mathbb{S}^1; \mathbb{R}^2))} \leq \| \gamma_u \|_{L^{\infty}((0,T); L^p(\mathbb{S}^1; \mathbb{R}^2))}^2 \| \gamma_{uu} \|_{L^{2p}((0,T); L^p(\mathbb{S}^1; \mathbb{R}^2))}^2
\end{align*}
	from which applying H\dd{o}lder's Inequality in the space variable, (ii) and the same $L^{\infty}$ estimate on $\gamma_u$ as above we find that
\[
	\| \langle \gamma_{uu}, \gamma_u \rangle^2 \nu \|_{L^p((0,T); L^p(\mathbb{S}^1; \mathbb{R}^2))} \leq C(T_0,p)T^{2\alpha} \| \gamma \|_{\mathbb{E}^{1,p}((0,T); L^p(\mathbb{S}^1; \mathbb{R}^2))}^4.
\]
\end{proof}
\end{prop}

	We are now ready to establish the local well-posedness to \eqref{2} by use of a fixed point argument. We must first define an appropriate map between a linear problem and our abstract quasilinear problem which is the purpose of the following proposition.

\begin{prop}\label{Prop3}
	Let $\sigma \in \mathbb{E}^{1,p}((0,T); L^p(\mathbb{S}^1; \mathbb{R}^2))$ be the unique solution to
\begin{align*}
	\begin{cases}
		\partial_t \sigma + \frac{\sigma_{u^4}}{|(\gamma_0)_u|^4} &= \bold F(\gamma) := \left(\frac{1}{|(\gamma_0)_u|^4} - \frac{1}{|\gamma_u|^4} \right) \gamma_{u^4} + F(\gamma_{u^3}, \gamma_{uu}, \gamma_u), \quad t > 0 \\
		\qquad \ \ \sigma(\cdot,0) &= \gamma_0(\cdot)
	\end{cases}
\end{align*}
	which is guaranteed by \ref{Thm1} and suppose we define $\Phi : \overline{B}_R^{\mathbb{E}^{1,p}}(\bar \gamma) \rightarrow \mathbb{E}^{1,p}((0,T); L^p(\mathbb{S}^1; \mathbb{R}^2))$ by $\Phi(\gamma) = \sigma$. Then there exists $T = T(T_0,R_0,p,\varepsilon, \bar \gamma), R = R(T_0,R_0,p,\varepsilon, \bar \gamma) > 0$ such that $\Phi : \overline{B}_R^{\mathbb{E}^{1,p}}(\bar \gamma) \rightarrow \overline{B}_R^{\mathbb{E}^{1,p}}(\bar \gamma)$ is a well defined $\varepsilon$-contraction.
\begin{proof}
	Let $\varepsilon \in (0,1)$ and $\gamma, \tilde \gamma \in \overline{B}_R^{\mathbb{E}^{1,p}}(\bar \gamma)$ with $\sigma = \Phi(\gamma), \tilde \sigma = \Phi(\tilde \gamma)$. As $\bold F$ is written as a linear combination of $\varepsilon$-contractions it then holds that $\bold F$ is also a well-defined $\varepsilon$-contraction by reducing $T = T(T_0,R_0,p,\varepsilon, \bar \gamma), R = R(T_0,R_0,p,\varepsilon, \bar \gamma) > 0$ as needed. By \eqref{5} there exists some $C = C(T_0,p,\gamma_0) > 0$ such that
\begin{align}
	\| \sigma - \tilde \sigma \|_{\mathbb{E}^{1,p}((0,T); L^p(\mathbb{S}^1; \mathbb{R}^2))} \leq C \| \bold F(\gamma) - \bold F(\tilde \gamma) \|_{L^p((0,T); L^p(\mathbb{S}^1; \mathbb{R}^2))} \leq C \frac{\varepsilon}{2C} \| \gamma - \tilde \gamma \|_{\mathbb{E}^{1,p}((0,T); L^p(\mathbb{S}^1; \mathbb{R}^2))} \label{11}
\end{align}
	reducing $T = T(T_0,R_0,p,\varepsilon, \bar \gamma), R = R(T_0,R_0, p, \varepsilon, \bar \gamma)$ as necessary. To see that $\Phi$ is well defined we apply \eqref{11}
\begin{align*}
	\| \sigma - \bar \gamma \|_{\mathbb{E}^{1,p}((0,T); L^p(\mathbb{S}^1; \mathbb{R}^2))} &= \| \Phi(\gamma) - \Phi(0) \|_{\mathbb{E}^{1,p}((0,T); L^p(\mathbb{S}^1; \mathbb{R}^2))} \\
		&\leq C \| \bold F(\gamma) - 0 \|_{L^p((0,T); L^p(\mathbb{S}^1; \mathbb{R}^2))} \\
		&\leq C (\| \bold F(\gamma) - \bold F(\bar \gamma) \|_{L^p((0,T); L^p(\mathbb{S}^1; \mathbb{R}^2))} + \| \bold F(\bar \gamma) \|_{L^p((0,T); L^p(\mathbb{S}^1; \mathbb{R}^2))}) \\
		&\leq \frac{\varepsilon}{2} \| \gamma - \bar \gamma \|_{\mathbb{E}^{1,p}((0,T); L^p(\mathbb{S}^1; \mathbb{R}^2))} + C \| \bold F(\bar \gamma) \|_{L^p((0,T); L^p(\mathbb{S}^1; \mathbb{R}^2))} \\
		&\leq \frac{\varepsilon}{2}R + C \| \bold F(\bar \gamma) \|_{L^p((0,T); L^p(\mathbb{S}^1; \mathbb{R}^2))}.
\end{align*}
	We can repeat similar estimates as in \ref{Prop1}, \ref{Prop2} to find that there exists some $C = C(T_0,p,\bar \gamma) > 0$ such that
\[
	\| \bold F(\bar \gamma) \|_{L^p((0,T); L^p(\mathbb{S}^1; \mathbb{R}^2))} \leq C(T_0,p,\bar \gamma)T^{\alpha} \leq \frac{\varepsilon}{2C}R
\]
	by restricting $T = T(T_0,R_0,p,\varepsilon, \bar \gamma) > 0$ as necessary. Thus
\[
	\| \Phi(\gamma) - \bar \gamma \|_{\mathbb{E}^{1,p}((0,T); L^p(\mathbb{S}^1; \mathbb{R}^2))} \leq R
\]
	for $R, T > 0$ small enough.
\end{proof}
\end{prop}

	This naturally leads into the main existence theorem.

\begin{theorem}
\label{TMlod}
Let $\gamma_0 \in W_{\Imm}^{4(1 - \frac{1}{p}),p}(\mathbb{S}^1; \mathbb{R}^2)$ be a convex curve, for some $p > \frac{5}{2}$. Then there exists $T, R > 0$ such that \eqref{2} has a unique solution $\gamma \in \mathbb{E}^{1,p}((0,T); L^p(\mathbb{S}^1; \mathbb{R}^2))$.
\end{theorem}
\begin{proof}
	Choosing $T, R > 0$ as in \ref{Prop3} with $\varepsilon = \frac{1}{2}$ the map $\Phi : \overline{B}_R^{\mathbb{E}^{1,p}}(\bar \gamma) \rightarrow \overline{B}_R^{\mathbb{E}^{1,p}}(\bar \gamma)$ defines a contraction on the complete metric space $\overline{B}_R^{\mathbb{E}^{1,p}}(\bar \gamma)$ and hence by the Banach fixed point Theorem $\Phi$ has a unique fixed point $\gamma \in \overline{B}_R^{\mathbb{E}^{1,p}}(\bar \gamma)$. Since any fixed point of $\Phi$ is a solution to \eqref{2} in $\overline{B}_R^{\mathbb{E}^{1,p}}(\bar \gamma)$ and conversely a solution to \eqref{2} in $\overline{B}_R^{\mathbb{E}^{1,p}}(\bar \gamma)$ is a fixed point of $\Phi$ this establishes local well-posedness for \eqref{2} in $\overline{B}_R^{\mathbb{E}^{1,p}}(\bar \gamma)$. We note that any restriction of $\gamma \in \overline{B}_R^{\mathbb{E}^{1,p}}(\bar \gamma)$ to a smaller time interval $[0,\tilde T]$ is also a unique solution to \eqref{2} in $\overline{B}_R^{\mathbb{E}^{1,p}}(\bar \gamma)$. \newline \newline
	Suppose that $T_1, T_2 > 0$ and $\gamma_i \in \mathbb{E}^{1,p}((0,T_i); L^p(\mathbb{S}^1; \mathbb{R}^2)), i = 1,2$ are two families of immersions satisfying \eqref{2} with $\gamma_0 \in W_{\Imm}^{4(1 - \frac{1}{p}),p}(\mathbb{S}^1; \mathbb{R}^2)$ which is convex. Without loss of generality we may assume that $T_1 \leq T_2$. We claim that $\gamma_2 \vert_{[0,T_1]} = \gamma_1$. To verify this claim suppose we define $\hat t = \sup \{t \in [0, T_1) : \gamma_1(\cdot,\tilde t) = \gamma_2(\cdot,\tilde t), \forall\  0 \leq \tilde t < t\}$ to be the maximal time where $\gamma_1, \gamma_2$ agree. $\hat t$ is well defined by use of the embedding
\[
	\mathbb{E}^{1,p}((0,T); L^p(\mathbb{S}^1; \mathbb{R}^2)) \hookrightarrow BUC([0,T]; W^{4(1 - \frac{1}{p}),p}(\mathbb{S}^1; \mathbb{R}^2))
\]
	which holds by \cite{MS} Theorem 4.2. Moreover since $\lim_{t \searrow 0} |\gamma_i - \bar \gamma| = 0$ for $i = 1,2$ it follows by the dominated convergence theorem that
\[
	\lim_{T \searrow 0} \| \gamma_i - \bar \gamma \|_{\mathbb{E}^{1,p}((0,T); L^p(\mathbb{S}^1; \mathbb{R}^2))} = 0
\]
	and hence by restricting $T,R > 0$ as necessary we can guarantee that $\gamma_i \in \overline{B}_R^{\mathbb{E}^{1,p}}(\bar \gamma)$ and hence must agree for at least as long as both $\gamma_i$ remain in $\overline{B}_R^{\mathbb{E}^{1,p}}(\bar \gamma)$. We claim that $\hat t = T_1$, we have already shown above that $\hat t \geq T > 0$. Suppose for the sake of contradiction that $\hat t < T_1$, since $\gamma_i \in \mathbb{E}^{1,p}((0,T_1); L^p(\mathbb{S}^1; \mathbb{R}^2)) \hookrightarrow BUC([0,T_1]; W^{4(1 - \frac{1}{p}),p}(\mathbb{S}^1; \mathbb{R}^2))$ and remain immersed for all times along with \ref{Lem4} shows that while they might not necessarily remain convex it still remains that $|\gamma_{uu}^{\perp}| > 0$ for all times which is the only necessary assumption needed to control the inverse curvature terms present in the operator. Thus the curve $\gamma_0(\cdot) := \gamma_1(\cdot,\hat t) \in W^{4(1 - \frac{1}{p}),p}_{\Imm}(\mathbb{S}^1; \mathbb{R}^2)$ is well defined and thus there exists some $T,R > 0$ such that \eqref{2} with initial data $\gamma_0$ has a unique solution $\gamma \in \overline{B}_R^{\mathbb{E}^{1,p}}(\bar \gamma)$. Observing that $\gamma_i(\cdot, \hat t + t) \vert_{0 \leq t \leq T_1 - \hat t}$ are both solutions to \eqref{2} with initial data $\gamma_0$ we can follow a similar argument as above to show that $\gamma_1(\cdot, \hat t + \cdot) = \gamma_2(\cdot, \hat t + \cdot)$ in $[0, T)$ contradicting the definition of $\hat t$.
\end{proof}



\bibliographystyle{plain}
\bibliography{ef}

\end{document}